\documentclass[arxiv]{imsart}
\RequirePackage{amsthm,amsmath,amsfonts,amssymb}
\RequirePackage[numbers]{natbib}
\usepackage{fix-cm}
\usepackage{docmute}
\usepackage{amsmath,amssymb} 

\usepackage{mathrsfs}

\usepackage{amsbsy}  
\usepackage{mathtools}  

\usepackage{isomath}  

\usepackage{xfrac}  
\usepackage{xspace}  
\usepackage{float}  

\usepackage{graphicx}  

\usepackage[colorlinks,linkcolor=blue,citecolor=blue,final]{hyperref}

\usepackage[capitalise]{cleveref}  
\usepackage{appendix}  
\usepackage{enumitem}  

\usepackage{multicol}  
\usepackage[noend]{algpseudocode}  
\usepackage{algorithm}  

\usepackage{acro}  
\usepackage{xr}  
\usepackage{bm} 
\usepackage{tabularx}  
\usepackage{etoolbox}  

\usepackage{booktabs}
\usepackage{multirow}
\usepackage{chngcntr}

\usepackage{apptools}
\usepackage{xcolor}
\usepackage{comment}
\usepackage{makecell} 
\usepackage{caption}

\AtAppendix{\counterwithin{lemma}{section}}

\startlocaldefs

\newtheorem{theorem}{Theorem}
\newtheorem{lemma}{Lemma}
\newtheorem{proposition}{Proposition}

\newtheorem{assump}{Assumption} 
\newtheorem{remark}{Remark}

\newcommand{\trueparam}{\theta^\dagger}
\newcommand{\truebeta}{\beta^\dagger}

\newcommand*{\sample}[2]{X_{t_{#1}}^{[#2]}}

\newcommand{\probconv}{\xrightarrow{\mathbb{P}_{\trueparam}}}
\newcommand{\distconv}{\xrightarrow{\mathcal{L}_{\trueparam}}}

\everymath {\displaystyle}

\newcommand{\ruby}[2]{
\leavevmode
\setbox0=\hbox{#1}
\setbox1=\hbox{\tiny #2}
\ifdim\wd0>\wd1 \dimen0=\wd0 \else \dimen0=\wd1 \fi
\hbox{
\kanjiskip=0pt plus 2fil
\xkanjiskip=0pt plus 2fil
\vbox{
\hbox to \dimen0{
\small \hfil#2\hfil}
\nointerlineskip
\hbox to \dimen0{\mathstrut\hfil#1\hfil}}}}
\allowdisplaybreaks[1]

\renewcommand{\refname }{References}

\makeatletter
\newcommand*{\addFileDependency}[1]{
\typeout{(#1)}
\@addtofilelist{#1}
\IfFileExists{#1}{}{\typeout{No file #1.}}
}\makeatother

\usepackage{multirow}
\usepackage{docmute}
\usepackage{xr} 
\endlocaldefs

\begin{document}

\begin{frontmatter}
\title{Parameter estimation for weakly interacting hypoelliptic diffusions}
\runtitle{Parameter estimation for weakly interacting diffusions}

\begin{aug}
\author[A]{\fnms{Yuga}~\snm{Iguchi}\ead[label=e1]{y.iguchi@lancaster.ac.uk}},
\author[B]{\fnms{Alexandros}~\snm{Beskos}\ead[label=e2]{a.beskos@ucl.ac.uk}}
\and
\author[C]{\fnms{Grigorios A.}~\snm{Pavliotis}\ead[label=e3]{g.pavliotis@imperial.ac.uk}}
\address[A]{School of Mathematical Sciences, Lancaster University, UK\printead[presep={,\ }]{e1}}

\address[B]{Department of Statistical Science, University College London, UK\printead[presep={,\ }]{e2}}

\address[C]{Department of Mathematics, Imperial College London, UK\printead[presep={,\ }]{e3}}

\end{aug}

\begin{abstract}
We study parameter estimation for interacting particle systems (IPSs) consisting of $N$ weakly interacting multivariate hypoelliptic SDEs. We propose a locally Gaussian approximation of the transition dynamics, carefully designed to address the degenerate structure of the noise (diffusion matrix), thus leading to the formation of a well-defined full likelihood. Our approach permits carrying out statistical inference for a wide class of hypoelliptic IPSs that are not covered by recent works as the latter rely on the Euler-Maruyama scheme. We analyze a contrast estimator based on the developed likelihood with $n$ high-frequency particle observations over a fixed period $[0,T]$ and show its asymptotic normality as $n, N \to \infty$ with a requirement that the step-size $\Delta_n = T/n$ is such that $N\Delta_n\rightarrow 0$, assuming that all particle coordinates (e.g.~position and velocity) are observed. In practical situations where only partial observations (e.g. particle positions but not velocities) are available, the proposed locally Gaussian approximation offers greater flexibility for inference, when combined with established Bayesian techniques. In particular, unlike the Euler-Maruyama-based approaches, we do not have to impose restrictive structures on the hypoelliptic IPSs. We present numerical experiments that illustrate the effectiveness of our approach, both with complete and partial particle observations.
\end{abstract}



\end{frontmatter}


\section{Introduction} \label{sec:intro}
Interacting particle systems (IPSs) and their mean field limit have been an active area of research for several decades, see~\cite{sznitman1991} for a classic reference, and the recent review articles~\cite{Diez_2022a, Diez_2022b} and the references therein. In addition to the well-known applications of such systems in physics~\cite{frank04}, synchronization~\cite{RevModPhys.77.137} and mathematical biology~\cite{Stevens_2000}, new and exciting applications of 
IPSs have emerged in recent years. As examples, we mention social dynamics and collective behaviour~\cite{NPT10, toscani2014}, pedestrian dynamics~\cite{GSW19}, large financial systems~\cite{giesecke2019}, active matter~\cite{peruani2008mean, Degond_al_2017}, urban dynamics~\cite{gaskin2022neural}, the dynamics of power grid networks~\cite{gaskin2024inferring} and machine learning \cite{SiS20, rigo_2025}.

In recent years, there has been a great amount of research activity on the problem of statistical inference for IPSs and their mean-field limit, starting from the classical works~\cite{Kasonga_1990, Bishwal_2011}. An incomplete list of recent work on inference for large IPSs and their mean-field limit (both the McKean SDE and the McKean-Vlasov PDE) includes kernel methods~\cite{LMT2021, LaL22}, contrast functions based on a pseudo-likelihood \cite{amo:23}, maximum likelihood estimation \cite{DeH23}, stochastic gradient descent \cite{PaReZa2025, sha:23}, approximate likelihood based on an empirical approximation of the invariant measure \cite{GeL24}, method of moments \cite{PaZ24} and eigenfunction martingale estimating functions~\cite{PaZ22,PaZ25}.  

In applications there are many examples of IPSs where each particle is modeled via a hypoelliptic SDE that has a degenerate noise (diffusion matrix) and thus involves rough and smooth components, with the latter not directly driven by the Brownian motion: E.g., the mean-field underdamped Langevin equations/generalized Langevin equations and the mean-field FitzHugh-Nagumo model.  In addition, most models for flocking and swarming, e.g.~the Cucker-Smale model and its variants~\cite{cattiaux_2018, GPY_2019}, are second-order in time whereas noise only appears in the equation for the velocity, with such a setting leading naturally to hypoelliptic SDEs. Furthermore, the Poisson-Vlasov-Fokker-Planck system~\cite{Herau_2016}, consisting of the Poisson equation coupled to the kinetic Fokker-Planck one, a most fundamental PDE system in plasma physics and galactic dynamics, can be obtained as the mean field limit of a system of weakly interacting hypoelliptic diffusions, of the form that we consider in this paper.

For standard hypoelliptic diffusions, i.e.~linear in the sense of McKean, parameter inference has been a very active research area in Statistics over the past two decades. 
\cite{poke:09} developed a locally non-degenerate Gaussian approximation to form a well-defined joint-likelihood for rough and smooth components rather than employing the Euler-Maruyama scheme that induces degeneracy. The authors numerically demonstrate that the joint likelihood performs well for estimating diffusion parameters when combined with MCMC methodologies, in a partial observation regime where only the smooth component is observed. 
However, the proposed likelihood ultimately leads to a biased estimation of the drift parameter in the rough component. 
\cite{SaTh:12} developed Euler-Maruyama (EM)-based contrast estimators under complete/partial observation regimes. In the latter scenario, the authors assume that the dynamics of the smooth component $X_S$ is determined via $dX_{S, t} = X_{R, t} dt$, where $X_{R, t}$ represents the rough component, and employ arguments in \cite{Gl:06} to develop a modified EM-based contrast function with the hidden component $X_R$ recovered via finite difference approximation from $X_{S}$'s dynamics. Note that the inferential approach based on the EM approximation requires that the smooth drift function is independent of any parameters and typically of the form $dX_{S, t} = X_{R, t} dt$. This restriction was later overcome in ~\cite{dit:19, mel:20, glot:21} by introducing a new locally Gaussian (LG) approximation that incorporates a higher-order term in the conditional mean approximation of the smooth component. 
Then, assuming ergodicity for the process and that all coordinates are observed, these works showed the asymptotic normality for the LG-based contrast estimators under the usual high-frequency complete observation regime, i.e., $\Delta_n \to 0$, $n\to \infty$ and $n \Delta_n \to \infty$ with the additional condition of $\Delta_n = o (n^{-1/2})$ on the data step-size, where $n$ is the number of observations. The latter condition on $\Delta_n$ was weakened as $\Delta_n = o (n^{-1/p})$, $p \ge 3$, in~\citep{igu:ejs,igu:bj}, and a similar high-frequency analysis was conducted by \cite{igu:24} for a specific class of highly degenerate hypoelliptic diffusions, referred to as (Hypo-II) in their study, which was inspired by the class of generalized Langevin equations.
Further contributions have involved treating hypoelliptic SDEs in a multilevel Monte-Carlo setting to improve accuracy of expectation estimates \cite{iguchi2025antithetic}.

In this paper, extending the aforementioned developments for standard hypoelliptic diffusions, we construct and analyze an approximate likelihood and contrast estimator for inferring parameters in IPSs consisting of multidimensional hypoelliptic SDEs. 
In particular, we establish asymptotic results for the proposed estimator by combining the ideas of \citep{amo:23, amo:24} and \citep{glot:21, igu:ejs, igu:bj}.  Furthermore, to provide a complete picture, we also provide results for weakly interacting elliptic diffusions as a multivariate extension of  \cite{amo:23} in terms of the state space of each particle.  
To our knowledge, the only papers that contribute to the statistical inference problem for mean field hypoelliptic diffusions are~\citep {amo:24} for unit dimension in space (i.e.~two-dimensional phase space) and \citep{Sanchez:25} for the interacting FitzHugh-Nagumo model (introduced later) via a moment-based estimator. 
In \citep {amo:24}, the authors construct a contrast estimator based on the \textcolor{black}{Euler-Maruyama} discretization of the SDE and then show its asymptotic normality under the setting $n, N \to \infty$ with a design condition $N \Delta_n \to 0$, where $n, N$ and $\Delta_n$ are the number of observations, particles and the step-size of data, respectively. Note that the time-interval of observations $[0,T]$ is assumed to be fixed, and ergodicity is not required. The setting and techniques that we consider in our paper are similar, but with some crucial differences: 
(a) we allow a multivariate setting; 
(b) we consider a locally non-degenerate Gaussian approximation instead of the Euler-Maruyama discretization that is locally degenerate; 
(c) due to the non-degenerate approximation, a more general class of interacting hypoelliptic diffusions can be treated; and (d) the asymptotic variance of our estimator for diffusion parameters can be smaller than the one obtained in~\cite{amo:24}.
\subsection{Settings} \label{sec:model}  
Throughout the paper we consider the following model class. The system has $N$ weakly interacting particles each modeled by a $d$-dimensional SDE, 
i.e.~for~$1\le i \le N$: 
\begin{align}
\begin{aligned}
& d X_t^{[i]} 
= 
\begin{bmatrix}
d X_{S, t}^{[i]}  \\[0.2cm]
d X_{R, t}^{[i]} 
\end{bmatrix}
= 
\begin{bmatrix}
V_{S, 0} (\alpha_S, X_t^{[i]}) \\[0.2cm]
V_{R, 0} (\alpha_{R}, X_t^{[i]}, \mu_t^{N}) 
\end{bmatrix} dt 
+ \sum_{j = 1}^{d_B} 
\begin{bmatrix}
\mathbf{0}_{d_S} \\[0.2cm]
V_{R, j} (\beta, X_t^{[i]}, \mu_t^{N}) dB_{j,t}^{[i]} 
\end{bmatrix},  \\[0.2cm] 
& \mathcal{L} (X_0^{[1]}, \ldots, X_0^{[N]}) = \mu_0 \times \cdots \times \mu_0, 
\end{aligned} \tag{IPS} 
\label{eq:ips-1}
\end{align}
where $X_t^{[i]} = (X_{S, t}^{[i]}, X_{R, t}^{[i]}) \in \mathbb{R}^{d_S} \times \mathbb{R}^{d_R}, \, \textcolor{black}{d_S \ge 0}, \,  d_R \ge 1, \, d_S + d_R = d$. Above, we~have:
\begin{itemize}[leftmargin=0.2cm]
\item $\theta = (\alpha_S, \alpha_R, \beta) 
\in \Theta = \Theta_{\alpha_S} \times \Theta_{\alpha_R} \times \Theta_{\sigma} \subseteq \mathbb{R}^{d_{\alpha_S}} \times \mathbb{R}^{d_{\alpha_R}} \times \mathbb{R}^{d_\beta}$ is the unknown parameter with  
\textcolor{black}{$d_{\alpha_S} \ge 0$}, 
$d_{\alpha_{R}} \ge 1$,  $d_{\beta} \ge 1$ satisfying  
$d_\theta = d_{\alpha} + d_{\beta} = d_{\alpha_S} + d_{\alpha_R} + d_{\beta}$ and \textcolor{black}{$d_{\alpha_S} = 0$ if $d_S = 0$}, where $\Theta_{\alpha_S}, \Theta_{\alpha_R}$ and $\Theta_{\sigma}$ are compact and convex sets;
\item The coefficients are defined as:
\begin{gather*}
V_{S, 0} : \Theta_{\alpha_S} \times \mathbb{R}^d \to \mathbb{R}^{d_S}, \quad 
V_{R, 0} : \Theta_{\alpha_R} \times \mathbb{R}^d \times \mathcal{P}_2 (\mathbb{R}^d) \to \mathbb{R}^{d_R}; \\ 
V_{R, j} : \Theta_{\beta} \times \mathbb{R}^d \times \mathcal{P}_2 (\mathbb{R}^d) \to \mathbb{R}^{d_R}, \quad 1 \le j \le d_B, 
\end{gather*}
with $\mathcal{P}_2 (\mathbb{R}^d)$ denoting the set of probability measures on $\mathbb{R}^d$ with a finite second moment, endowed with the 2-Wasserstein metric;   
\item $\{B_t^{[i]}\}_{t \ge 0} = \{ (B_{1, t}^{[i]}, \ldots, B_{d_B, t}^{[i]}) \}_{t \ge 0}$, $i = 1, \ldots, N$, are independent $d_B$-dimensional standard Brownian motions and are also independent of the initial particles $\{X_0^{[i]}\}_{i = 1, \ldots, N}$;  
\item $\mu_t^{N}$ is the empirical measure of the system at time $t$, defined as $\textstyle \mu_t^{N} = \tfrac{1}{N }\sum_{i = 1}^N \delta_{X_t^{[i]}}$.  
\end{itemize}  
\textcolor{black}{The particles in the initial state $\{X_0^{[i]}\}_{i = 1, \ldots, N}$ are assumed to be independent of the parameter $\theta$.}
For this model class, when $d_S \ge 1$, we assume that the drift function in the smooth component $V_{S,0}$ depends on the rough component and that the noise of the rough component directly propagates into the smooth component.  Also, the interacting particle system (\ref{eq:ips-1}) corresponds to the following $d$-dimensional McKean-Vlasov SDE by taking the limit $N \to \infty$:  
\begin{align} \label{eq:mv_I}
d X_t = 
\begin{bmatrix}
d X_{S, t} \\[0.1cm] 
d X_{R, t}
\end{bmatrix} 
= 
\begin{bmatrix}
V_{S, 0} (\alpha_S, X_t) \\[0.1cm] 
V_{R, 0} (\alpha_R, X_t, \mu_t^\theta) 
\end{bmatrix}  dt 
+ 
\sum_{j = 1}^{d_B} 
\begin{bmatrix}
\mathbf{0}_{d_S} \\[0.1cm]
V_{R, j} (\beta,  X_t, \mu_t^\theta) dB_{j, t}
\end{bmatrix},   
\end{align} 
where $\{B_t\}_{t \ge 0} = \{  (B_{1, t}, \ldots, B_{d_B, t}) \}_{t \ge 0}$ is a standard Brownian motion, \textcolor{black}{$X_0 \sim \mu_0$ is independent of $B$}, and $\mu_t^\theta$ is the law of $X_t$ under the parameter value $\theta \in \Theta$. 
\\ 

For the model class in equation (\ref{eq:ips-1}), we investigate parameter estimation within data frameworks specified as follows. We write the set of $(n+1)$-discrete time instances, $n \in \mathbb{N}$, as $\mathbb{T}_n = \{t_{j} \}_{0 \le j \le n}$, where $t_j = j T/ n$ with the time interval $[0,T]$ being fixed. In addition, we denote by $\Delta_n = T/n$ the (equidistant) step-size between observations. For $t \in \mathbb{T}_n$, we write 
$
\mathbb{X}_{t}^N = \bigl\{ X_{t}^{[i]} \bigr\}_{1 \le i \le N}$ and 
$ \mathbb{X}_{S, t}^N = \bigl\{ X_{S, t}^{[i]} \bigr\}_{1 \le i \le N} 
$ 
for model class (\ref{eq:ips-1}). Throughout the paper we will consider the following two observation regimes: 
\begin{enumerate}
\item[Complete:] All coordinates of particle trajectories $\bigl\{  \mathbb{X}_{t}^N \bigr\}_{t \in \mathbb{T}_n}$ are observed; 
\item[Partial:] Some coordinates of particle trajectories are observed, typically~$\bigl\{ \mathbb{X}_{S, t}^N \bigr\}_{t \in \mathbb{T}_n}$. 
\end{enumerate} 

\subsection{Motivating Model Examples}
The class of hypoelliptic interacting diffusions (\ref{eq:ips-1}) that we consider in this paper is sufficiently broad to encompass several important examples that arise in applications. Here, we present two such cases. 
\\

\noindent 
\emph{1. Interacting Underdamped Langevin dynamics.} We consider a system of $N$ particles in a confining potential, linked via an interaction potential and subject to noise and dissipation~\cite[Ch. 6]{pavl:14}. The $i$-th particle $(q_t^i, p_t^i) \in \mathbb{R}^{d_S} \times \mathbb{R}^{d_R}$ (with $d_S\equiv d_R=: d'$) writes as:  
\begin{align} \label{eq:I-UL} 
\begin{aligned} 
d q_t^{[i]} & = p_t^{[i]} dt; \\ 
d p_t^{[i]} & =  - \nabla V (q_t^{[i]}) dt 
- \tfrac{\kappa}{N} \sum_{j = 1}^N \nabla U (q_t^{[i]} - q_t^{[j]}) dt - \gamma p_t^{[i]} dt + \sqrt{2 \gamma \beta^{-1}} \, d B_t^{[i]},  
\end{aligned} \tag{I-UL}
\end{align}
where $V, U : \mathbb{R}^{d'} \to \mathbb{R}$ are some potential functions, $B_t^{[i]}, t \ge 0$, are $d'$-dimensional standard Brownian motions and positive parameters $\gamma, \beta$ and $\kappa$ are the friction, the inverse temperature and the interaction strength, respectively. (\ref{eq:I-UL}) belongs to model class (\ref{eq:ips-1}).

\noindent 
\\ 
\emph{2. Interacting FitzHugh-Nagumo SDEs  \cite{lu:21,Mischler_al_2016}.} The model describes the time evolution of the membrane potential of interacting neurons. This IPS can be expressed as:  
\begin{align} \label{eq:I-FHN} 
\begin{aligned}
d Y^{[i]}_t & = \tfrac{1}{c} (X^{[i]}_t + a - b Y_t^{[i]} ) dt; \\
d X^{[i]}_t & = \bigl( X_t^{[i]} - \tfrac{(X_t^{[i]})^3}{3} - Y_t^{[i]} \bigr) dt - \tfrac{\kappa}{N} \sum_{j = 1}^N \bigl(  X^{[i]}_t - X_t^{[j]} \bigr) dt + \sigma d B_t^{[i]}, 
\end{aligned} \tag{I-FHN}
\end{align} 
for $i = 1, \ldots, N$, where $a \in \mathbb{R}$ and $b, c, \kappa, \sigma$ are positive parameters (here $d_S=d_R=d'=1$). $Y_t^{[i]}$ and $X_t^{[i]}$ represent the recovery and the voltage variable, respectively, of the $i$-th neuron. Note that compared to (\ref{eq:I-UL}), the drift function of the smooth component $Y_t^{[i]}$ involves some parameters and the state itself, thus~\eqref{eq:I-FHN} is out of the scope of the recent contribution \cite{amo:24}. 
\cite{PaZ24} make use of a method of moments (which utilizes the fact that the drift in~\eqref{eq:I-FHN} is a polynomial) to infer the parameters from observations of the SDE system.
We note that likelihood-based approaches pursued here have well-established optimality properties versus method of moment methods.   
%
\subsection{Contribution and Organization} The main objective of this paper is to develop and analyze parameter estimation methodologies for~\eqref{eq:ips-1} from discrete-time particle observations. As in the case of non-interacting hypoelliptic diffusions, the model likelihood and transition density are intractable. Thus, our focus is on developing tractable approximate model dynamics that enable likelihood-based inference within various data frameworks, including the partial observation regime.  
We stress that designing such a likelihood is a 
non-trivial task due to the degenerate noise structure. We build an approximate \emph{joint likelihood} by allowing Gaussian variates to propagate onto the smooth particles $\{ X_{S, t}^{[i]}\}_{i = 1, \ldots, N}$ via an appropriate high-order stochastic Taylor expansion of the smooth drift function $V_{S, 0} (\alpha_S, X_t^{[i]})$, while Euler-Maruyama approximation is applied to the rough particles $\{ X_{R, t}^{[i]}\}_{i = 1, \ldots, N}$. 
This extends the arguments for standard (non-interacting) hypoelliptic diffusions studied in \citep{glot:21,  igu:ejs, igu:24}. Then, our choice of the proposed likelihood is justified by a rigorous asymptotic analysis of the deduced \textcolor{black}{quasi-maximum likelihood estimator} or contrast estimator.
In particular, we show under appropriate conditions that the estimator is asymptotically normal as $n, N \to \infty$ with the \emph{design condition} $\Delta_n = o (N^{-1})$, i.e.~$N \Delta_n \to 0$, which is the asymptotic regime recently studied in \cite{amo:23, amo:24}. 
We remark that under our observation setting, the time-interval $[0,T]$ is fixed and that $n \to \infty$ implies $\Delta_n = T/n \to 0$. 
\textcolor{black}{Notably, the asymptotic result is obtained under the locally Lipschitz condition in the state for the drift function, as detailed in Assumption \ref{ass:growth}. Our design of the approximate likelihood partially relies on the Euler-Maruyama scheme and can therefore be unstable as a numerical integrator in this setting. Nevertheless, we stress that the introduced time-discretization is tailored to parametric inference under high-frequency observation regimes and indeed yields the desirable asymptotic result. This is essentially because the contrast function is evaluated solely with the data -- the true dynamics of (\ref{eq:ips-1}) -- which is well-posed with finite moments under our introduced assumptions. This point is recalled later in Remark \ref{rem:local_lip}.}

After validating the developed approximate likelihood through the asymptotic analysis, we also discuss the partial observation regime and emphasize that the proposed non-degenerate approximation enables filtering or data augmentation for latent particle trajectories, once combined with standard Bayesian statistics methodologies. We can thus carry out inference within a broad class of hypoelliptic IPSs and for general observation regimes without necessitating a specific structure on the smooth particles, such as $dX_{S, t}^{[i]} = X_{R, t}^{[i]} dt$, which is a crucial requirement for the Euler-Maruyama-based approach \cite{amo:24}.  
\\ 

Our main contributions are summarized as follows: 
\begin{enumerate}[leftmargin=0.3cm]
\item[a.] We construct a new contrast estimator for a broad class of weakly interacting hypoelliptic SDEs. In more detail, we develop an approximate joint transition density for $X_{t_{j+1}}^{[i]} | X_{t_j}^{[i]}$, which we can call upon 
for likelihood-based parameter estimation, both in full and partial observation settings.
Such a non-degenerate likelihood/transition density approximation is carefully designed to avoid biases in estimation methodologies that appear in earlier literature \cite{poke:09} for linear (in the sense of McKean) hypoelliptic diffusions. 
\item[b.] Under the high-frequency complete observation regime, we prove consistency and asymptotic normality for the proposed estimator. In particular, a new CLT rate is obtained for parameter $\alpha_S$ contained in the drift function of $X_{S, t}^{[i]}$. In addition, the asymptotic variance for the diffusion parameter estimator can be smaller than that of the Euler-Maruyama-based one \cite{amo:24} due to utilising the full joint likelihood of smooth and rough components. Our proof of the asymptotic results
and analysis of the asymptotic behaviors of the proposed non-degenerate likelihood
requires different and more careful arguments from those in \cite{amo:24}. 
\item[c.] We numerically verify the effectiveness of our proposed joint likelihood in applications involving both complete and partial observation regimes, \textcolor{black}{where in the latter case we compute estimates by optimizing a marginal likelihood based upon the proposed approximate joint transition density. While theoretical asymptotic properties are not presented for the partial observation regime, the numerical results are encouraging, demonstrating asymptotic unbiasedness of the estimator. }
\end{enumerate} 

The paper is organized as follows. Section \ref{sec:main} develops our non-degenerate approximate likelihood and provides the corresponding asymptotic results under the high-frequency complete observation regime. The partial observation setting is also discussed in this section. Section \ref{sec:numerics} presents numerical experiments. Proofs for the main analytic results are collected in Section \ref{sec:pfs}. We conclude in Section \ref{sec:conclusion}. 
\\

\noindent 
\textbf{Notation.} 
Throughout the paper, we will use the following notation: 
\begin{align*}
V_0 (\theta, x, \mu) & = 
\begin{bmatrix}
V_{S, 0} (\alpha_S, x) \\[0.1cm] 
V_{R, 0} (\alpha_R, x, \mu)
\end{bmatrix}, 
\quad 
V_{k} (\theta, x, \mu) = 
\begin{bmatrix} 
\mathbf{0}_{d_S} \\[0.1cm] 
V_{R, k} (\beta, x, \mu) 
\end{bmatrix}, 
\quad  1 \le k \le d_B.  
\end{align*}
We denote by $\mathbb{P}_\theta$ the law of process $\{X_t \}_{t \ge 0}$ under the parameter value $\theta \in \Theta$ and write the expectation under the probability measure $\mathbb{P}_\theta$ as $\mathbb{E}_\theta [\cdot]$.  
Symbols
$\xrightarrow{\mathbb{P}_{\theta}}$,  $\xrightarrow{\mathcal{L}_{\theta}}$ denote convergence in probability and distribution, respectively, associated with $\mathbb{P}_\theta$.
We define 
\textcolor{black}{$$
\mathcal{F}_t^N:= \sigma \bigl\{ X_0^{[i]}, \{ B_s^{[i]} \}_{s \in [0, t]} \, : \,  i = 1, \ldots , N \bigr\}
. $$} 
$\mathcal{P} (E)$ denotes the set of probability measures on a measurable space $(E, \mathcal{E})$, and $\mathcal{P}_p (E)$, for $p \in \mathbb{N}$, denotes the subspace of $\mathcal{P}(E)$ with $p$-th finite moments.  
The $p$-Wasserstein metric $\mathcal{W}_p $ on $\mathcal{P}_p (\mathbb{R}^d)$ is defined as follows, for $\mu, \nu \in \mathcal{P}_p (\mathbb{R}^d)$: 
$$ 
\mathcal{W}_p (\mu, \nu) = \inf_{\pi} \Bigl\{  
\bigl( \textstyle\int_{\mathbb{R}^d \times \mathbb{R}^d} 
|x- y|^p \pi (dx, dy) \bigr)^{1/p} \, : \,  
\pi  \in \mathcal{P}_p (\mathbb{R}^d \times \mathbb{R}^d) \  \mathrm{with \ marginals} \ \mu, \nu \Bigr\}.  
$$
We write $\mathbb{E}_\mu[\varphi(W)]$, for a $\varphi:\mathbb{R}^{d}\to \mathbb{R}^{d_0}$, $d_0\ge 1$, to represent an expectation under $W\sim \mu$. For $z \in \mathbb{R}^m, \, m \in \mathbb{N}$, we define $\partial_z = [\partial/\partial z_1, \ldots, \partial/\partial z_m]^\top$ and $\partial_{z, k} \equiv \partial / \partial z_k, \, 1 \le k \le m$. For $k, d_1, d_2 \in \mathbb{N}$, we define $C_p^k (\Theta \times \mathbb{R}^{d_1}, \mathbb{R}^{d_2})$ as the space of functions $f: \Theta \times \mathbb{R}^{d_1} \to \mathbb{R}^{d_2}$ such that $x \mapsto f(\theta, x)$ are $k$-times differentiable for all $\theta \in \Theta$ and all their partial derivatives up to order $k$ have polynomial growth in $x \in \mathbb{R}^{d_1}$ uniformly in $\theta \in \Theta$.  
%

\section{Parameter Estimation for Interacting Hypoelliptic Diffusions}
\label{sec:main}
%
A primary objective of this paper is to construct an approximate \emph{joint} likelihood function that will allow both frequentist and Bayesian inference to be conducted over model class (\ref{eq:ips-1}) under complete and partial observation regimes. 
The construction of such an appropriate likelihood proxy via an underlying time-discretization of (\ref{eq:ips-1}) requires careful steps in a hypoelliptic setting as the standard Euler-Maruyama-type discretization cannot yield a well-defined joint data-likelihood due to the degenerate SDE noise. Hence, we instead introduce a tailored time-discretization, similar in spirit to the one developed in \citep{glot:21, igu:24, igu:bj} for non-interacting hypoelliptic diffusions. We note two key points here: 
(i) to build a non-degenerate covariance for each particle, we apply a higher-order stochastic Taylor expansion to $V_{S, 0} (\alpha_S, X_t^{[i]})$ so that a conditional approximation of $X_{S, t_{j+1}}^{[i]}$ given $\mathbb{X}_{t_j}^N$ contains Gaussian variates of size $\mathcal{O} (\Delta_n^{3/2})$;  
(ii) upon introduction of the  Gaussian noise the size $\mathcal{O} (\Delta_n^{3/2})$, the approximation of the drift function in the smooth coordinates must be carefully designed to include terms of sufficiently high order in $\Delta_n$, otherwise biases will emerge at estimates of the parameters in the drift of the rough components. Our construction of the joint likelihood below proceeds in manner that respects the above two  considerations.  
%
%
%
%
\subsection{Construction of Approximate Likelihood}
\subsubsection{Main Assumptions}
We present here a few initial assumptions on the model structure of (\ref{eq:ips-1}). Additional regularity assumptions are introduced later on. 
%
\begin{assump}[Linearity in $\mu$-Component]  \label{ass:law_dep}
(i) The coefficients of the rough components in  (\ref{eq:ips-1})  have the following form. For $(x, \mu) \in \mathbb{R}^d \times \mathcal{P}_2 (\mathbb{R}^d)$, 
$(\alpha_{R},  \beta) \in \Theta_{\alpha_R} \times \Theta_{\beta}$:
\begin{align}
\begin{aligned}  
V_{R, 0} (\alpha_R, x, \mu) 
& =  V_{R, 0}^{I} (\alpha_{R}, x) + 
\mathbb{E}_\mu\big[V_{R, 0}^{II}(\alpha_{R}, x,W)\big];
\\ 
V_{R, j}  (\beta, x, \mu) 
& =  V^{I}_{R, j} (\beta, x) 
+ 
\mathbb{E}_{\mu}\big[V_{R, j}^{II}(\beta, x,W)\big],
\quad 1 \le j \le d_{B},  
\end{aligned}  \label{eq:phi_R}
\end{align}
with $V^{I}_{R,0}: \Theta_{\alpha_R} \times \mathbb{R}^d \to \mathbb{R}^{d_R}$, $V^{I}_{R, j}: \Theta_{\beta} \times \mathbb{R}^{d} \to \mathbb{R}^{d_R}$ and some Borel measurable functions $V^{II}_{R, 0} : \Theta_{\alpha_R} \times \mathbb{R}^d \times \mathbb{R}^d \to \mathbb{R}^{d_R}$, $V^{II}_{R, k} : \Theta_{\beta} \times \mathbb{R}^d \times \mathbb{R}^d \to \mathbb{R}^{d_R}$, $0 \le k \le d_B$. 
\\ 
%
%
%
\end{assump} 
\begin{assump}[Uniform H\"ormander-Type Condition] \label{ass:hor}
We assume that: 
\textcolor{black}{
\begin{itemize}[leftmargin=0.2cm]
\item If $d_S = 0$, then  
\begin{align*}
\inf_{(\beta, x, \mu) \in \Theta_{\beta} \times \mathbb{R}^d \times \mathcal{P}_2(\mathbb{R}^d)} \det a (\beta, x, \mu) > 0,  
\end{align*}
where $a (\beta, x, \mu) := V (\beta, x, \mu) V (\beta, x , \mu)^\top \in \mathbb{R}^{d \times d }$ with $V \equiv [V_{1}, \ldots, V_{d_B}]$. 
\item If $d_S \ge 1$, then 
    \begin{align*}
\inf_{(\beta, x, \mu) \in \Theta_{\beta} \times \mathbb{R}^d \times \mathcal{P}_2(\mathbb{R}^d)} \det a_R (\beta, x, \mu) > 0, 
\quad 
\inf_{(\theta, x, \mu) \in \Theta \times \mathbb{R}^d \times \mathcal{P}_2(\mathbb{R}^d)} 
\det a_S (\theta, x, \mu)  > 0,    
\end{align*} 
where we defined 
%
\begin{align}
a_R (\beta, x, \mu) 
&:= V_R (\beta, x, \mu) V_R (\beta, x, \mu)^\top; 
\nonumber 
\\
a_S (\theta, x, \mu) &:= \partial_{x_R}^\top {V}_{S, 0} (\alpha_S, x) a_R (\beta, x, \mu) 
\bigl( \partial_{x_R}^\top {V}_{S, 0} (\alpha_S, x) \bigr)^\top, 
\label{eq:a_S}
\end{align}
with $V_R \equiv [V_{R, 1}, \ldots, V_{R, d_B}]$. 
\end{itemize}
} 
\end{assump} 
Assumption~\ref{ass:hor} ensures that the SDE system~\eqref{eq:ips-1}  forms a hypoelliptic SDE, in the sense of H\"{o}rmander~\cite[Sec. 6.2]{pavl:14} or~\cite[Sec. V.38]{RW00b}. In particular, the transition density of the $N$-particle system has a Lebesgue density: Equation~\eqref{eq:a_S} implies that the Brownian noise, $\{B_t^{[i]}\}$, in the rough component $X_{R, t}^{[i]}$ propagates onto $X_{S ,t}^{[i]}$ via the state variable $X_t^{[i]}$ of the drift $V_{S,0}(X_t^{[i]})$ \textcolor{black}{so that vector fields associated with the sets of noise vectors $\textstyle{\int_{s}^t \int_s^u dB_v^{[i]} du}$ and $B_t^{[i]} - B_s^{[i]}, \, 0 \le s \le t$, spans $\mathbb{R}^d$.} 
\textcolor{black}{
\begin{remark} \label{rem:GLE}
The latter condition of Assumption \ref{ass:hor} covers model examples that cannot be treated by the uniform ellipticity condition with $d_S = 0$ and $d_R = d$. The motivated examples (\ref{eq:I-UL}) and (\ref{eq:I-FHN}) clearly satisfy the condition. For instance, for the model (\ref{eq:I-FHN}), we have $\det a_R (\beta, x, \mu) = \sigma^2 > 0$ and $\det a_S (\theta, x, \mu) = {\sigma^2}/{c^2} > 0$ for the parameters $\sigma, c > 0$. 
However, Assumption \ref{ass:hor} is restrictive compared to the general H\"ormander condition that can be satisfied with additional higher-order commutators, and excludes the case where the noise of $X_{R, t}^{[i]}$ does not directly propagate into all coordinates of $X_{S, t}^{[i]}$. Indeed, the uniform positive definiteness of the matrix $a_s$ imposes $d_S \le d_R$ and $\mathrm{rank} \bigl(\partial_{x_R}^\top V_{S, 0} (\alpha_S, x) \bigr) = d_S$ for all $(\alpha_S, x) \in \Theta_{\alpha_S} \times \mathbb{R}^d$. An interesting model example not covered by Assumption \ref{ass:hor} is the interacting generalized Langevin equations under a Markovian approximation~\cite{OttobrePavliotis11, DuongPavliotis2018}, with each particle defined as follows. For $i = 1, \ldots, N$, 
\begin{align}
\begin{aligned} \label{eq:I-GL}
d q_t^{[i]} & = p_t^{[i]} dt; \\ 
d p_t^{[i]} & = \Bigl( - \nabla V \bigl(q_t^{[i]} \bigr) - \tfrac{1}{N} \sum_{j = 1}^N \nabla U \bigl(q_t^{[i]} - q_t^{[j]} \bigr) + \lambda^\top z_t^{[i]} \Bigr) dt; \\ 
d z_t^{[i]} & = \bigl( - \lambda p_t^{[i]} - A z_t^{[i]} \bigr) dt + \sqrt{2 \beta^{-1} A} d B_t^{[i]}, 
\end{aligned} \tag{I-GL}
\end{align}
where $(q_t^{[i]}, p_t^{[i]}, z_t^{[i]}) \in \mathbb{R}^d \times \mathbb{R}^d \times \mathbb{R}^{d m}$ for $d, m \in \mathbb{N}$; $\{B_t^{[i]}\}$ is an independent $dm$-dimensional Brownian motion; and $d_S = 2d, \, d_R = dm$. 
In the above, $\lambda \in \mathbb{R}^{d m \times d}$ and $A \in \mathbb{R}^{dm \times dm}$ are given matrices where $A$ is symmetric positive definite; $\beta > 0$ is the inverse temperature; and $V, U: \mathbb{R}^d \to \mathbb{R}$ are confining and interaction potentials, respectively. 
For this model, we have $ a_R = 2 \beta^{-1} A$ which is positive definite, but the matrix $a_S$ is not because $\partial_{x_R} V_{S, 0} (\alpha_S, x)$ is no longer full-rank, which comes from the fact that the drift of the component $q_t^{[i]}$ is independent of $z_t^{[i]}$. In contrast to (\ref{eq:I-UL}) and (\ref{eq:I-FHN}), the noise propagation within the system (\ref{eq:I-GL}) takes two steps: the noisy component $z_t^{[i]}$ first affects $p_t^{[i]}$ via its drift and then arrives at the equation of $q_t^{[i]}$ through its drift function $p_t^{[i]}$. Weakening Assumption \ref{ass:hor} to cover (\ref{eq:I-GL}) could be possible by further considering hierarchical block matrices as in the works \cite{igu:24, igu:ejs} for the non-interacting, ergodic hypoelliptic diffusions, but we leave it as potential future work.    
\end{remark}
}
\subsubsection{Approximate Likelihood for (\ref{eq:ips-1})}
We start via some notation. We write the true parameter value as  
$
\trueparam = 
(\alpha_{R}^\dagger, \beta^\dagger)$ for $d_S = 0$ and 
$\trueparam = (\alpha_{S}^\dagger, \alpha_{R}^\dagger, \beta^\dagger)$ for $d_S \ge 1$,
and we assume it to lie in the interior of the parameter space $\Theta$. \textcolor{black}{In what follows, $X_t^{[i]}, \, t \ge 0, \, i = 1, \ldots, N,$ denotes the solution of (\ref{eq:ips-1}) under the true parameter $\theta^\dagger$, and $\mu_t^N$ is the empirical measure of $\{X_t^{[i]}\}_{i = 1, \ldots, N}$}. We also write:   
\begin{align}
\label{eq:tripl}
{Z}_{j}^{[i], \theta} = (\theta, X_{t_j}^{[i]},  \textcolor{black}{\mu_{t_j}^{N}}), \qquad 1 \le i \le N, \quad \ 0 \le j \le n,  
\end{align}
with $\theta \in \Theta$. For twice differentiable function $f : \mathbb{R}^d \to \mathbb{R}$ and $z = (\theta, x, \mu) \in \Theta \times \mathbb{R}^d \times \mathcal{P}_2 (\mathbb{R}^d)$, we define: 
\begin{align} \label{eq:L}  
\begin{aligned} 
\mathscr{L}_m [f] (z) 
& := \sum_{1 \le k \le d} V_m^k (z) \partial_{x_k} f (x) 
\\ 
& +  \tfrac{1}{2} \sum_{1 \le k_1, k_2 \le d} a_{k_1 k_2} (z) \partial_{x_{k_1}} \partial_{x_{k_2}} f(x) \times \mathbf{1}_{m = 0}, 
\end{aligned} \quad m = 0,1, \ldots, d_B,  
\end{align} 
where $a = VV^\top$ with $V \equiv [V_1, \ldots, V_{d_B}]$. For $F: \mathbb{R}^d \to \mathbb{R}^{d'}$ with $d' \ge 1$, we interpret $\mathscr{L}_m [F] (z) \equiv \bigl[\mathscr{L}_m [F^1] (z), \ldots, \mathscr{L}_m [F^{d'}] (z) \bigr]^\top$. \textcolor{black}{For notational ease, we write: for $1 \le i \le N$ and $0 \le j \le n$, 
\begin{align} \label{eq:int}
I_{(k, 0)}^{[i]} (j) 
=
\begin{cases}
    \textstyle{ \int_{t_j}^{t_{j+1}} \int_{t_j}^u ds du = \tfrac{\Delta_n^2}{2}}, & k = 0; \\[0.3cm] 
    \textstyle{ \int_{t_j}^{t_{j+1}} \int_{t_j}^u d B_{k, s}^{[i]} du}, 
    & k = 1, \ldots, d_B.  
\end{cases}
\end{align}}

We can now introduce a time-discretisation for (\ref{eq:ips-1}) that provides a tractable transition density approximation, resulting in a proxy likelihood with desirable statistical properties. For 
$1 \le i \le N$, $0 \le j \le n$, and $\theta = (\alpha_S, \alpha_R, \beta) \in \Theta$, \textcolor{black}{we consider the following one-step approximation $X_{t_{j+1}}^{[i]} \approx \widetilde{X}_{t_{j+1}}^{[i]} = (\widetilde{X}_{S, t_{j+1}}^{[i]}, \widetilde{X}_{R, t_{j+1}}^{[i]})$ given the current state $\widetilde{X}_{t_{j}}^{[i]} = x = (x_S, x_R)$ and the empirical measure $\mu$: 
\begin{align} \label{eq:LG-1}
\begin{aligned} 
\widetilde{X}_{S, t_{j+1}}^{[i]} 
& = x_S 
+ V_{S, 0} (\alpha_{S}, x) \Delta_n 
+ \! \! \! \!  \sum_{0 \le k \le d_B} \! \! \! \!   
\mathscr{L}_k \bigl[ {V}_{S, 0} (\alpha_S, \cdot) \bigr] (z^\theta) \times I_{(k, 0)}^{[i]} (j); \\[0.1cm] 
\widetilde{X}_{R, t_{j+1}}^{[i]} 
& = x_R 
+ V_{R, 0} (\alpha_R, x, \mu) \Delta_n  
+ \sum_{1 \le k \le d_B} 
V_{R, k} (\beta, x, \mu) \times \bigl(B_{k, t_{j+1}}^{[i]} - B_{k, t_{j}}^{[i]}  \bigr).  
\end{aligned} \tag{LG}
\end{align}
where $z^\theta = (\theta, x, \mu)$. } 
The summand with $k=0$ in the expression for the smooth coordinates corresponds to the term $\mathscr{L}_k[ {V}_{S, 0} (\alpha_S, \cdot)](z^\theta) \Delta_n^2/2$.
In (\ref{eq:LG-1}) the rough component $X_{R}^{[i]}$ is treated via the Euler-Maruyama scheme. In contrast, the discretization of the smooth part $X_S^{[i]}$ is obtained via a higher-order It\^o-Taylor expansion of the drift  $V_{S, 0}$. The Brownian motion in the rough coordinate $X_R$ propagates onto $X_S$ via the drift ${V}_{S, 0}$ in the form of \textcolor{black}{$\textstyle I_{(k,0)}^{[i]} (j)$}, which also follows a Gaussian distribution and has a correlation with $B_{k, t_{j+1}}^{[i]} - B_{k, t_{j}}^{[i]}$. Specifically, we have that:  
\begin{align*}
\textstyle 
\mathrm{Cov} \Bigl( \,  I_{(k_1, 0)}^{[i]} (j), \, B_{k_2, t_{j+1}}^{[i]} - B_{k_2, t_{j}}^{[i]} \Bigr)  = 
\begin{bmatrix}
\tfrac{\Delta_n^3}{3} & \tfrac{\Delta_n^2}{2} \times \mathbf{1}_{k_1 = k_2} \\[0.2cm] 
\tfrac{\Delta_n^2}{2} \times \mathbf{1}_{k_1 = k_2} & \Delta_n  
\end{bmatrix}, 
\end{align*}
for $1 \le k_1, k_2 \le d_B$, which implies that \textcolor{black}{$\widetilde{X}_{t_{j+1}}^{[i]}$ given the state and empirical measure $(x, \mu)$ at time $t_j$ follows a non-degenerate Gaussian law}. 

Building upon the scheme (\ref{eq:LG-1}), we propose an approximate likelihood (or `contrast function'). For notational simplicity, we make use of the following standardisation for (\ref{eq:LG-1}). For $1 \le i \le N$, $1 \le j \le n$, $\theta = (\alpha_S, \alpha_R, \beta) \in \Theta$, we set: 
\begin{align}
\label{eq:def_m}
\mathbf{m}_{j-1}^{[i], \theta} \equiv  
\begin{bmatrix}
\mathbf{m}_{S, j-1}^{[i], \theta} \\[0.2cm] 
\mathbf{m}_{R, j-1}^{[i], \theta} 
\end{bmatrix}
:= 
\begin{bmatrix}
\tfrac{X_{S, t_{j}}^{[i]} - {X}_{S, t_{j-1}}^{[i]} 
- V_{S, 0} (\alpha_{S}, {X}_{t_{j-1}}^{[i]}) \Delta_n 
- \mathscr{L}_0 [ V_{S, 0} (\alpha_S, \cdot))] (Z_{j-1}^{[i], \theta}) \frac{\Delta_n^2}{2}}{\sqrt{\Delta_n^3}} \\[0.5cm] 
\tfrac{X_{R, t_{j}}^{[i]} - {X}_{R, t_{j-1}}^{[i]} 
- V_{R, 0} (Z_{j-1}^{[i], \theta}) \Delta_n}{\sqrt{\Delta_n}} 
\end{bmatrix}, 
\end{align}
%
%
where we divide smooth/rough components by the size of the corresponding noise 
w.r.t.~$\Delta_n$. 
%
We have that 
$
\mathbf{m}_{j-1}^{[i], \theta} \, \bigl| \, \mathbb{X}_{t_{j-1}}^N
\sim \mathscr{N} \bigl( \mathbf{0}_d, \Sigma(Z_{j-1}^{[i], \theta})
\bigr)
$
with $\Sigma : \Theta \times \mathbb{R}^d \times \textcolor{black}{\mathcal{P}_2 (\mathbb{R}^d)} \to \mathbb{R}^{d \times d}$ defined as: 
\begin{align} 
& \Sigma (z) 
\equiv 
\begin{bmatrix}
\Sigma_{SS} (z) & \Sigma_{SR} (z) \\[0.2cm]  
\Sigma_{RS} (z) & \Sigma_{RR} (z) 
\end{bmatrix}
:= 
\begin{bmatrix}
\tfrac{1}{3} a_{S} (z) 
& \tfrac{1}{2} \partial_{x_R}^\top 
{V}_{S, 0} (\alpha_S, x) a_R (z) 
\\[0.2cm]  
\tfrac{1}{2} \bigl( \partial_{x_R}^\top {V}_{S, 0} (\alpha_S, x) a_R (z) \bigr)^\top 
& a_{R} (z) 
\end{bmatrix} \label{eq:Sigma}
\end{align}
for $z = (\theta, x, \mu) \in  \Theta \times \mathbb{R}^d \times \mathcal{P}_2 (\mathbb{R}^d)$.
\begin{lemma} \label{lemma:inv_Sigma}
Under Assumption \ref{ass:hor} it holds that: 
\begin{align*}
\inf_{(\theta, x, \mu) \in \Theta \times \mathbb{R}^d \times \mathcal{P}_2(\mathbb{R}^d)}  
\det \Sigma (\theta, x, \mu)  > 0,
\end{align*}
i.e.~$\Sigma$ is uniformly positive definite. Thus, the law of $\mathbf{m}_{j-1}^{[i], \theta}$ conditionally on $\mathbb{X}_{t_{j-1}}^N$ admits a non-degenerate Gaussian density for all $1 \le i \le N$, $1 \le j \le n$, $\theta \in \Theta$. 
\end{lemma}
\begin{proof} We have that $\det \Sigma (z) = \textcolor{black}{\left(\tfrac{1}{12}\right)^{d_S}} \det a_S (z) \det  a_R (z)$, which is positive uniformly in $z = (\theta, x, \mu)$ under Assumption \ref{ass:hor},  thus we conclude. 
\end{proof}
%
%
\noindent Thus,  (\ref{eq:LG-1}) leads to the following non-degenerate approximate transition density $\bar{p}_{\Delta_n}^N$ for $X_{t_{j}}^{[i]}$ given $X_{t_{j-1}}^{[i]}$: 
%
%
\begin{align}
\label{eq:lg_density}
\bar{p}_{\Delta_n}^{N} ({X}_{t_{j-1}}^{[i]}, {X}_{t_{j}}^{[i]}; \mu_{t_{j-1}}^N, \theta)
\equiv  
\tfrac{1}{\sqrt{(2\pi)^d \Delta_n^{3 d_S + d_R} 
\det \Sigma (Z_{j-1}^{[i], \theta})}} e^{ - \frac{1}{2} 
( \mathbf{m}_{j-1}^{[i], \theta} )^\top 
\Sigma ( Z_{j-1}^{[i], \theta} )^{-1}  
\mathbf{m}_{j-1}^{[i], \theta} }.   
\end{align}
for $1 \le i \le N, \, 1 \le j \le n$, $\theta \in \Theta$.
The contrast function for (\ref{eq:ips-1}) is then defined as follows: 
\begin{align}
\label{eq:contrast}
\ell_{n, N} (\theta)
:= \sum_{1 \le i \le N} 
\sum_{1 \le j \le n}
\bigl\{ 
\bigl( \mathbf{m}_{j-1}^{[i], \theta}  \bigr)^\top 
\Sigma ( Z_{j-1}^{[i], \theta} )^{-1}  
\mathbf{m}_{j-1}^{[i], \theta}  
+  \log \det \Sigma (Z_{j-1}^{[i], \theta})
\bigr\}, \ \theta \in \Theta. 
\end{align} 
\begin{remark}
For the case $d_S = 0$, i.e.~an interacting elliptic SDE with parameter $\theta = (\alpha_R, \beta)$, the contrast function (\ref{eq:contrast}) is expressed with $\mathbf{m}_{j-1}^{[i], \theta}$ and $\Sigma$ being replaced by $\mathbf{m}_{R, j-1}^{[i], \theta}$ and $a_R$, respectively. This is precisely a multivariate extension of the contrast function for IPSs of scalar elliptic SDEs studied in \cite{amo:23}.  
\end{remark}
\begin{remark}
It is critical to include the $\mathcal{O} (\Delta_n^2)$-term for the smooth component under the hypoelliptic setting, i.e.~$d_S \ge 1$, in the definition of the discretisation (\ref{eq:LG-1}). In the case of ergodic (non-interacting) hypoelliptic SDEs, \citep{poke:09} found that the contrast estimator based on a likelihood that omits the $\mathcal{O} (\Delta_n^2)$-term in the smooth component suffers from bias in the estimation of the parameter $\alpha_R$ in the drift of the rough parameter. 
\end{remark}
\subsection{Asymptotic Results}
For the model class (\ref{eq:ips-1}), we define the contrast estimator as:
\begin{align} \label{eq:estimator}
\hat{\theta}^{n, N} = \mathrm{argmin}_{\theta \in \Theta} \ell_{n, N} (\theta)
=
\begin{cases}
(\hat{\alpha}_{R}^{n ,N}, \hat{\beta}^{n, N}),  & d_S = 0;  \\[0.2cm]  
(\hat{\alpha}_S^{n ,N}, \hat{\alpha}_{R}^{n ,N}, \hat{\beta}^{n, N}), & d_S \ge 1. 
\end{cases}
\end{align} 
\textcolor{black}{We also write $\mu_t$ as the distribution of the McKean-Vlasov SDE (\ref{eq:mv_I}) under the true parameter $\theta^\dagger$.} To state our main result, we introduce some additional conditions.  
\begin{assump} \label{ass:initial_dist} 
For any $m \ge 1$, there is a constant $C > 0$ such that $\mathbb{E}_{\mu_0} |W|^m < C$. 
\end{assump}
\begin{assump} \label{ass:growth}
\begin{enumerate}
\item[(a)] \emph{(Globally \textcolor{black}{Lipschitz} Condition on the Diffusion Coefficients.)} 
There exists $C > 0$ such that for all $\beta \in \Theta_\beta$, all $x, y \in \mathbb{R}^d$ and all $\mu, \nu \in \mathcal{P}_2 (\mathbb{R}^d)$: 
\begin{align*}
\sup_{1 \le j \le d_B} 
| V_{R, j} (\beta, x, \mu)  - V_{R, j} (\beta, y, \nu) |^2 
\le C \bigl( | x- y |^2 + \mathcal{W}_2 (\mu, \nu)^2  \bigr). 
\end{align*}
\item[(b)] \emph{(Locally Lipschitz Condition on the Drift Function.)} There exist constants $C_1, C_2, C_3 > 0$ and an integer $q > 1$ such that for all $\alpha \in \Theta_\alpha$, all $x, y \in \mathbb{R}^d$ and all $\mu, \nu \in \mathcal{P}_2 (\mathbb{R}^d)$:
\begin{align*}
\langle x- y, V_0 (\alpha, x, \mu) - V_0 (\alpha, y, \mu) \rangle & \le C_1 | x - y |^2; \\[0.2cm]  
| V_0 (\alpha, x, \mu) - V_0 (\alpha, x, \nu) |^2 & \le 
C_2 \mathcal{W}_2 (\mu, \nu)^2; \\[0.2cm]
| V_{0} (\alpha, x, \mu) - V_0 (\alpha, y, \mu) | 
& \le C_3 |x - y| \bigl( 1 + |x|^q + |y|^q \bigr). 
\end{align*}
\end{enumerate} 
\end{assump}
\begin{assump}  
\label{ass:test_function}
We have that 
$V_{S, 0} \in C_p^6 (\Theta \times  \mathbb{R}^d, \mathbb{R}^{d_S})$, 
$V_{R, m}^{I} \in C_p^4 (\Theta \times \mathbb{R}^d, \mathbb{R}^{d_R})$ and  $V_{R, j}^{II} \in C_p^4 (\Theta \times \mathbb{R}^{2d}, \mathbb{R}^{d_R})$, $j = 0, 1, \ldots, d_B$. Furthermore, there exist constants $C > 0$ and $q \in \{0, 1, \ldots \}$ such that for all $x, y \in \mathbb{R}^d$, all $\alpha_S \in \Theta_{\alpha_S}$ and all $\gamma \in \{1, \ldots, d \}^{\iota}$, $\iota = 1, 2$:
\begin{align*}
| \partial_\gamma^x V_{S, 0} (\alpha_S, x) - \partial_\gamma^y V_{S, 0} (\alpha_S, y) | \le C |x- y| (1 + |x|^{q} + |y|^q ).
\end{align*} 
%
\end{assump}
\begin{assump} 
\label{ass:reg_coeff}
For all $(x, \mu) \in \mathbb{R}^d \times \mathcal{P}_2 (\mathbb{R}^d)$ and all multi-indices $\lambda \in \{1, \ldots, d\}^\iota$, 
$\iota = 0, 1, 2$, the mappings $\partial_\lambda^x V_{S, 0} (\cdot, x, \mu): \Theta \to \mathbb{R}^{d_S}$ and $V_{R, j}  (\cdot, x, \mu): \Theta \to \mathbb{R}^{d_R}$, $0 \le j \le d_B$, are three times continuously differentiable.  Furthermore, it holds:
\begin{enumerate}
\item[(a)] There exist constants $C_1, C_2 > 0$, $q_1, q_2, q_3 \in\{ 0, 1,\ldots\}$ such that for all multi-indices $\gamma \in \{1, \ldots,  d_\theta \}^{\ell}$, $\ell = 1, 2$, $\lambda \in \{1, \ldots, d\}^{\iota}$, $\iota = 0, 1, 2$,  all $(x, \mu), (y, \nu) \in \mathbb{R}^d \times \mathcal{P}_2 (\mathbb{R}^d)$ and all $\theta \in \Theta$:  
\begin{align*} 
&\bigl|  \partial_{\gamma}^\theta \partial_\lambda^x V_{S, 0} (\theta, x) 
-  \partial_{\gamma}^\theta \partial_\lambda^x V_{S, 0} (\theta, y) \bigr|
\le C_1 \bigl( 1 + |x|^{q_1} + |y|^{q_1} \bigr) |x-y|;  \\[0.3cm] 
& \sum_{0 \le j \le d_B}
\bigl| \partial_{\gamma}^\theta V_{R, j} (\theta, x, \mu) - \partial_{\gamma}^\theta V_{R, j} (\theta, y, \nu)  \bigr| 
\\[-0.2cm] 
& \qquad\,\,\,\,\,\,\,\,\, 
\le C_2 \bigl( | x - y| + \mathcal{W}_2 (\mu, \nu) \bigr) 
\bigl( 1 + |x|^{q_2} + |y|^{q_2} + \mathcal{W}_2(\mu,\delta_0)^{q_3}
+ \mathcal{W}_2(\nu,\delta_0)^{q_3} \bigr).  
\end{align*} 
\item[(b)] 
There exist constants $C_1, C_2 > 0$, $q_1, q_2, q_3 \in \{0, 1, \ldots\}$ such that for all multi-indices  $\gamma \in \{1, \ldots, d_{\theta} \}^{3}$, $\lambda \in \{1, \ldots, d\}^{\iota}$, 
$\iota = 0, 1, 2$, and all $(\theta, x, \mu) \in \Theta \times \mathbb{R}^d \times \mathcal{P}_2 (\mathbb{R}^d)$: 
\begin{gather*} 
\bigl| \partial_{\gamma}^{\theta} \partial_{\lambda}^x V_{S,  0} (\theta, x) \bigr| \le  C_1 ( 1 + |x|^{q_1}); \\   
\sum_{0 \le j \le d_B}
\bigl| \partial_{\gamma}^\theta V_{R, j} (\theta, x, \mu) \bigr| \le  C_2 ( 1 + |x|^{q_2} + \mathcal{W}_2 (\mu, \delta_0)^{q_3}). 
\end{gather*}  
\end{enumerate}

\end{assump}
\begin{assump} \label{ass:ident}
If it holds that:
\begin{gather*}
V_{S, 0} (\alpha_S, x) = V_{S, 0} (\alpha_S^\dagger, x), \quad 
V_{R, 0} (\alpha_R, x, \mu_t) = V_{R, 0} (\alpha_R^\dagger, x, \mu_t); \\  
\textcolor{black}{a_R (\beta, x, \mu_t) = a_R (\beta^\dagger, x, \mu_t)}, 
\end{gather*}
$\mu_t (dx) \times dt \ \ \mathrm{a.e.}, \ \textrm{for all } t \in [0, T]$,
then $\alpha_S = \alpha_S^\dagger$, $\alpha_R = \alpha_R^\dagger$, $\beta = \beta^\dagger$ holds respectively.  
\end{assump}
Assumption \ref{ass:ident} is related to the identifiability of the statistical parameters. A similar condition is assumed in \cite{DeH23} where the authors studied the local asymptotic normality (LAN) for an $N$ weakly interacting elliptic diffusions under the continuous observation regime and discussed identifiability and non-degeneracy of asymptotic Fisher information. 
Assumption~\ref{ass:ident} is used in proving the consistency of the contrast estimator in Section \ref{sec:pf_consistency} and also implies the uniqueness of the minimiser for the objective function \eqref{eq:contrast} for large enough 
$n, N$; see Remark~\ref{rem:unique}.   

Given the above assumptions, we state our main results as follows.
\begin{theorem}[Consistency]  \label{thm:consistency}
Under Assumptions \ref{ass:law_dep}--\ref{ass:ident}, we have that, as $n, N \to 
\infty$, then $\hat{\theta}^{n, N} \probconv \trueparam$.
\end{theorem}
\begin{theorem}[Asymptotic Normality]
\label{thm:clt}
Let Assumptions \ref{ass:law_dep}--\ref{ass:ident} hold. 
\begin{enumerate}[leftmargin=0.4cm]
\item[(a)] For $d_S = 0$, i.e.~ an interacting elliptic diffusion  with $\theta = (\alpha_R, \beta)$, it holds that:  
\begin{align*}
\begin{bmatrix}
\sqrt{N} (\hat{\alpha}_{R}^{n, N} - \alpha_{R}^\dagger ) \\[0.2cm] 
\sqrt{\tfrac{N}{\Delta_n}} (\hat{\beta}^{n, N} - \beta^\dagger) 
\end{bmatrix}
\distconv 
\mathscr{N} \Bigl(\mathbf{0}_{d_\theta},  
\mathrm{diag} \bigl[ \Gamma^{(\mathrm{E})}_{\alpha_R} (\trueparam)^{-1}, \Gamma^{(\mathrm{E})}_\beta (\trueparam)^{-1} \bigr]  \Bigr),  
\end{align*} 
%
%
as $n, N \to \infty,  \, N \Delta_n \to 0$,  where for indices $1 \le k_1, k_2 \le d_{\alpha_R}$, $1 \le k_3, k_4 \le d_{\beta}$, we have: 
\begin{align*}
\bigl[ \Gamma^{(\mathrm{E})}_{\alpha_R} (\trueparam) \bigr]_{k_1 k_2} 
& = \int_0^T \mathbb{E}_{\mu_t} \Bigl[
\bigl( \partial_{\alpha_R, k_1} V_{R , 0}^\top \cdot  
\Sigma_{RR}^{-1}  \cdot
\partial_{\alpha_R, k_2} V_{R , 0}
\bigr) (\trueparam, W, \mu_t)  \Bigr]  dt;  \\[0.1cm] 
\bigl[ \Gamma^{(\mathrm{E})}_{\beta} (\trueparam) \bigr]_{k_3 k_4} 
& = \tfrac{1}{2}
\int_0^T 
\mathbb{E}_{\mu_t} \Bigl[  
\mathrm{tr} \bigl\{ \bigl( \partial_{\beta, k_3} \Sigma_{RR}  \cdot  \Sigma_{RR}^{-1} \cdot 
\partial_{\beta, k_4} \Sigma_{RR} \cdot  
\Sigma_{RR}^{-1} \bigr)  (\trueparam, W, \mu_t) 
\bigr\}  \Bigr]  dt. 
\end{align*}
%
\item[(b)] For  $d_S \ge 1$, i.e.~ an interacting hypoelliptic diffusion with $\theta = (\alpha_S, \alpha_R, \beta)$, it holds that: 
\begin{align*}
\begin{bmatrix}
\sqrt{\tfrac{N}{\Delta_n^2}} (\hat{\alpha}_S^{n, N} - \alpha_S^\dagger ) \\[0.3cm]
\sqrt{N} (\hat{\alpha}_{R}^{n, N} - \alpha_{R}^\dagger ) \\[0.2cm] 
\sqrt{\tfrac{N}{\Delta_n}} (\hat{\beta}^{n, N} - \beta^\dagger) 
\end{bmatrix}
\distconv 
\mathscr{N} \Bigl(\mathbf{0}_{d_\theta},  
\mathrm{diag} 
\bigl[ 
\Gamma^{(\mathrm{H})}_{\alpha_S} (\trueparam)^{-1},  
\Gamma^{(\mathrm{H})}_{\alpha_R} (\trueparam)^{-1}, 
\Gamma^{(\mathrm{H})}_\beta (\trueparam)^{-1} \bigr]   \Bigr),  
\end{align*}
as $n, N \to \infty,  \, N \Delta_n \to 0$, where for indices $1 \le k_1, k_2 \le d_{\alpha_S}$, $1 \le k_3, k_4 \le d_{\alpha_R}$, $1 \le k_5, k_6 \le d_{\beta}$, we have:    
\begin{align}
\bigl[ \Gamma^{(\mathrm{H})}_{\alpha_S} (\trueparam) \bigr]_{k_1 k_2} 
& = 4  \int_0^T \mathbb{E}_{\mu_t} \Bigl[
\bigl(
\partial_{\alpha_S, k_1} V_{S , 0}^\top \cdot  \Sigma_{SS}^{-1}  \cdot  
\partial_{\alpha_S, k_2} V_{S, 0}
\bigr)  (\trueparam, W, \mu_t)  \Bigr]  dt; \label{eq:G_alpha_S} 
\\[0.1cm]  
\bigl[ \Gamma^{(\mathrm{H})}_{\alpha_R} (\trueparam) \bigr]_{k_3 k_4} 
& = \int_0^T \mathbb{E}_{\mu_t} 
\Bigl[
\bigl( \partial_{\alpha_R, k_3} V_{R , 0}^\top \cdot  \Sigma_{RR}^{-1}  \cdot  
\partial_{\alpha_R, k_4} V_{R , 0} 
\bigr)  (\trueparam, W, \mu_t)  \Bigr]  dt; \label{eq:G_alpha_R}   \\[0.1cm] 
\bigl[ \Gamma^{(\mathrm{H})}_{\beta} (\trueparam) \bigr]_{k_5 k_6} 
& = \tfrac{1}{2}
\int_0^T 
\mathbb{E}_{\mu_t} 
\Bigl[  \mathrm{tr} \bigl\{ 
\bigl( \partial_{\beta, k_5} \Sigma \cdot \Sigma^{-1} \cdot \partial_{\beta, k_6} \Sigma  \cdot \Sigma^{-1} 
\bigr) (\trueparam, W, \mu_t) 
\bigr\}   \Bigr] dt. \label{eq:G_beta} 
\end{align} 
%
\end{enumerate}
\end{theorem}
%
%
%
\begin{remark} \label{rem:diff_param}
We compare our asymptotic normality result, Theorem~\ref{thm:clt}, with that obtained in~\cite{amo:24} and highlight the advantage of our estimator. In~\cite{amo:24} a hypoelliptic/kinetic system of interacting diffusions in a scalar space (i.e.~the state space is $\mathbb{R}^{2N}$) of the following form is considered: 
\begin{align*}
\begin{aligned}
d X_{S, t}^{[i]} & = X_{R, t}^{[i]} dt; \\ 
d X_{R, t}^{[i]} & = V_{R, 0} (\alpha, X_{t}^{[i]}, \mu_t^N) dt + V_{R, 1} (\beta, X_t^{[i]}, \mu_t^N) d B_{1, t}^{[i]},   
\end{aligned} \qquad 1 \le i \le N.
\end{align*} 
Here $(X_{S, t}^{[i]}, X_{R, t}^{[i]}) \in \mathbb{R}^2$, where $X_{R, t}^{[i]}$ is driven by a scalar Brownian motion $B_{1, t}^{[i]}$. In the complete observation regime, \cite{amo:24} proposes a contrast function based on the Euler-Maruyama approximation for the rough component $X_{R, t}^{[i]}$, and 
then shows that the proposed estimator minimising the contrast function,  $\hat{\theta}^{n, N}_{\mathrm{EM}} = (\hat{\alpha}_{\mathrm{EM}}^{n, N}, \hat{\beta}_{\mathrm{EM}}^{n, N})$, satisfies the following CLT:  
\begin{align*}
\begin{bmatrix}
\sqrt{N} (\hat{\alpha}_{\mathrm{EM}}^{n, N} - \alpha^\dagger ) \\[0.2cm] 
\sqrt{\tfrac{N}{\Delta_n}} (\hat{\beta}_{\mathrm{EM}}^{n, N} - \beta^\dagger) 
\end{bmatrix} \distconv 
\mathscr{N} \Bigl(\mathbf{0}_{d_\theta},  
\mathrm{diag} \bigl[ 
\Gamma_{\alpha} (\trueparam)^{-1}, 
\Gamma_{\beta} (\trueparam)^{-1} \bigr]
\Bigr),  
\end{align*}  
as $n, N \to \infty,  \, N \Delta_n \to 0$,  where  for $1\le k_1, k_2 \le d_\alpha$, $1 \le k_3, k_4 \le d_\beta$: 
\begin{align}
\bigl[ \Gamma_{\alpha} (\trueparam) \bigr]_{k_1 k_2} & = 
\int_0^T \mathbb{E}_{\mu_t} \Bigl[
\Bigl( \tfrac{ \partial_{\alpha, k_1} V_{R, 0}    \times \partial_{\alpha, k_2} V_{R, 0} }{\Sigma_{RR}} 
\Bigr) (\theta^\dagger, W, \mu_t) \Bigr] dt; \label{eq:em_alpha} \\ 
\bigl[ \Gamma_{\beta} (\trueparam)  \bigr]_{k_3 k_4} 
& = \tfrac{1}{2} \int_0^T \mathbb{E}_{\mu_t} 
\Bigl[
\Bigl(  \tfrac{\partial_{\beta, k_3} \Sigma_{RR} \times  
\partial_{\beta, k_4} \Sigma_{RR}}{\Sigma_{RR}^2}
\Bigr) (\trueparam, W, \mu_t)
\Bigr] dt,
\label{eq:em_beta} 
\end{align} 
with a scalar valued function $\Sigma_{RR} = V_{R, 1}^2$. 
We note that the asymptotic precision matrix (\ref{eq:G_alpha_R}) for $\alpha_R$ based on the locally non-degenerate Gaussian approximation coincides with (\ref{eq:em_alpha}) when $d\equiv 2$. An important difference appears in $\Gamma_\beta(\trueparam)$,  i.e.~the matrix associated with the diffusion parameter. In particular, the matrix (\ref{eq:G_beta}) involves the covariances across all coordinates, while (\ref{eq:em_beta}) contains only the covariances within the rough component, and such a structural difference can lead to smaller asymptotic variances for our estimator $\hat{\beta}^{n ,N}$. To observe this, we consider the interacting underdamped Langevin equation defined in (\ref{eq:I-UL}) with its diffusion coefficient defined by $V_{R, 1} \equiv \sigma := \sqrt{2 \gamma \beta^{-1}}$, where $\gamma, \beta > 0$ and thus the true parameter $\sigma^\dagger$ assumed to be positive. 
%
%
%
In this case we have that:
\begin{align*}
\Sigma = \sigma^2
\begin{bmatrix}
1/3 & \ 1/2 \\[0.2cm]
1/2 & \ 1 
\end{bmatrix}, \qquad 
\Lambda = \sigma^{-2}
\begin{bmatrix}
12 & - 6  \\[0.2cm]
- 6 & 4 
\end{bmatrix} 
\end{align*}
and then: 
\begin{align*}
\Gamma_{\sigma}^{(\mathrm{H})} (\trueparam) = 
\tfrac{4 T}{(\sigma^{\dagger})^2}. 
\end{align*}
%
%
On the other hand, the precision matrix $\Gamma_\sigma$ derived in \cite{amo:24} gives rise to:
\begin{align*}
\Gamma_\sigma (\trueparam)  = \tfrac{2T}{(\sigma^{\dagger})^2}, 
\end{align*} 
thus, $\Gamma_{\sigma}^{(\mathrm{H})} (\trueparam)^{-1} = \tfrac{1}{2} \Gamma_{\sigma} (\trueparam)^{-1}$ for this model example. 
\end{remark}

\begin{remark} 
We note that in the above statement invertibility of matrices $\Gamma^{(\mathrm{E})}$ and $\Gamma^{(\mathrm{H})}$ is implicitly assumed. Checking such invertibility in practice is in general difficult. However, for the submatrix $(\ref{eq:G_beta})$ one can observe its invertibility for the class of interacting underdamped Langevin type dynamics in (\ref{eq:I-UL}). For instance, let us consider (\ref{eq:ips-1}) as specified for $d_S = d_R = d_B$ and with: 
\begin{align*} 
V_{S, 0} (\alpha_S, x) = x_R,  \quad 
V_{R, j} (\beta, x, \mu) = \beta_j e_j, \quad j = 1, \ldots, d_B, 
\end{align*}  
with $\beta = (\beta_1, \ldots, \beta_{d_B}) \in \Theta_\beta , \, d_\beta = d_B$, for $x = (x_S, x_R) \in \mathbb{R}^d$, $\alpha_S \in \Theta _{\alpha_S}$, $\mu \in \mathcal{P}_2 (\mathbb{R}^d)$, where $e_j$ is the standard unit vector. Then, we have:
\begin{align*}
\Gamma^{(\mathrm{H})}_\beta (\trueparam) = 4T \times \mathrm{diag} 
\bigl[ (\beta_1^\dagger)^{-2}, \ldots, (\beta_{d_B}^\dagger)^{-2}  \bigr] \in \mathbb{R}^{d_B \times d_B},  
\end{align*}
which is invertible as long as $\beta^\dagger_j \neq 0$, $j = 1, \ldots, d_B$. For the precision matrices \eqref{eq:G_alpha_S}, \eqref{eq:G_alpha_R}, associated with the drift parameter $\alpha = (\alpha_S, \alpha_R)$, one may be able to derive a sufficient condition for these to be positive definite under Assumption \ref{ass:hor}, similarly to Proposition 15 in \cite{DeH23}. Here, let us focus on (\ref{eq:G_alpha_S}) and assume for simplicity that $d_{\alpha_S} = 1$. Then, noticing that $\Sigma_{SS}^{-1}$ is positive definite, we have that $\partial_{\alpha_S} V_{S, 0}^\top (\alpha_S^\dagger, x) \cdot \, \Sigma_{SS}^{-1} (\trueparam, x, \mu) \, \cdot \partial_{\alpha_S} V_{S, 0} (\alpha_S^\dagger, x) \ge 0$ for all $(x, \theta) \in \mathbb{R}^d \times \mathcal{P}_2(\mathbb{R}^d)$ where $0$ is attained iff $\partial_{\alpha_S} V_{S, 0} (\alpha_S^\dagger, x) = \mathbf{0}_{d_S}$. This implies that $\Gamma_{\alpha_R}^{(\mathrm{H})} (\trueparam) > 0 $ if there exists $(A, B) \in \mathcal{B} ([0, T]) \otimes \mathcal{B} (\mathbb{R}^d)$ with $\mathrm{Leb} (A) > 0$ s.t. for all $t \in A$, we have $\mu_t (B) > 0$ and $\partial_{\alpha_S}  V_{S, 0} (\alpha_S^\dagger, x) \neq \mathbf{0}_{d_S}$, $x \in B$. Similar arguments can be applied for the case of $\alpha_R$. Further exploration under more general settings is left for future research.  
\end{remark} 

\textcolor{black}{\begin{remark} \label{rem:local_lip}
Theorems \ref{thm:consistency} and \ref{thm:clt} stated that the statistical estimator constructed from the approximate one-step transition density based on the scheme (\ref{eq:LG-1}) attained desirable asymptotic properties, particularly under the setting of super-linear drift specified in Assumption \ref{ass:growth}. From the numerical analysis perspective, it is well known that the  Euler-Maruyama (EM) scheme can be an unstable simulation method for SDEs with super-linear drift, see e.g., \cite{matt:02}. 
Thus, the scheme (\ref{eq:LG-1}), which includes EM in the approximation of $X_{R, t}^{[i]}$, is not intended to be used as a stable numerical solver.
However, from a statistical point of view, the use of the LG-based contrast function is justified because the statistical estimation problem relies solely on observations, i.e., the true dynamics of (\ref{eq:ips-1}), rather than on the simulated trajectories. 
Indeed, the proofs of the main results investigate the asymptotic behavior of (appropriately scaled) contrast function (\ref{eq:contrast}) given well-posed observations $\{X_{t_j}^{[i]}\}$ with finite moments, which is ensured by the introduced assumptions. 
Details can be found in the proofs in Section \ref{sec:pfs}. We thus stress that (\ref{eq:LG-1}) is introduced for statistical purposes to form a tractable transition density to circumvent the degeneracy of the diffusion matrix and not for the simulation method of SDEs. 
\end{remark}}

\subsection{Parameter Estimation for Partially Observed Hypoelliptic IPS} \label{sec:partial_obs}
An important point is that the development of the non-degenerate, locally Gaussian approximation (\ref{eq:lg_density}) enables one to conduct parameter estimation for a broad class of IPSs arising in applications beyond the restricted class of kinetic IPSs studied in \cite{amo:24}.
We highlight this by considering the \emph{partial observation regime}, i.e.~the case where only a subset of coordinates of each particle's state $X_{t_j}^{[i]}$ 
are observed.
The non-degenerate transition density (\ref{eq:lg_density}) can be used within broader algorithms for parameter inference, either in a Monte-Carlo setting (e.g., via MCMC, Sequential Monte-Carlo approaches and combinations of thereof) or in an analytical one (via Kalman Filter, as in the discussion that follows).
For instance, one can aim to recover the \emph{marginal likelihood} by integrating out hidden coordinates from the well-defined full likelihood. In the case of standard (non-interacting) hypoelliptic SDEs, optimizing such a marginal likelihood deduced via the non-degenerate transition density is shown numerically to perform very effectively for parameter estimation; see e.g.~\cite{dit:19, igu:ejs, igu:24}. This approach can be naturally extended to the weakly interacting case (\ref{eq:ips-1}) when an approximate marginal likelihood is constructed from the full transition density (\ref{eq:lg_density}). 
We focus on a common scenario in applications where only the smooth components $\{ \mathbb{X}_{S, t}^N \}_{t \in \mathbb{T}_n}$ are observed. For the state $X_t^{[i]} = [X_{S, t}^{[i], \top}, X_{R, t}^{[i], \top}]^\top$ with the hidden (rough) component $X_{R, t}^{[i]}$, the (log) marginal approximate likelihood is recovered as follows:
\textcolor{black}
{\begin{align}
\begin{aligned} \label{eq:partial_original}
& \ell_{n, N}^P (\theta)  \\ 
& \equiv 
\log \Bigl(\int 
\Bigl\{ \prod_{i = 1}^N \prod_{j = 1}^n \bar{p}_{\Delta_n}^{N} (X_{t_{j-1}}^{[i]}, X_{t_{j}}^{[i]}; \mu_{t_{j-1}}^N, \theta) 
\Bigr\} \prod_{i = 1}^N p (X_{t_0}^{[i]}) \, 
d \mathbb{X}_{R, t_0}^{N} \cdots d \mathbb{X}_{R, t_n}^N 
\Bigr), 
\end{aligned} 
\end{align}}
where $p (X_{t_0}^{[i]})$ is the density of initial state, $\bar{p}_{\Delta_n}^N$ is the locally Gaussian density in (\ref{eq:lg_density}) and \textcolor{black}{$d \mathbb{X}_{R, t_j}^{N} \equiv d X_{R, t_j}^{[1]} d X_{R, t_j}^{[2]} \cdots d X_{R, t_j}^{[N]}, \, j = 0, 1, \ldots, n$}.  
Delivering an asymptotic analysis for the marginal likelihood is beyond the scope of this paper. 
Note that such an analysis for the non-degenerate marginal likelihood is missing in the literature even for the class of standard SDEs (linear in the sense of McKean).  
%

\begin{remark} \label{rem:CGN-IPS}
For a rich subclass of weakly interacting elliptic/hypoelliptic diffusions, we can calculate analytically the marginal likelihood (\ref{eq:partial_original}) via the Kalman filter. To observe this, we again consider the case where only the smooth particle trajectories $\{ \mathbb{X}_{S, t}^N \}_{t \in \mathbb{T}_n}$ are observed and assume that the locally Gaussian approximation (\ref{eq:LG-1}) applied to an IPS is expressed as a linear stochastic system in the hidden component $X_{R,t}^{[i]}$ given the observation $\{ \mathbb{X}_{S, t}^N \}_{t \in \mathbb{T}_n}$, i.e.~for $0 \le j \le n-1$ and $1 \le i \le N$: 
\begin{align*}
\begin{bmatrix}
{X}_{S, t_{j+1}}^{[i]} \\[0.1cm]
{X}_{R, t_{j+1}}^{[i]}
\end{bmatrix}
& = A (\Delta_n, {X}_{S, t_{j}}^{[i]}, \mu_{S, t_j}^{N}, \theta) 
+ B (\Delta_n, X_{S, t_j}^{[i]}, \mu_{S, t_j}^{N}, \theta)  X_{R, t_j}^{[i]}  \\ 
& \qquad 
+ C (\Delta_n, {X}_{S, t_{j}}^{[i]}, \mu_{S, t_j}^{N}, \theta) 
\begin{bmatrix}
\textstyle \int_{t_j}^{t_{j+1}} \int_{t_j}^{u } dB_{s}^{[i]} du \\[0.2cm]
B_{t_{j+1}}^{[i]} - B_{t_{j}}^{[i]}
\end{bmatrix}
\end{align*}
for $A: [0, \infty) \times \mathbb{R}^{d_S} \times \mathcal{P}_2 (\mathbb{R}^{d_S}) \times \Theta \to \mathbb{R}^{d}$, 
$B: [0, \infty) \times \mathbb{R}^{d_S} \times \mathcal{P}_2 (\mathbb{R}^{d_S}) \times \Theta \to \mathbb{R}^{d\times d_R}$
and $C: [0, \infty) \times \mathbb{R}^{d_S} \times \mathcal{P}_2 (\mathbb{R}^{d_S}) \times \Theta \to \mathbb{R}^{d\times 2 d_B}$, 
where $\mu_{S, t_j}^{N}$ is the empirical distribution of $X_{S ,t_j}$. Note that $A$, $B$, $C$ can be nonlinear in the state $X_{S,t}^{[i]}$. \textcolor{black}{Then, the approximate marginal likelihood is simplified as 
\begin{align} 
\ell_{n, N}^P (\theta) 
& = \log \left( \prod_{i = 1}^N \int 
\Bigl\{ \prod_{j = 1}^n \bar{p}_{\Delta_n}^{N} (X_{t_{j-1}}^{[i]}, X_{t_{j}}^{[i]}; \mu_{S, t_{j-1}}^N, \theta) 
\Bigr\} p (X_{t_0}^{[i]}) \, 
d {X}_{R, t_0}^{[i]} \cdots d {X}_{R, t_n}^{[i]} 
\right) \nonumber \\ 
& \equiv \sum_{i = 1}^N \log f_i \bigl(
\{ \mathbb{X}_{S, t}^N \}_{t \in \mathbb{T}_n}; \theta \bigr),  \label{eq:partial}
\end{align}
where, with a slight abuse of notation, we used the locally Gaussian density given $\mu_{S, t_j}^N$ and expressed (\ref{eq:partial_original}) as the product of particle-wise marginal likelihoods under the model structure.    
In particular, for each particle index $i$, $f_i \bigl(
\{ \mathbb{X}_{S, t}^N \}_{t \in \mathbb{T}_n}; \theta \bigr)$ is exactly computed via the standard Kalman recursion formula applied to the conditionally linear system given the data $\{\mathbb{X}_{S, t}^N\}_{t \in \mathbb{T}_n}$.} 
This model class encompasses a range of important model examples in practice, including interacting underdamped Langevin equations (\ref{eq:I-UL}) and the interacting FHN model (\ref{eq:I-FHN}). 
Conditionally Gaussian nonlinear SDEs without interaction have been studied in the literature, e.g.~\cite{igu:24, vr:22}, and the closed-form marginal likelihood for the weakly interacting case is derived from the same arguments as in those references by incorporating now the (observed) interaction dependency $\mu_S^N$. 

\begin{remark}
In Remark \ref{rem:CGN-IPS} we assumed that the smooth component $X_{S,t}$ is observed. This is the case that arises frequently in applications, i.e.~we observe the particle position but not its velocity. However, the same ideas apply for problems when the rough component is observed and one obtains a linear system in $X_{S,t}$ conditionally on $X_{R,t}$.
Thus, our non-degenerate approach gives greater flexibility in practice.  
In the next section, we provide numerical results for the interacting FitzHugh-Nagumo model, where only particles of rough coordinates $X_{R}$ are assumed to be observed.
\end{remark} 

%
\end{remark}

\section{Numerical Experiments}
\label{sec:numerics} 
We illustrate numerically the effectiveness of the proposed contrast function (\ref{eq:contrast}) for parameter estimation for models in class (\ref{eq:ips-1}). We examine two model examples referenced in Section \ref{sec:intro} and perform parameter inference under both a complete and  a partial observation regime. The codes that reproduce the results below are available at \url{https://github.com/YugaIgu/parameter_estimation_ips_hypo}.

\subsection{Interacting FitzHugh-Nagumo Model} 
\label{sec:i-fhn-numerics}
This IPS is defined in (\ref{eq:I-FHN}) and involves the parameter vector $\theta = (a, b, c, \kappa, \sigma)$, where we now have $\alpha_S \equiv (a, b, c)$, $\alpha_R \equiv \kappa$ and $\beta \equiv \sigma$. 
We note that the Euler-Maruyama (EM)-based estimator studied in \cite{amo:24} cannot be employed in this case as the drift function of the smooth coordinate $Y_t^{[i]}$ is not equal to the rough coordinate $X_t^{[i]}$, as is the case for the kinetic Langevin equation. 
\textcolor{black}{We assume that $c, \sigma$ are defined over compact sets in $(0, \infty)$. Then, Assumptions \ref{ass:law_dep} and \ref{ass:hor} are readily verified. The regularity conditions (Assumptions \ref{ass:growth}--\ref{ass:reg_coeff}) also hold. For instance, the first inequality in Assumption \ref{ass:growth} (b) follows from the fact that the function $f(x) = - \tfrac{1}{3}x^3$ satisfies $(x - y) (f(x) - f(y)) \le 0$ for all $x, y \in \mathbb{R}$. Moreover, for any $\mu, \nu \in \mathcal{P}_2 (\mathbb{R}^d)$ and coupling $\pi \in \Gamma (\mu, \nu)$, we apply Jensen's inequality to get 
\begin{align*}
\bigl| \mathbb{E}_\mu [X] - \mathbb{E}_\nu [X] \bigr| \le \int |x - y| \, \pi (dx, dy) 
\le \Bigl\{ \int |x - y|^2 \, m (dx, dy) \Bigr\}^{1/2}.  
\end{align*}
Taking the infinum over the coupling gives $\bigl| \mathbb{E}_\mu [X] - \mathbb{E}_\nu [X] \bigr|^2 \le W_2 (\mu, \nu)^2$, which yields the second inequality of Assumption \ref{ass:growth} (b). The remaining conditions in Assumptions \ref{ass:test_function} and \ref{ass:reg_coeff} follow directly from the polynomial form of the coefficients and their smooth dependence on the parameters. } 

In this experiment, we set the true parameter to $\theta^\dagger = (0.2, 0.8, 1.5, 2.0, 0.5)$ and generate $100$ independent  replicates of complete observations $D_{\mathrm{C}}^{n, N} = \{ X_{t_j}^{[i]}, Y_{t_j}^{[i]} \}_{0 \le j \le n}^{1 \le i \le N}$ and partial trajectories of rough particles $D_{\mathrm{P}}^{n, N} 
= \{ X_{t_j}^{[i]}\}_{0 \le j \le n}^{1 \le i \le N}$ by applying the EM scheme to (\ref{eq:I-FHN}) with a small step-size $0.0005$ up to the terminal time $T = 30$ and then subsampling to obtain datasets with $(N, n, \Delta_n) = (100, 6000, 0.005)$. We compute the estimators $\hat{\theta}^{n, N}_{\, \mathrm{C}}$ and $\hat{\theta}^{n, N}_{\, \mathrm{P}}$ by optimising the contrast function (\ref{eq:contrast}) given $D^{n, N}_{\mathrm{C}}$ and the marginal likelihood~(\ref{eq:partial}) given $D_{\mathrm{P}}^{n ,N}$, respectively, where we note that the marginal likelihood is tractable as described in  Remark~\ref{rem:CGN-IPS}. 
To find the estimates we used the adaptive moments optimiser (ADAM) with the following  specifications. For $\hat{\theta}^{n ,N}_{\mathrm{C}}$ we set (step-size) = $0.01$, 
(exponential decay rate for the first moment estimates) = $0.9$, (exponential decay rate for the second moment estimates) = $0.999$, (additive term for numerical stability) = $10^{-8}$ and (number of iterations) = $8,000$; for $\hat{\theta}^{n, N}_{\mathrm{P}}$, we made the  same choices as above except for (step-size) = $0.005$ and (number of iterations) = $5,000$. 

\begin{figure}[h]
\centering
\includegraphics[width=13.5cm]{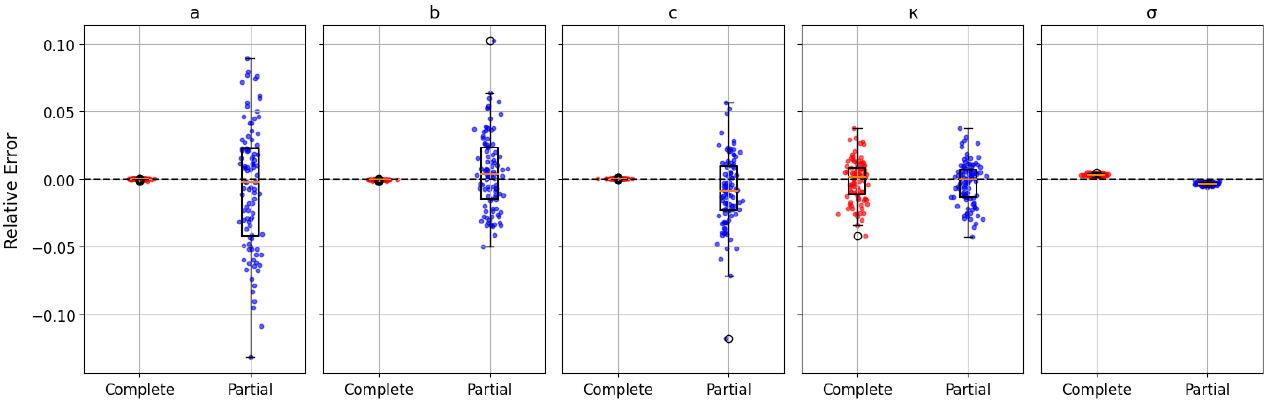}
\caption{Boxplots of $100$ realisations of estimators for the interacting FHN model (\ref{eq:I-FHN}), 
with parameters $(a,b,c,\kappa,\sigma)$. The left-side (resp.~right-side) in each panel corresponds to complete observations (resp.~partial ones).}
\label{fig:ips-fhn}
\end{figure} 

Fig.~\ref{fig:ips-fhn} shows boxplots of relative discrepancies $(\hat{\theta}^{n, N}_{\, \mathrm{w}, j} - \theta^{\dagger}_j)/\theta^{\dagger}_j,$, $1 \le j \le 5$, $w \in \{\mathrm{C}, \mathrm{P}\}$, from the $100$ replicate datasets.  
The individual realisations of $\hat{\theta}^{n, N}_{\mathrm{C}}$ and $\hat{\theta}^{n, N}_{\mathrm{P}}$ are indicated on the left and right sides of each frame, respectively. We observe that both $\hat{\theta}_{\mathrm{C}}^{n, N}$ and $\hat{\theta}_{\mathrm{P}}^{n, N}$ seem to correctly estimate the true values of all parameters. 
The  results for $\hat{\theta}_{\mathrm{C}}^{n, N}$ are in agreement with Theorem \ref{thm:clt} (CLT). That is, the drift parameters $\alpha_S = (a, b, c)$ in the smooth particles $X_{t_j}^{[i]}$ are estimated with the least uncertainty  as expected 
by the CLT rate $\sqrt{{N}/{\Delta_n^2}}$ for $\hat{\alpha}_S^{n, N}$. For partial observations, we note that the accuracy in estimating parameters ($\kappa, \sigma$) appearing in the dynamics of the observable components remains comparable to that in the case of full data. On the other hand, the uncertainty in estimating $\alpha_S = (a, b, c)$ increases, which can be expected by the lack of direct observation of $Y_{t}^{[i]}$. Still, the estimates for $\alpha_S$ are centered around the true values, which suggests that the locally Gaussian approximation \eqref{eq:LG-1} also provides asymptotically unbiased estimators in the partial observation scenario.

\subsection{Interacting Underdamped Langevin Equation} 
\label{sec:I-UL-num}
For the second example, we consider the interacting underdamped Langevin dynamics (\ref{eq:I-UL}) with $d' = 1$, for \textcolor{black}{a bistable potential function $q \mapsto V (q) = \lambda q^2 (q^2 - 1)$}, $\lambda > 0$, and the quadratic interacting function $q \mapsto U (q) = q^2/2$. To observe the performance of estimation for the drift and diffusion parameters separately, we set $\sigma =\textstyle \sqrt{2 \gamma \beta^{-1}}$ so that the drift and diffusion coefficients of the velocity $p_t^{[i]}$ do not share the friction coefficient parameter $\gamma$. This choice of parameters is the one that is consistent with statistical mechanics.  Thus, the parameter vector is $\theta = (\lambda, \gamma, \kappa, \sigma)$ with its true value set at $\theta^\dagger = (2.0, 1.5, 2.0, 0.5)$. 
We generate $100$ replicates of synthetic complete trajectories $D_{\mathrm{C}}^{n, N} = \{ q_{t_j}^{[i]}, p_{t_j}^{[i]}\}_{0 \le j \le n}^{1 \le i \le N}$ and partial trajectories of the smooth components $D_{\mathrm{P}}^{n, N} = \{ q_{t_j}^{[i]}\}_{0 \le j \le n}^{1 \le i \le N}$ by using the EM scheme for (\ref{eq:I-UL}) with a small step-size $0.0005$ up to time $T = 30$. We then consider the following two experimental designs: 
\textbf{Design I.} $(N, n, \Delta_n) = (50, 3000, 0.01)$ and \textbf{Design II.} $(N, n, \Delta_n) = (50, 6000, 0.005)$,  obtained by subsampling from the full trajectories. Similarly to the previous experiment, we compute maximum-likelihood type estimators $\hat{\theta}^{n, N}_{\mathrm{C}}$ and $\hat{\theta}^{n, N}_{\mathrm{P}}$ based on the locally Gaussian approximation (\ref{eq:LG-1}) from $D_{\mathrm{C}}^{n, N}$ and $D_{\mathrm{P}}^{n, N}$, respectively. 
In addition, we computed the EM-based estimators proposed in \cite{amo:24} with the same data for comparison. In particular, the main focus here is that, first, under the complete observation regime, we numerically confirm Remark~\ref{rem:diff_param}, 
i.e.~that the variance of the estimator of the diffusion parameter obtained by (\ref{eq:LG-1}) is half that of the corresponding estimator derived via the EM approximation, for IPSs with a constant diffusion coefficient. Second, under the partial observation regime, we compare the performance of two estimators in terms of bias and variance. We recall here that in the latter scenario, \cite{amo:24} uses the finite-difference approximation of $d q_t^{[i]} = p_t^{[i]} dt$ to recover the unobserved particles $p_{t_j}^{[i]}$, and then these $ p_{t_j}^{[i]}$'s, together with the data $\{ q_{t_j}^{[i]} \}$, are connected to a contrast function obtained via the EM scheme applied only to the hidden component $p_{t_j}^{[i]}$. To compute the estimates, we used the ADAM optimizer and adopted the same parameter specification presented in Section \ref{sec:i-fhn-numerics} except for (step size) = $0.01$ and (number of iterations) = $5,000$.

\begin{figure}[h]
\centering
\includegraphics[width=7.5cm]{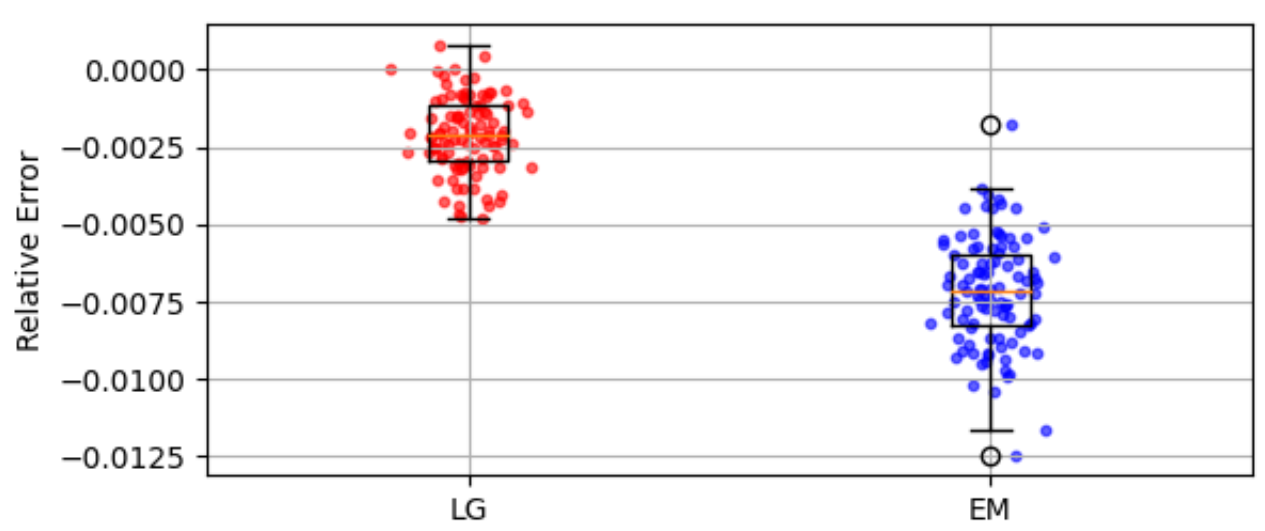}
\caption{Boxplots of $100$ realisations of estimators of $\sigma$ (for the LG- and EM-based methods, on the left and right, resp.) for the interacting underdamped Langevin equation (\ref{eq:I-UL}) under the complete observation regime, for Design I.}
\label{fig:sigma}
\end{figure}  

In Fig.~\ref{fig:sigma}, we plot the relative discrepancies of diffusion estimators from two approaches ((\ref{eq:LG-1}) and EM approximation) with the complete observations $D_{\mathrm{C}}^{n, N}$ under the experimental design I. These plots demonstrate that the standard deviation (stddev) of the EM-based estimator is larger than that of the LG-based one. Indeed, the stddev of the former is $0.00088$, which is approximately that of the latter (0.00061) multiplied by $\sqrt{2}$, as expected by the arguments in Remark \ref{rem:diff_param}. We also note that for the drift parameters, estimators from the two approaches exhibit quite similar performance in terms of bias and variance, in this experiment. 
Table \ref{table:ips_ul} summarizes the mean and the stddev of relative discrepancies for the two competitive estimators from the partial particle trajectories $D_{\mathrm{P}}^{n, N}$. We observe that both type of estimators achieve nearly the same standard deviation, but the centre of realisations for the LG-based estimators is closer to the true value compared to those from EM-based ones for all parameters. As the data step-size decreases (from Design I to II), the bias is improved for both estimators, though the LG-based estimators still perform better, i.e.~the EM-based estimators require a smaller data step-size for bias reduction. We expect this is due to scheme (\ref{eq:LG-1}) utilising a higher order stochastic Taylor expansion for $X_{S, t}^{[i]}$, so this increased precision results in a more accurate approximation of the target (perfect) objective function or marginal likelihood for a given step-size $\Delta_n$.    

\begin{table}
\begin{center} 
\setlength{\extrarowheight}{2pt}
\caption{Parameter estimation of the interacting underdamped Langevin equation under the partial observation regime: Mean and standard deviation (in brackets) of relative discrepancies for $\hat{\theta}_{\mathrm{P}}^{n, N}$ from $100$ trajectories.} 
\label{table:ips_ul} 
\begin{tabular}{cccccc} 
\toprule 
{Design} & {Approach} & $\lambda$ & $\gamma$ & $\kappa$ & $\sigma$ \\ 
\midrule
\multirow{2}{*}{I} & LG (our work) & -0.0087	(0.03776) & 0.0204 (0.04176) & -0.0030  (0.07553) & -0.0036 (0.00087)	\\ 
& EM \cite{amo:24} & -0.0221 (0.03719) & 0.0775 (0.04109) & -0.0164	(0.07436) & -0.0104 (0.00091) \\[0.1cm] 
\hline \\[-0.4cm]
\multirow{2}{*}{II} & LG (our work) & -0.0054 (0.03787) & 0.0136 (0.04206) & 0.0004 (0.07584) & 0.0016 (0.0006) \\ 
& EM \cite{amo:24} & -0.0119 (0.03761) & 0.0388 (0.04137) & -0.0062	(0.07527) & -0.0028 (0.00064) \\ 
\bottomrule
\end{tabular} 
\end{center} 
\end{table}

\section{Proofs} \label{sec:pfs}

\subsection{Preliminaries} \label{sec:pre}
We first prepare some notation for the proofs in the next subsections. We assume that the data $\{ \mathbb{X}^N_t \}_{t \in \mathbb{T}_n}$ obeys dynamics determined by the true $\trueparam \in \Theta$. Recalling that the matrix $\Sigma$ defined in \eqref{eq:Sigma} is invertible under Assumption \ref{ass:hor} due to Lemma \ref{lemma:inv_Sigma}, we write the inverse matrix of $\Sigma$ as:
\begin{align*}
\Sigma^{-1}=: \Lambda =
\begin{bmatrix}
\Lambda_{SS} & \Lambda_{SR} \\ 
\Lambda_{RS} & \Lambda_{RR}  
\end{bmatrix},  \qquad \Lambda: \Theta \times \mathbb{R}^d \times \mathcal{P}_2 (\mathbb{R}^d) \to \mathbb{R}^{d \times d},  
\end{align*}
with 
$ 
\Lambda_{SS} : \Theta \times \mathbb{R}^d \times \mathcal{P}_2 (\mathbb{R}^d) \to \mathbb{R}^{d_S \times d_S}, \,
\Lambda_{SR} : \Theta \times \mathbb{R}^d \times \mathcal{P}_2 (\mathbb{R}^d) \to \mathbb{R}^{d_S \times d_R}, \, 
\Lambda_{RR} : \Theta \times \mathbb{R}^d \times \mathcal{P}_2 (\mathbb{R}^d) \to \mathbb{R}^{d_R \times d_R}$ 
and $\Lambda_{RS} = \Lambda_{SR}^\top$. \textcolor{black}{We note that the matrix $\Lambda_{\iota_1 \iota_2}, \, \iota_1, \iota_2 \in \{S, R \}$ is distinguished from the inverse of the block matrix $\Sigma_{\iota_1 \iota_2}$, defined in \eqref{eq:Sigma} -- e.g., $\Lambda_{SS}$ is not identified with $\Sigma_{SS}^{-1}$.}
We write $\mu_t^N = \mu_t^{N, \trueparam}$ and $\mu_t \equiv \mu_t^{\trueparam}$, where we recall that $\mu_t^\theta$ is the law of the McKean-Vlasov SDE (\ref{eq:mv_I}) with the parameter value $\theta \in \Theta$. Recall the definition of $\mathbf{m}_{j-1}^{[i],\theta}$ in (\ref{eq:def_m}). We change the notation as $\mathbf{m}_{j-1}^{[i],\theta}\leftrightarrow \mathbf{m}_{j-1}^{[i]} (\theta)$, and we write, for $1 \le i \le N, \, 1 \le j \le n$:
$ 
\mathbf{m}_{j-1}^{[i]} (\theta)
= 
\bigl[\mathbf{m}_{j-1}^{[i], 1} (\theta), \ldots,  \mathbf{m}_{j-1}^{[i], d} (\theta) \bigr]^\top. 
$ 
We set $\mathbf{m}_{j-1}^{[i]} (\theta) = (\mathbf{m}_{S, j-1}^{[i]} (\theta), \mathbf{m}_{R, j-1}^{[i]} (\theta) ) \in \mathbb{R}^{d_S} \times \mathbb{R}^{d_R}$ for model class (\ref{eq:ips-1}). 
Recall from (\ref{eq:tripl}), that $Z_{j}^{[i], \theta} = (\theta, \sample{j}{i}, \mu_{t_j}^{N})$ and write $Z_j^{[i]} \equiv Z_j^{[i], \trueparam}$, for $ 1\le i \le N, \, 0  \le j \le n$.  
%
%
For $f : \Theta \times \mathbb{R}^d \times \mathcal{P}_2 (\mathbb{R}^d) \to \mathbb{R}$,  we write $f (Z_j^{[i], \theta}) \equiv f (\theta, \sample{j}{i}, \mu_{t_j}^{N})$, and with a slight abuse of notation we set:  
\begin{gather*}
V_{S, 0} (Z_j^{[i], \theta}) 
\equiv V_{S, 0} ( \alpha_S, \sample{j}{i}), \quad  
V_{R, 0} (Z_j^{[i], \theta} ) 
\equiv V_{R, 0} ( \alpha_R, \sample{j}{i}, \mu_{t_j}^{N}); \\ 
V_{R, k}  (Z_t^{[i], \theta} ) 
\equiv V_{R, k} ( \beta, \sample{j}{i}, \mu_{t_j}^{N}), \quad 1 \le k \le d_B.  
\end{gather*}  
We also introduce for $\theta \in \Theta$ and $1 \le j \le n$:
\begin{align*}
\Sigma_{j-1}^{[i]} (\theta) \equiv \Sigma (Z_{j-1}^{[i], \theta}) ,  \qquad
\Lambda_{j-1}^{[i]} (\theta) \equiv  \bigl( \Sigma_{j-1}^{[i]} (\theta) \bigr)^{-1} 
\equiv 
\begin{bmatrix}
\Lambda_{SS, j-1}^{[i]} (\theta)  & \Lambda_{SR, j-1}^{[i]} (\theta) \\[0.1cm] 
\Lambda_{RS, j-1}^{[i]} (\theta)  & \Lambda_{RR, j-1}^{[i]} (\theta) 
\end{bmatrix}.  
\end{align*} 
We denote by $\mathcal{R}$ a set of functions $f : 
\Theta \times \mathbb{R}^d \times \mathcal{P}_2 (\mathbb{R}^d) \to \mathbb{R}$ satisfying the following three properties: (i) There exist $C > 0$ and $\ell_1, \ell_2 = 0,1 \ldots$ such that for all $(y_1, \nu_1), (y_2, \nu_2) \in \mathbb{R}^d \times \mathcal{P}_{\ell} (\mathbb{R}^d)$ and all $\theta \in \Theta$:  
\begin{align} 
\label{eq:f_bd} & | f (\theta, y_1, \nu_1) - f (\theta, y_2, \nu_2) |  \\
& \quad 
\le C \bigl( |y_1 - y_2| + \mathcal{W}_2 (\nu_1, \nu_2) \bigr)
\bigl(1 + |y_1|^{\ell_1} + |y_2|^{\ell_1} + \mathcal{W}_{2}(\nu_1, \delta)^{\ell_2}  + \mathcal{W}_{2}(\nu_2, \delta)^{\ell_2} \bigr);  
\nonumber
\end{align}
(ii) For all $\theta \in \Theta$, $(y, t) \mapsto f (\theta, y, \mu_t)$ is integrable w.r.t. $\mu_t (dy) \times dt$ over $\mathbb{R}^d \times [0, T]$; 
%

\noindent(iii) There exist constants $C > 0$, $k_1, k_2 \in \{0, 1, \ldots\}$ such that for all $1 \le j \le d_\theta$ and all $(\theta, x, \mu) \in \Theta \times \mathbb{R}^d \times \mathcal{P}_2 (\mathbb{R}^d)$: 
\begin{align} \label{eq:R_3}
\bigl| \partial_{\theta_j} 
f (\theta, x, \mu) \bigr| \le  C ( 1 + |x|^{k_1} + \mathcal{W}_2 (\mu, \delta_0)^{k_2}).  
\end{align} 
\\ 

\noindent 
\textit{Two Key Convergence Results.}
We present two convergence results that are crucial for proving Theorems \ref{thm:consistency} and \ref{thm:clt}, particularly in examining the limits of scaled contract functions as $n, N \to \infty$:  %
\begin{lemma} \label{lemma:base_conv}
Let $f \in \mathcal{R}$. It holds under Assumptions \ref{ass:initial_dist}, \ref{ass:growth} that: 
\begin{align} \label{eq:conv_aux}
\tfrac{\Delta_n}{N} \sum_{1 \le i \le N} \sum_{1 \le j \le n} f (\theta, X_{t_{j-1}}^{[i]}, \mu_{t_{j-1}}^N) \probconv 
\int_0^T \mathbb{E}_{\mu_t} 
\bigl[ f (\theta, W, \mu_t) \bigr] dt, \quad n, N \to \infty,    
\end{align}
uniformly in $\theta \in \Theta$. 
\end{lemma}
\begin{lemma}  \label{lemma:key_conv}
Let Assumptions \ref{ass:law_dep}--\ref{ass:growth} hold and $g \in \mathcal{R}$. Then, for any $k_1, k_2 = 1, \ldots, d$, we have that:
\begin{align}
& \tfrac{\sqrt{\Delta_n}}{N} 
\sum_{1 \le i \le N} \sum_{1 \le j \le n}  
g (Z_{j-1}^{[i], \theta}) \mathbf{m}_{j-1}^{[i], k_1} (\trueparam)  \probconv 0; \label{eq:aux_conv_1} \\[0.3cm]
& 
\begin{aligned} 
& \tfrac{\Delta_n}{N}  
\sum_{1 \le i \le N} \sum_{1 \le j \le n} 
g (Z_{j-1}^{[i], \theta})  
\mathbf{m}_{j-1}^{[i], k_1} (\trueparam) \mathbf{m}_{j-1}^{[i], k_2} (\trueparam) \\[-0.1cm] 
& \qquad \probconv  
\int_0^T 
\mathbb{E}_{\mu_t} \bigl[ 
g (\theta, W, \mu_t) 
\Sigma_{k_1 k_2} (\trueparam, W, \mu_t) \bigr] dt,    
\end{aligned} 
\label{eq:aux_conv_2} 
\end{align} 
as $n, N \to \infty$, uniformly in $\theta \in \Theta$. 
\end{lemma} 
The first result, Lemma \ref{lemma:base_conv}, is used and obtained 
in the works \cite{amo:24, amo:23} in the sense of point convergence
(for fixed $\theta$), where they assumed the global Lipschitz condition on the drift function. We note that Lemma \ref{lemma:base_conv} is now established under a locally Lipschitz condition (Assumption \ref{ass:growth}) on the drift function with the aid of results on finite moments of IPS and McKean-Vlasov SDE, e.g. \cite{dos:22, dos:19}, which are stated in Lemma \ref{lemma:bds_aux} in Appendix \ref{sec:aux}. 
Then, the proof of Lemma \ref{lemma:base_conv} is contained in Appendix \ref{sec:pf_base_conv}. 
Lemma \ref{lemma:key_conv} also plays a critical role when carrying out asymptotic analysis of the proposed contrast function, and its proof is postponed to Appendix \ref{sec:pf_key_conv}. \\ 

\noindent 
\textit{Decomposition of Contrast Function.} 
The fundamental procedure for proving the asymptotic results (Theorems \ref{thm:consistency}, \ref{thm:clt}) involves decomposing the contrast function (\ref{eq:contrast}), scaled by an appropriate factor, into several components. Subsequently, we apply the above two convergence results (Lemmas \ref{lemma:base_conv}, \ref{lemma:key_conv}) to these components under Assumptions \ref{ass:law_dep}-\ref{ass:reg_coeff}. 
%
In particular, we often use the following decomposition:
\begin{align}
\ell_{n, N} (\theta)  = \sum_{1 \le k \le 4} 
L^{(k)}_{n, N} (\theta) \label{eq:contrast_decomp}
\end{align}
with constituent components:
\begin{align*}
L^{(1)}_{n, N} (\theta) 
& = \sum_{1 \le i \le N} \sum_{1 \le j \le n} \mathbf{m}_{j-1}^{[i]} (\trueparam)^\top \Lambda_{j-1}^{[i]} (\theta) \mathbf{m}_{j-1}^{[i]} (\trueparam); \\ 
L^{(2)}_{n, N} (\theta) 
& = 2 \sum_{1 \le i \le N} \sum_{1 \le j \le n} \bigl( \mathbf{m}_{j-1}^{[i]} (\theta) 
- \mathbf{m}_{j-1}^{[i]} (\trueparam)  
\bigr)^\top \Lambda_{j-1}^{[i]} (\theta) \mathbf{m}_{j-1}^{[i]} (\trueparam); \\ 
L^{(3)}_{n, N} (\theta) 
& = \sum_{1 \le i \le N} \sum_{1 \le j \le n} 
\bigl( \mathbf{m}_{j-1}^{[i]} (\theta) 
- \mathbf{m}_{j-1}^{[i]} (\trueparam)  
\bigr)^\top \Lambda_{j-1}^{[i]} (\theta)
\bigl( \mathbf{m}_{j-1}^{[i]} (\theta) 
- \mathbf{m}_{j-1}^{[i]} (\trueparam)  
\bigr); \\  
L^{(4)}_{n, N} (\theta)   
& =  \sum_{1 \le i \le N} \sum_{1 \le j \le n} 
\log \det \Sigma_{j-1}^{[i]} (\theta). 
\end{align*} 
Furthermore, the following expression for the term $\mathbf{m}_{j-1}^{[i]} (\theta)  - \mathbf{m}_{j-1}^{[i]} (\trueparam)$ is crucial in our analysis:
\begin{align}
\label{eq:m_diff}   
&\mathbf{m}_{j-1}^{[i]} (\theta)  - \mathbf{m}_{j-1}^{[i]} (\trueparam)   = \tfrac{1}{\sqrt{\Delta_n}}
\begin{bmatrix}
b_{S} (\alpha_S, X_{t_{j-1}}^{[i]})  \\[0.1cm]
\mathbf{0}_{d_R} 
\end{bmatrix} \\
& \qquad \qquad \qquad + \sqrt{\Delta_n} 
\begin{bmatrix}
\tfrac{1}{2} \mathscr{L}_0 
[V_{S, 0 } (\alpha_S^\dagger, \cdot)] (Z_{j-1}^{[i], \trueparam})
- 
\tfrac{1}{2} \mathscr{L}_0  
[V_{S, 0 } (\alpha_S, \cdot)] (Z_{j-1}^{[i], \theta}) \\[0.2cm] 
b_{R} (\alpha_R, X_{t_{j-1}}^{[i]}, \mu_{t_{j-1}}^N)
\end{bmatrix},
\nonumber 
\end{align}
where we have set, 
for $(\alpha_S, \alpha_R) \in \Theta_{\alpha_S} \times \Theta_{\alpha_R}$,   
$(x , \mu) \in \mathbb{R}^d \times \mathcal{P}_2 (\mathbb{R}^d)$: 
\begin{align*} 
\begin{bmatrix}
b_S  (\alpha_S, \mu)  \\[0.1cm]
b_R  (\alpha_R , x, \mu) 
\end{bmatrix}
= 
\begin{bmatrix}
V_{S, 0} (\alpha_S^\dagger,  x) - V_{S, 0} (\alpha_S,  x)  \\[0.1cm]
V_{R, 0} (\alpha_R^\dagger,  x, \mu) - V_{R, 0} (\alpha_R,  x, \mu)  
\end{bmatrix}.
\end{align*}  

\subsection{Proof of Theorem \ref{thm:consistency}}  \label{sec:pf_consistency}
We show the consistency for $d_S \ge 1$, i.e. an interacting hypoelliptic diffusion. 
We carefully extend the arguments used in the case of linear in the sense of McKean hypoelliptic SDEs 
in \citep{igu:ejs, igu:bj} to our setting of an IPS. We emphasize here that the proof strategy differs from that in \cite{amo:24} because our case treats a more general smooth drift function $V_{S, 0} (\alpha_S, X_{t}^{[i]})$ and relies on the locally Gaussian approximation (\ref{eq:LG-1}). More precisely, due to the inclusion of Gaussian noise of size $\mathcal{O} (\Delta_n^{3/2})$  in approximation of the smooth coordinates or the existence of $\Delta_n^{3/2}$ in the denominator of $\mathbf{m}_S$ in the contrast function (\ref{eq:contrast}), one must carefully analyze the (scaled) contrast function as $\Delta_n \to 0$ to avoid explosions. Our proof of consistency contains four main steps, and in the process (\textbf{Step II.} below), we derive a convergence rate of the estimator $\hat{\alpha}_S^{n, N}$ that is effectively used in the rest of the proof to control the asymptotic behavour of scaled contrast function. The elliptic case, i.e. $d_S = 0$, can be shown by skipping the first two steps without relying on the convergence rate.  

We now prove the consistency of the proposed estimator through the following four steps: \\
\noindent 
\textbf{Step I.} We start by showing that $\hat{\alpha}_S^{n, N}  \probconv \alpha_S^\dagger$ as $n, N \to \infty$. In particular, we prove the following lemma: 

\begin{lemma} \label{lemma:step-1}
Under Assumptions \ref{ass:law_dep}--\ref{ass:reg_coeff} it holds that: 
\begin{align*}
\sup_{\theta \in \Theta}\, 
\Bigl|   \tfrac{\Delta_n^2}{N}  \ell_{n, N} (\theta)  - \mathcal{Q} (\theta) \Bigr| \probconv 0, \qquad 
n , N \to \infty,
\end{align*}
where we have defined: 
\begin{align} 
\label{eq:limit_Q}
\mathcal{Q} (\theta) = \int_0^T \mathbb{E}_{\mu_t} 
\bigl[  b_S (\alpha_S, W)^\top \Lambda_{SS} (\theta, W, \mu_t^{\dagger})\,b_S (\alpha_S, W)  \bigr] dt.  
\end{align} 
\end{lemma} 

\noindent 
\textit{Proof of Lemma \ref{lemma:step-1}.} From (\ref{eq:contrast_decomp}), (\ref{eq:m_diff}) we get $\textstyle \tfrac{\Delta_n^2}{N} \ell_{n, N} (\theta) = \sum_{k = 1}^4 R^{(k)}_{n, N} (\theta)$ with: 
\begin{align*}
R^{(1)}_{n, N} (\theta)  
& =  \tfrac{\Delta_n^2}{N} L_{n, N}^{(1)} (\theta), \qquad
R^{(2)}_{n, N} (\theta)   
=  \tfrac{\Delta_n^2}{N} L_{n, N}^{(2)} (\theta) ; \\[0.2cm] 
R^{(3)}_{n, N} (\theta)   
& = \tfrac{\Delta_n}{N} 
\sum_{1 \le i \le N} \sum_{1 \le j \le n}  
b_{S} (\alpha_S, X_{t_{j-1}}^{[i]})^\top 
\Lambda_{SS, j-1}^{[i]} (\theta)  \,
b_{S} (\alpha_S, X_{t_{j-1}}^{[i]}), \\
R^{(4)}_{n, N} (\theta)   
&= \tfrac{\Delta_n^2}{N} 
\sum_{1 \le i \le N} \sum_{1 \le j \le n}   
\widetilde{R}^{(4)} (Z_{j-1}^{[i]}),    
\end{align*} 
where $\widetilde{R}^{(4)} \in \mathcal{R}$. We now use Lemmas \ref{lemma:base_conv}, \ref{lemma:key_conv} to obtain that, if $n, N \to \infty$ then 
$
{R}^{(3)}_{n, N} (\theta) \probconv \mathcal{Q} (\theta)$,
${R}^{(k)}_{n, N} (\theta) \probconv 0$,  $k  =  1, 2, 4, 
$
uniformly in $\theta \in \Theta$. We thus conclude Lemma \ref{lemma:step-1}. 
\qed
\\ 

Given Lemma \ref{lemma:step-1}, consistency for $\hat{\alpha}_S^{n, N}$ is established as follows. Under Assumption \ref{ass:hor}, the matrix $\Lambda_{SS}$ is positive-definite uniformly in $(\theta, x, \mu) \in \Theta \times \mathbb{R}^d \times \mathcal{P}_2 (\mathbb{R}^d)$. Then, under Assumption \ref{ass:ident} we have that for any $\varepsilon > 0$ there exists $\delta > 0$ such that
$
\mathbb{P}_{\trueparam} \bigl( |\hat{\alpha}_S^{n ,N} - \alpha_S^\dagger| > \varepsilon \bigr)
\le \mathbb{P}_{\trueparam} \bigl( 
\mathcal{Q} (\hat{\theta}^{n, N}) > \delta  \bigr).   
$
We also have that: 
\begin{align*}
\mathbb{P}_{\trueparam} &\bigl( 
\mathcal{Q} (\hat{\theta}^{n, N}) > \delta  \bigr) 
\le \mathbb{P}_{\trueparam} \Bigl( 
\tfrac{\Delta_n^2}{N} \ell_{n, N} (\alpha_S^\dagger, 
\hat{\alpha}_R^{n, N}, \hat{\beta}^{n, N})
-\tfrac{\Delta_n^2}{N} \ell_{n, N} (\hat{\theta}^{n, N})
+ \mathcal{Q} (\hat{\theta}^{n, N}) > \delta  \Bigr) 
\\ 
& \le 
\mathbb{P}_{\trueparam} \Bigl( 
\, \sup_{\theta \in \Theta} 
\bigl| 
\tfrac{\Delta_n^2}{N} \ell_{n, N} (\alpha_S^\dagger, 
{\alpha}_R, \beta)
-\tfrac{\Delta_n^2}{N} \ell_{n, N} (\theta)
+ \mathcal{Q} (\theta) 
\bigr| > \delta  \Bigr) 
\to 0, \qquad n, N \to \infty, 
\end{align*}
where in the last line we made use of Lemma \ref{lemma:step-1}. We thus conclude the consistency of $\hat{\alpha}_S^{n, N}$. 

\begin{remark} \label{rem:unique}
Lemma \ref{lemma:step-1} implies that optimising the scaled objective function $\textcolor{black}{\tfrac{\Delta_n^2}{N}} \ell_{n, N} (\theta)$ is asymptotically equivalent to optimizing function $\mathcal{Q}(\theta)$ given in (\ref{eq:limit_Q}). In particular, $Q(\theta)$ admits the unique minima $\alpha_S = \alpha_S^\dagger$ under Assumptions \ref{ass:hor} and \ref{ass:ident} due to the positive definiteness of $\Lambda_{SS} (\theta, W, \mu^\dagger)$ for any $(\theta, W) \in \Theta \times \mathbb{R}^d$ and $b_S (\alpha_S, x) \neq 0$, $\mu_t (dx) \, \mathrm{a.e.} \,$ for all $t \in [0, T]$ whenever  $\alpha_S \neq \alpha_S^\dagger$. Similarly, the asymptotic results below imply that scaled objective functions asymptotically allow a unique minimum for the cases of $\alpha_R$ and $\beta$ as well under Assumptions \ref{ass:hor} and \ref{ass:ident}.
\end{remark}

\noindent 
\textbf{Step II.} We obtain a convergence rate for $\hat{\alpha}_S^{n, N}$, specifically $\hat{\alpha}_S^{n , N} - \alpha^\dagger_S  = o_{\mathbb{P}_\trueparam} ( \sqrt{\Delta_n} )$, which is not sharp but suffices to prove consistency for the remaining estimators. 
To this end, we consider the Taylor expansion of $\partial_{\alpha_S} \ell_{n, N} (\hat{\theta}^{n,N})$ around  $(\alpha_S^\dagger, \hat{\alpha}_R^{n,N}, \hat{\beta}^{n,N})$. As  $\partial_{\alpha_S} \ell_{n, N} (\hat{\theta}^{n,N}) = \mathbf{0}_{d_{\alpha_S}}$ from the definition of $\hat{\theta}^{n, N}$, we have that: 
\begin{align} 
\label{eq:taylor_step_II}
{A}^{n ,N} (\alpha_S^\dagger, \hat{\alpha}_R^{n,N}, \hat{\beta}^{n,N}) = {B}^{n,N} (\hat{\theta}^{n, N})  \times \tfrac{1}{\sqrt{\Delta_n}} 
\bigl( \hat{\alpha}_S^{n,N} - \alpha_S^\dagger  \bigr),  
\end{align} 
where for $\theta = (\alpha_S, \alpha_R, \beta) \in \Theta$ we set:
\begin{align*}
{A}^{n,N} (\theta) & := - \tfrac{\sqrt{\Delta_n^3}}{N} \partial_{\alpha_S} \ell_{n, N} (\theta), \\ 
{B}^{n,N} (\theta) & := \tfrac{\Delta_n^2}{N}  
\textstyle \int_0^1 \partial_{\alpha_S}^\top  \partial_{\alpha_S}  
\ell_{n, N} \bigl( \alpha_S^\dagger + \lambda (\alpha_S - \alpha_S^\dagger), \alpha_R, \beta   \bigr) d \lambda.  
\end{align*} 
\begin{lemma} \label{lemma:step-II}
Under Assumptions \ref{ass:law_dep}--\ref{ass:reg_coeff} it holds that: 
\begin{align} 
& \sup_{(\alpha_R, \beta) \in \Theta_{\alpha_R} \times \Theta_{\beta} } \bigl| A^{n, N} (\alpha_S^\dagger,  \label{eq:limit_A}
\alpha_R, \beta)  \bigr|  \probconv 0, \qquad n, N \to \infty; \\ 
& \sup_{(\alpha_R, \beta) \in \Theta_{\alpha_R} \times \Theta_{\beta}} 
\bigl| B^{n, N} (\hat{\alpha}_S^{n, N},  \alpha_R,  \beta)  - 2 \times \mathcal{B} (\alpha_S^\dagger, \beta)   \bigr|  \probconv 0, \qquad n, N \to \infty,  \label{eq:limit_B}
\end{align} 
where we have defined:
\begin{align*}
\mathcal{B} (\alpha_S, \beta) = \int_0^T 
\mathbb{E}_{\mu_t} 
\Bigl[ \bigl(  \partial_{\alpha_S}^\top V_{S, 0} (\alpha_S, W) \bigr)^\top 
\Lambda_{SS} \bigl( (\alpha_S, \beta),  W, \mu_t \bigr) 
\,\partial_{\alpha_S}^\top V_{S, 0} (\alpha_S, W) 
\Bigr]  dt. 
\end{align*}
\end{lemma} 

\noindent 
\textit{Proof of Lemma \ref{lemma:step-II}.}  \textit{I. Proof of  (\ref{eq:limit_A})}. 
We set $\theta_S^\dagger := (\alpha_S^\dagger, \alpha_R, \beta)$.
We have from (\ref{eq:m_diff}) that: 
\begin{align} \label{eq:decomp_S}
\mathbf{m}_{j-1}^{[i]} (\alpha_S^\dagger, \alpha_R, \beta) 
- \mathbf{m}_{j-1}^{[i]} (\trueparam)  
= \sqrt{\Delta_n} \times f (\theta_S^\dagger,  X_{t_{j-1}}^{[i]}, \mu_{t_{j-1}}^N), 
\end{align}
for some $f: \Theta \times \mathbb{R}^d \times \mathcal{P}_2 (\mathbb{R}^d) \to \mathbb{R}^d$ with  each coordinate shown to satisfy $f^m \in \mathcal{R}$, $1 \le m \le d$, under Assumptions \ref{ass:growth}, \ref{ass:test_function}, \ref{ass:reg_coeff}. We then have from (\ref{eq:contrast_decomp}) and \eqref{eq:decomp_S} that, for $1 \le k \le d_{\alpha_S}$,  $(\alpha_R, \beta) \in \Theta_{\alpha_R} \times \Theta_{\beta}$:
\begin{align*} 
- \tfrac{\sqrt{\Delta_n^3}}{N} \partial_{\alpha_S, k} \ell_{n, N} (\alpha_S, \alpha_R, \beta) |_{\alpha_S = \alpha^\dagger_S} 
= \sum_{1 \le \iota \le 3} \mathcal{S}_{n, N}^{(\iota), k} (\alpha_S^\dagger, \alpha_R, \beta),
\end{align*} 
with individual terms:
\begin{align*}
\mathcal{S}_{n, N}^{(1), k} (\theta)
& = 
\tfrac{\sqrt{\Delta_n^3}}{N} \sum_{1 \le i \le N} 
\sum_{1 \le j \le n} \sum_{1 \le \ell_1, \ell_2 \le d} f^{(1)}_{\ell_1 \ell_2} (Z_{j-1}^{[i], \trueparam_S}) 
\mathbf{m}_{j-1}^{[i], \ell_1} (\trueparam)
\mathbf{m}_{j-1}^{[i], \ell_2} (\trueparam);  \\   
\mathcal{S}_{n, N}^{(2), k} (\theta) 
& =  h (\Delta_n^{1/2}) \times 
\tfrac{\sqrt{\Delta_n}}{N} 
\sum_{1 \le i \le N} 
\sum_{1 \le j \le n} 
\sum_{1 \le \ell \le d} 
f^{(2)}_{\ell} (Z_{j-1}^{[i], \trueparam_S}) \mathbf{m}_{j-1}^{[i], \ell} (\trueparam), \\ 
\mathcal{S}_{n, N}^{(3), k} (\theta)   
&= h  (\Delta_n^{1/2}) \times 
\tfrac{\Delta_n}{N} 
\sum_{1 \le i \le N} 
\sum_{1 \le j \le n}  
f^{(3)} (Z_{j-1}^{[i], \trueparam_S}), 
\end{align*} 
where $f_{\ell_1 \ell_2}^{(1)}, \, f_{\ell}^{(2)}, f^{(3)} \in \mathcal{R}$ and $h  : [0, \infty) \to \mathbb{R}$ such that $| h (t) | \lesssim t$, $t \in [0, \infty)$.  Applying Lemmas \ref{lemma:base_conv}, \ref{lemma:key_conv} to $\mathcal{S}_{n, N}^{(\iota), k} (\theta)$, $\iota = 1, 2, 3$, we conclude \eqref{eq:limit_A}.  \\ 

\noindent  
\textit{II. Proof of  (\ref{eq:limit_B})}. We will show that for $1 \le k_1, k_2 \le d$: 
\begin{align} 
\label{eq:f_a_S}
\begin{aligned} 
\tfrac{\Delta_n^2}{N} &\partial_{\alpha_S, k_1}  \partial_{\alpha_S, k_2}  \ell_{n, N} (\theta) \probconv \\   &2  
\textstyle \int_{0}^T \mathbb{E}_{\mu_t} \bigl[  \partial_{\alpha_S, k_1} V_{S,0} (\alpha_{S},  W)^\top  \Lambda (\theta, W, \mu_t)  \partial_{\alpha_S, k_2} V_{S,0} (\alpha_{S},  W)  \bigr] dt \\ 
&   
\qquad + 2  \textstyle \int_{0}^T \mathbb{E}_{\mu_t} \bigl[  \partial_{\alpha_S, k_1}  \partial_{\alpha_S, k_2} \bigl( V_{S,0} (\alpha_{S}, W)^\top  \Lambda (\theta, W, \mu_t)  \bigr) 
b_S (\alpha_S, W)
\bigr] dt,  
\end{aligned}  
\end{align} 
uniformly in $\theta$, as $N, n \to \infty$.  We have that  $\textstyle \tfrac{\Delta_n^2}{N} \partial_{\alpha_S, k_1}  \partial_{\alpha_S, k_2}  \ell_{n, N} (\theta) = \sum_{i=1}^5 \mathcal{T}_{n, N}^{(\iota)} (\theta)$ with:
\begin{align*}
\mathcal{T}_{n, N}^{(1)} (\theta) 
& = 2 \times \tfrac{\Delta_n}{N} \sum_{1 \le i \le N} 
\sum_{1 \le j \le n} 
\partial_{\alpha_S, k_1} V_{S, 0} (\alpha_S, \sample{j-1}{i})^\top \Lambda_{j-1}^{[i]} (\theta)
\partial_{\alpha_S, k_2} V_{S, 0} (\alpha_S, \sample{j-1}{i}); \\  
\mathcal{T}_{n, N}^{(2)} (\theta) 
& = 2 \times \tfrac{\Delta_n}{N} \sum_{1 \le i \le N} 
\sum_{1 \le j \le n} 
\partial_{\alpha_S, k_1}  \partial_{\alpha_S, k_2} \bigl( V_{S, 0} (\alpha_S, \sample{j-1}{i}) \Lambda_{j-1}^{[i]} (\theta) \bigr)^\top  
b_S (\alpha_S, \sample{j-1}{i}) ; \\   
\mathcal{T}_{n, N}^{(3)} (\theta)  
& = \tfrac{\Delta_n^2}{N} \sum_{1 \le i \le N} 
\sum_{1 \le j \le n}  
\mathbf{m}_{j-1}^{[i]} (\trueparam)^\top  \partial_{\alpha_S, k_1}  \partial_{\alpha_S, k_2}  \Lambda_{j-1}^{[i]} (\theta) \, \mathbf{m}_{j-1}^{[i]} (\trueparam);  \\ 
\mathcal{T}_{n, N}^{(4)} (\theta)   
& = \tfrac{\sqrt{\Delta_n^3}}{N} \sum_{1 \le i \le N} 
\sum_{1 \le j \le n} 
\sum_{1 \le \ell \le d}  
f_{k_1 k_2}^{(4), \ell}  (Z_{j-1}^{[i], \theta}) \,  \mathbf{m}_{j-1}^{[i], \ell} (\trueparam), \\ 
\mathcal{T}_{n, N}^{(5)} (\theta)   
&= \tfrac{\Delta_n^2}{N} \sum_{1 \le i \le N} 
\sum_{1 \le j \le n}  
f_{k_1 k_2}^{(5)} 
(Z_{j-1}^{[i], \theta}),  
\end{align*} 
where $f_{k_1 k_2}^{(4), \ell}, f_{k_1 k_2}^{(5)} \in \mathcal{R}$. We then obtain \eqref{eq:f_a_S} by applying Lemmas \ref{lemma:base_conv}, \ref{lemma:key_conv} to $\mathcal{T}_{n, N}^{(\iota)} (\theta), \ 1 \le \iota \le 5$. Due to the uniform convergence w.r.t. $\theta \in \Theta$ and the consistency of $\hat{\alpha}_{S}^{n, N}$, we get (\ref{eq:limit_B}). Proof of Lemma \ref{lemma:step-II} is now complete. \qed

Given Lemma \ref{lemma:step-II}, from equation (\ref{eq:taylor_step_II}) we immediately obtain that: 
\begin{align} 
\label{eq:rate_a_S}
\hat{\alpha}_{S}^{n ,N} - \alpha_S^\dagger = o_{\,\mathbb{P}_{\trueparam}} (\sqrt{\Delta_n}). 
\end{align}  

\noindent 
\textbf{Step III.} Given the convergence rate (\ref{eq:rate_a_S}) for $\hat{\alpha}_{S}^{n, N}$ in \textbf{Step  II.}, we prove that $\hat{\beta}^{n, N} \probconv \beta^\dagger$ as $n, N \to \infty$. Using the consistency of $\hat{\alpha}_S^{n, N}$ with (\ref{eq:rate_a_S}), we obtain the following result. 
%
\begin{lemma} \label{lemma:step-III}
Under Assumptions \ref{ass:law_dep}--\ref{ass:reg_coeff} it holds that: 
\begin{align*}
\sup_{ (\alpha_R, \beta)  \in \Theta_{\alpha_R} \times \Theta_{\beta}} 
\bigl|  
\tfrac{\Delta_n}{N}  \ell_{n, N} (\hat{\alpha}_S^{n, N}, \alpha_R, \beta)   
- \mathcal{R} (\beta) 
\bigr| \probconv 0, 
\end{align*}
where we have defined:
\begin{align*}
\mathcal{R} (\beta) := \int_0^T \mathbb{E}_{\mu_t}
\bigl[ \mathrm{tr} \bigl\{ \Lambda \bigl(
(\alpha_S^\dagger, \beta), W, \mu_t
\bigr) \Sigma \bigl( (\alpha_S^\dagger, \beta^\dagger), W, \mu_t \bigr) 
\bigr\} + \log \det \Sigma \bigl( (\alpha_S^\dagger, \beta), W, \mu_t \bigr) \bigr] dt.  
\end{align*}
%
%
\end{lemma} 

\noindent 
\textit{Proof of Lemma \ref{lemma:step-III}.} From (\ref{eq:contrast_decomp}), we obtain    $\textstyle \tfrac{\Delta_n}{N} \ell_{n, N} (\theta) = \sum_{\iota = 1}^4 \mathcal{U}_{n, N}^{(\iota)} (\theta)$ for the functions: 
\begin{align*}
\mathcal{U}_{n, N}^{(1)} (\theta) 
& :=  \tfrac{1}{N} \sum_{1 \le i \le N} \sum_{1 \le j \le n} 
b_S (\alpha_S, \sample{j-1}{i})^\top 
\Lambda_{SS, j-1}^{[i]} (\theta) b_S (\alpha_S, \sample{j-1}{i}), \\[0.2cm]
\mathcal{U}_{n, N}^{(2)} (\theta)  
&:=  \tfrac{\Delta_n}{N} \bigl\{ L_{n, N}^{(1)} (\theta) +  L_{n, N}^{(4)} (\theta)  \bigr\}; 
\\[0.2cm] 
\mathcal{U}_{n, N}^{(3)} (\theta)  
& :=  h^{(3)} (1) \times \tfrac{\sqrt{\Delta_n}}{N} \sum_{1 \le i \le N} \sum_{1 \le j \le n} \sum_{1 \le k \le d} f^{(3), k} (Z_{j-1}^{[i], \theta}) \mathbf{m}_{j-1}^{[i], k} (\trueparam); \\  
\mathcal{U}_{n, N}^{(4)} (\theta)  
&  := h^{(4)} (\Delta_n^{1/2}) \times \tfrac{\Delta_n}{N} \sum_{1 \le i \le N} \sum_{1 \le j \le n} 
f^{(4)} (Z_{j-1}^{[i], \theta}),  
\end{align*} 
for $\theta = (\alpha_S, \alpha_R, \beta) \in \Theta$, where $f^{(3), k}, f^{(4)} \in \mathcal{R}$ and $h^{(\iota)} : [0, \infty) \to \mathbb{R}, \, \iota = 3, 4$, such that $| h^{(\iota)} (t) | \lesssim t$, $t \in [0, \infty)$. We directly have from Lemmas \ref{lemma:base_conv}, \ref{lemma:key_conv} that if $n, N \to \infty$, then: $\mathcal{U}_{n, N}^{(\iota)} (\theta)  \probconv 0, \, \iota = 3, 4$ and 
\begin{align*}
& \mathcal{U}_{n, N}^{(2)} (\theta)  \\ 
& \probconv 
\textstyle \int_0^T \mathbb{E}_{\mu_t} \bigl[  \mathrm{tr} \bigl\{ \Lambda ((\alpha_S, \beta), W, \mu_t)  \Sigma ((\alpha_S, \beta^\dagger), W, \mu_t)  \bigr\} 
+ \log \det \Sigma ((\alpha_S, \beta), W, \mu_t) \bigr] dt,   
\end{align*} 
uniformly in $\theta \in \Theta$. 
We now study the limit of $\mathcal{U}_{n, N}^{(1)} (\hat{\alpha}_S^{n, N}, \alpha_R, \beta)$. Performing Taylor expansion for $V_{S, 0} (\alpha_S, \sample{j-1}{i}) $ at $\alpha_S = \alpha_S^\dagger$, we have that: 
$
b_S (\alpha_S, \sample{j-1}{i}) / \sqrt{\Delta_n} 
= \Phi (\alpha_S, \sample{j-1}{i}) \times \Bigl( \tfrac{\alpha_S^\dagger -\alpha_S}{\sqrt{\Delta_n}} \Bigr),
$
where we have considered the function:
\begin{align*}
\Phi (\alpha_S, \sample{j-1}{i})  
:= \textstyle 
\int_0^1  \partial_{\xi}^\top V_{S, 0}  \bigl( \xi, \sample{j-1}{i}  \bigr) |_{\xi = \alpha_S + \nu (\alpha_S^\dagger - \alpha_S)}  d \nu. 
\end{align*}
We then express $\mathcal{U}_{n, N}^{(1)} (\theta)$ as  
$
\mathcal{U}_{n, N}^{(1)} (\theta)  := \Bigl(  \tfrac{ \alpha_S^\dagger- \alpha_S}{\sqrt{\Delta_n}} \Bigr)^\top 
\,  
\widetilde{\mathcal{U}}_{n, N}^{(1)} (\theta) 
\, 
\Bigl(  \tfrac{ \alpha_S^\dagger- \alpha_S}{\sqrt{\Delta_n}} \Bigr) 
$ 
with: 
\begin{align*} 
\widetilde{\mathcal{U}}_{n, N}^{(1)} (\theta)
:= \tfrac{\Delta_n}{N} 
\sum_{1 \le i \le N} 
\sum_{1 \le j \le n} 
\Phi (\alpha_S, \sample{j-1}{i})^\top 
\Lambda_{SS, j-1}^{[i]} (\theta)
\,\Phi (\alpha_S, \sample{j-1}{i}).   
\end{align*}
Assumptions \ref{ass:hor}-\ref{ass:reg_coeff} together with the finite moments of IPS in $L_p$ sense (Lemma \ref{lemma:bds_aux} in \cite{supp_AoS}) leads to that the $\{ \widetilde{\mathcal{U}}_{n, N}^{(1)} (\theta)\}_{n, N}$ is a  bounded sequence uniformly in $\theta \in \Theta$ in probability. Then, due to (\ref{eq:rate_a_S}), we obtain:
\begin{align*}
\mathcal{U}_{n, N}^{(1)} (\hat{\alpha}_S^{n, N}, \alpha_R, \beta) 
\probconv 0,  \quad n, N \to \infty,
\end{align*}
uniformly in $(\alpha_R, \beta) \in \Theta_{\alpha_R} \times \Theta_\beta$.  
The proof of Lemma \ref{lemma:step-III} is now complete.  \qed 
\\ 

Given Lemma \ref{lemma:step-III}, we prove the consistency of $\hat{\beta}^{n, N}$ as follows. First, we have that 
$ \textstyle 
\mathcal{R} (\beta) - \mathcal{R} (\truebeta)  = \int_0^T \mathbb{E}_{\mu_t} \bigl[ \widetilde{\mathcal{R}}  (\beta, W, \mu_t )\bigr]  dt, 
$
with: 
\begin{align*}
\widetilde{\mathcal{R}}  (\beta, W, \mu_t ) :=
\mathrm{tr} \bigl\{ \Lambda \bigl( (\alpha_S^\dagger, \beta), W, \mu_t
\bigr) \Sigma \bigl( (\alpha_S^\dagger, \truebeta), W, \mu_t \bigr) 
\bigr\} - d 
+ \log \tfrac{\det \Sigma \bigl( (\alpha_S^\dagger, \beta), W, \mu_t \bigr) }{\det \Sigma \bigl( (\alpha_S^\dagger, \truebeta), W, \mu_t \bigr)}.   
\end{align*}
Note that $\tfrac{1}{2} \widetilde{\mathcal{R}}  (\beta, W, \mu_t ) $ coincides with the Kullback-Leibler divergence between the Gaussian distributions $\mathscr{N} \bigl( \mathbf{0}_d, \Sigma \bigl( (\alpha_S^\dagger, \beta), W, \mu_t \bigr) \bigr)$ and $\mathscr{N} \bigl( \mathbf{0}_d, \Sigma \bigl( (\alpha_S^\dagger, \truebeta), W, \mu_t \bigr) \bigr)$,   so $\mathcal{R} (\beta) - \mathcal{R}(\truebeta) \ge 0$ for all $\beta \in \Theta_\beta$ with the equality satisfied iff $\beta = \truebeta$  under Assumption \ref{ass:ident}. Thus, for any $\epsilon > 0 $ there exists $\delta > 0$ such that 
$
\mathbb{P}_{\trueparam} \bigl( | \hat{\beta}^{n, N} - \truebeta | > \varepsilon \bigr)\le 
\mathbb{P}_{\trueparam} \bigl( \mathcal{R}(\hat{\beta}^{n, N}) - \mathcal{R} (\truebeta) > \delta \bigr). 
$ 
%
Then, the last term is shown to converge to $0$ as $n, N \to \infty$ from a similar argument used in \textbf{Step I.} together with Lemma \ref{lemma:step-III}, and then we conclude the consistency of $\hat{\beta}_{n, N}$.  \\

\noindent 
\textbf{Step IV.}  
We finally consider $\hat{\alpha}_R^{n,N}$. We introduce 
$
D (\theta) := \tfrac{1}{N} \ell_{n, N} (\theta) - \tfrac{1}{N} \ell_{n, N} (\alpha_S, \alpha_R^\dagger, \beta)$,  $\theta = (\alpha_S, \alpha_R, \beta) \in \Theta.  
$
%
%
\begin{lemma} \label{lemma:step_IV}
Under Assumptions \ref{ass:law_dep}--\ref{ass:reg_coeff}, it holds that:
\begin{align*}
\sup_{ \theta \in \Theta} 
\bigl| D (\theta) - \textcolor{black}{\mathcal{S} (\alpha_R, \beta)} \bigr| \probconv 0, 
\qquad n , N \to \infty, 
\end{align*}
where we have defined:
\begin{align*}
\textcolor{black}{\mathcal{S} (\alpha_R, \beta) = \textstyle \int_0^T 
\mathbb{E}_{\mu_t}
\bigl[ 
b_R (\alpha_R, W, \mu_t)^\top 
\Sigma_{RR}^{-1} (\beta, W, \mu_t)
\, b_R (\alpha_R, W, \mu_t)  
\bigr] dt.} 
\end{align*}
%
%
%
\end{lemma}

\noindent 
\textit{Proof of Lemma \ref{lemma:step_IV}.} We write $\widetilde{\theta} = (\alpha_S, \alpha_R^\dagger, \beta)$. We note that $L_{n, N}^{(1)} (\theta)$ and $L_{n, N}^{(4)} (\theta)$ defined in (\ref{eq:contrast_decomp}) are independent of $\alpha_R$ (since $\Sigma_{j-1}^{[i]} (\theta)$ and  $\Lambda_{j-1}^{[i]} (\theta)$ depend only  on $(\alpha_S, \beta)$) and then $L_{n, N}^{(k)} (\theta)  -  L_{n, N}^{(k)} (\widetilde{\theta}) = 0$ for $k = 1, 4$.  We then have that $D (\theta) = D^{(1)} (\theta) + D^{(2)} (\theta)$ with 
$
D^{(1)} (\theta) :=  \tfrac{1}{N} \bigl( L_{n, N}^{(2)} (\theta) -  L_{n, N}^{(2)} (\widetilde{\theta}) \bigr)$,  $D^{(2)} (\theta) :=  \tfrac{1}{N} \bigl( L_{n, N}^{(3)} (\theta) -  L_{n, N}^{(3)} (\widetilde{\theta}) \bigr).  
$ 
From the definition of $L^{(2)}$ in (\ref{eq:contrast_decomp}) and formula (\ref{eq:m_diff}) we have that: 
\begin{align*}
D^{(1)} (\theta) =  & \tfrac{\sqrt{\Delta_n}}{N}  \sum_{i = 1}^N \sum_{j = 1}^n 
\begin{bmatrix}
\mathscr{L}_0 [V_{S, 0} (\alpha_S, \cdot)] (Z_{j-1}^{[i], \trueparam})
- 
\mathscr{L}_0 [V_{S, 0} (\alpha_S, \cdot)] (Z_{j-1}^{[i], \theta}) \\[0.2cm] 
2 b_{R} (\alpha_R, X_{t_{j-1}}^{[i]}, \mu_{t_{j-1}}) 
\end{bmatrix}^\top 
\Lambda_{j-1}^{[i]} (\theta) \, \mathbf{m}_{j-1}^{[i]} (\trueparam) 
\\ 
& \probconv 0, 
\end{align*} 
as $n, N \to \infty$ uniformly in $\theta \in \Theta$, where we have applied Lemma \ref{lemma:key_conv}. 
Similarly, noticing again that the terms that do not  involve $\alpha_R$ cancel out, we have that: 
\begin{align*}
& D^{(2)} (\theta)   
= \tfrac{2}{N} 
\sum_{1 \le i \le N} 
\sum_{1 \le j \le n} 
\begin{bmatrix}
b_S (\alpha_S, \sample{j-1}{i})^\top, \,  
\mathbf{0}_{d_R}^\top 		
\end{bmatrix}  
\Lambda_{j-1}^{[i]} (\theta) \mathcal{B} (Z_{j-1}^{[i], \theta}) 
\\  & +  \tfrac{\Delta_n}{N} \sum_{1 \le i \le N} \sum_{1 \le j \le n} 
\Biggl\{ 
\begin{bmatrix}
\tfrac{1}{2} \mathscr{L}_0 
[V_{S, 0 } (\alpha_S^\dagger, \cdot)] (Z_{j-1}^{[i], \trueparam})
- 
\tfrac{1}{2} \mathscr{L}_0 
[V_{S, 0 } (\alpha_S, \cdot)] (Z_{j-1}^{[i], \theta}) \\[0.2cm] 
b_{R} (\alpha_R, X_{t_{j-1}}^{[i]}, \mu_{t_{j-1}}^N)
\end{bmatrix}^\top \\
& \qquad + 
\begin{bmatrix}
\tfrac{1}{2} \mathscr{L}_0 
[V_{S, 0 } (\alpha_S^\dagger, \cdot)] (Z_{j-1}^{[i], \trueparam})
- 
\tfrac{1}{2} \mathscr{L}_0  
[V_{S, 0 } (\alpha_S, \cdot)] (Z_{j-1}^{[i], \widetilde{\theta}}) \\[0.2cm] 
\mathbf{0}_{d_R}
\end{bmatrix}^\top 
\Biggr\} 
\Lambda_{j-1}^{[i]} (\theta) \mathcal{B} (Z_{j-1}^{[i], \theta}), 
\end{align*}
where we have considered the term:
\begin{align} 
\label{eq:B}
\mathcal{B} (Z_{j-1}^{[i], \theta}) :=
\begin{bmatrix}
\tfrac{1}{2} \partial_{x_R}^\top V_{S, 0} (\alpha_S, X_{t_{j-1}}^{[i]})  \, 
b_R (\alpha_R, \sample{j-1}{i}, \mu_{t_{j-1}}^N ) \\[0.2cm]
b_R (\alpha_R, \sample{j-1}{i}, \mu_{t_{j-1}}^N ) 
\end{bmatrix}.   
\end{align}
We note that the first term of the RHS of $D^{(2)} (\theta)$ is exactly zero due to Lemma \ref{lemma:key_matrix} provided below. Further application of Lemma \ref{lemma:key_matrix} to the second term together with Lemma \ref{lemma:base_conv} yields: 
\begin{align*}
D^{(2)} (\theta)  
& = \tfrac{\Delta_n}{N} \sum_{1 \le i \le N} 
\sum_{1 \le j \le n} 
b_R (\alpha_R, \sample{j-1}{i}, \mu_{t_{j-1}}^N )^\top  
\Sigma_{RR}^{-1} (\beta, \sample{j-1}{i}, \mu_{t_{j-1}}^N ) 
b_R (\alpha_R, \sample{j-1}{i}, \mu_{t_{j-1}}^N )   \\ 
& \probconv  \textstyle \int_0^T \mathbb{E}_{\mu_t} \bigl[ 
b_R (\alpha_R, W, \mu_{t} )^\top   
\Sigma_{RR}^{-1} (\beta, W, \mu_t)  
b_R (\alpha_R, W, \mu_{t} )   \bigr] dt 
\end{align*}  
as $n, N \to \infty$, uniformly in $(\alpha_R, \beta) \in \Theta_{\alpha_R} \times \Theta_\beta$. We thus conclude Lemma \ref{lemma:step_IV}. \qed 

Then, the consistency of $\hat{\alpha}_R^{n,N}$ is shown via Lemma \ref{lemma:step_IV} as follows. \textcolor{black}{First, we have that for any $\varepsilon > 0$, there exists a constant $\delta > 0$ satisfying: for any $(\alpha_R, \beta) \in \Theta_{\alpha_R} \times \Theta_{\beta}, $ if $|\alpha_R - \alpha_R^\dagger| > \varepsilon$ holds, then $S (\alpha_R, \beta) > \delta$. This is due to the fact that the uniform positive definiteness of $\Sigma_{RR}^{-1}$ under Assumption \ref{ass:hor} and that for $\alpha_R \neq \alpha_R^\dagger$, $b (\alpha_R, x, \mu_t)$ is a non-zero vector $\mu_t (dx) \times dt$ a.e. for all $t \in [0, T]$ under Assumption \ref{ass:ident}.  We thus obtain: 
\begin{align*}
\mathbb{P}_{\trueparam} \Bigl( \bigl| \hat{\alpha}_R^{n, N} - \alpha_R^\dagger \bigr| > \varepsilon \Bigr)
& \le \mathbb{P}_{\trueparam} \Bigl( S (\hat{\alpha}_R^{n, N}, \hat{\beta}^{n, N}) > \delta \Bigr) \\
& \le \mathbb{P}_{\trueparam} \Bigl( S (\hat{\alpha}_R^{n, N}, \hat{\beta}^{n, N}) - D (\hat{\theta}^{n, N}) > \delta \Bigr) \to 0 
\end{align*}
as $n, N \to \infty$, where in the last line we used the definition of the contrast estimator $\hat{\theta}^{n, N}$ and Lemma \ref{lemma:step_IV}. 
}

\begin{lemma} \label{lemma:key_matrix}
It holds that for any $\theta \in \Theta$ and $j = 1, \ldots, n$, $i = 1, \ldots, N$:  
\begin{align*}
\Lambda_{j-1}^{[i]} (\theta) \, \mathcal{B} (Z_{j-1}^{[i]}) 
= 
\begin{bmatrix}
\mathbf{0}_{d_S} \\[0.2cm] 
\Sigma_{RR}^{-1} (\beta, \sample{j-1}{i}, \mu_{t_{j-1}}) 
b_R (\alpha_R, \sample{j-1}{i}, \mu_{t_{j-1}} ) 
\end{bmatrix}. 
\end{align*} 
\end{lemma}

\noindent 
\textit{Proof of Lemma \ref{lemma:key_matrix}.} Recall the definition of $\Sigma$ in (\ref{eq:Sigma}). Noticing that $\mathcal{B} (Z_{j-1}^{[i]})$ can be expressed as:
\begin{align*}
\mathcal{B} (Z_{j-1}^{[i]}) = 
\begin{bmatrix}
\Sigma_{SR} (Z_{j-1}^{[i]}) \\[0.1cm]
\Sigma_{RR} (Z_{j-1}^{[i]})
\end{bmatrix} 
\Sigma_{RR}^{-1} (Z_{j-1}^{[i], \theta}) b_R (\alpha_R, \sample{j-1}{i}, \mu_{t_{j-1}} ) 
\end{align*} 
and that: 
\begin{align*} \
\Lambda_{j-1}^{[i]} (\theta) 
\begin{bmatrix}
\Sigma_{SR} (Z_{j-1}^{[i], \theta})  \\[0.1cm] 
\Sigma_{RR} (Z_{j-1}^{[i], \theta}) 
\end{bmatrix}
= 
\begin{bmatrix}
\mathbf{0}_{d_S \times d_R} \\[0.1cm] 
I_{d_R}
\end{bmatrix}, 
\end{align*}
we obtain the desired result. \qed \\ 

\noindent 
Proof of Theorem \ref{thm:consistency} is now complete. 

\subsection{Proof of Theorem \ref{thm:clt}} 
\label{sec:pf_clt}
A Taylor expansion of $\nabla \ell_{n, N} (\hat{\theta}^{n, N})$ around $\nabla \ell_{n, N} (\trueparam) $ yields: 
\begin{align} 
\label{eq:taylor_contrast}
- \nabla \ell_{n, N} (\trueparam)
= \textstyle \int_0^1 \nabla^2  \ell_{n, N} \bigl( \trueparam  + \lambda (\hat{\theta}^{n,N}  - \trueparam) \bigr) d\lambda \, \bigl( \hat{\theta}^{n,N}  - \trueparam \bigr).  
\end{align} 
Multiplying both sides of (\ref{eq:taylor_contrast}) by an appropriate scaling factor,  we obtain: 
\begin{align*}
\mathscr{G}_{n, N} (\trueparam)	= \int_0^1 \mathscr{H}_{n, N} \bigl( \trueparam  + \lambda (\hat{\theta}^{n,N}  - \trueparam) \bigr) d\lambda \times 
\begin{bmatrix}
\sqrt{{N}/{\Delta_n^2}} (\hat{\alpha}_S^{n, N} - \alpha_S^\dagger) \\[0.2cm]
\sqrt{N} (\hat{\alpha}_R^{n, N} - \alpha_R^\dagger)  \\[0.1cm]
\sqrt{{N}/{\Delta_n}} (\hat{\beta}^{n, N} - \beta^\dagger)
\end{bmatrix},
\end{align*}
where we have defined, for $\theta \in \Theta$: 
\begin{align}
\mathscr{G}_{n, N} (\theta) :=  - M_{n, N}  \cdot \nabla \ell_{n, N} (\theta), 
\qquad 
\mathscr{H}_{n, N} (\theta) := M_{n, N} \cdot \nabla^2 \ell_{n, N} (\theta)  \cdot M_{n, N},  
\label{eq:G}
\end{align}
with the scaling term:
$$
M_{n, N} = \mathrm{diag} \Bigl[ \,  \underbrace{\sqrt{{\Delta_n^2}/{N}} \cdots \sqrt{{\Delta_n^2}/{N}}  }_{d_{\alpha_S}}, \, 
\underbrace{\sqrt{{1}/{N}} \cdots \sqrt{{1}/{N}}  }_{d_{\alpha_R}}, \,  
\underbrace{\sqrt{{\Delta_n}/{N}} \cdots \sqrt{{\Delta_n}/{N}}  }_{d_{\beta}}   \,  \Bigr]. 
$$  
We conclude from the following two results, whose proofs are given in Appendices \ref{sec:pf_prop_lln}-\ref{sec:pf_prop_clt} in the Supplementary Material \cite{supp_AoS}. 
\begin{proposition} \label{prop:lln}
Under Assumptions \ref{ass:law_dep}--\ref{ass:ident} it holds that: 
\begin{align*} 
\sup_{\lambda \in [0,1]} \Bigl|  
\mathscr{H}_{n, N} \bigl( \trueparam+ \lambda ( \hat{\theta}^{n, N} - \trueparam)  \bigr)
-  2 \Gamma (\trueparam)  
\Bigr|  \probconv 0,  \qquad n , N \to \infty,   
\end{align*}
where we have set: 
\begin{align} 
\label{eq:precision2}
\Gamma (\trueparam)  
= 
\begin{cases}
\mathrm{diag} \bigl[
\Gamma^{(\mathrm{E})}_{\alpha_R} (\trueparam), \Gamma^{(\mathrm{E})}_\beta (\trueparam) \bigr]  , & d_S = 0; \\[0.2cm] 
\mathrm{diag} \bigl[
\Gamma^{(\mathrm{H})}_{\alpha_S} (\trueparam), \, 
\Gamma^{(\mathrm{H})}_{\alpha_R} (\trueparam),
\,  
\Gamma^{(\mathrm{H})}_\beta (\trueparam) \bigr]  , & d_S \ge  1; \\  
\end{cases}
\end{align} 
\end{proposition} 

\begin{proposition}  \label{prop:clt}
Under Assumptions \ref{ass:law_dep}--\ref{ass:ident} it holds that 
$
\mathscr{G}_{n, N} (\trueparam) \distconv \mathscr{N} \bigl( \mathbf{0}_{d_\theta}, 4 \Gamma (\trueparam) \bigr)$ as 
$n , N \to \infty, \,  N \Delta_n \to 0,
$  
with the matrix $\Gamma (\trueparam)$ as defined in (\ref{eq:precision2}). 
\end{proposition}
\section{Conclusion} \label{sec:conclusion}
We have studied parameter estimation for the class of weakly interacting hypoelliptic diffusions determined in (\ref{eq:ips-1}) and which includes a variety of important models used in applications. E.g., our numerical implementations considered the interacting FitzHugh-Nagumo model (\ref{eq:I-FHN}). In brief, we developed a conditionally non-degenerate Gaussian approximation for the intractable particle dynamics thus forming an analytically available joint likelihood proxy, and we analysed the asymptotic properties of the associated contrast estimator. While inferential procedures relying on the Euler-Maruyama approximation -- a conditionally degenerate scheme -- requires the  specific structure of  $dX_{S, t}^{[i]} = X_{R, t}^{[i]} dt$ under the partial observation regime to recover the latent variables via finite differences, our own development -- a non-degenerate approximate density -- allows for filtering or for incorporation to general Bayesian approaches,  thus offering great flexibility for statistical inference for a very broad class of IPS models. Numerical experiments confirmed the effectiveness of the developed likelihood in both complete and partial observation regimes. The experiments presented in Section \ref{sec:I-UL-num} demonstrated the advantages of the proposed likelihood approach against alternatives in the current literature. 

Potential future research directions include consideration of IPSs of \emph{highly degenerate} diffusions, as studied in \cite{igu:24} within a non-interacting setting. Weakening the design condition of $\Delta_n = o (N^{-1})$ 
on the step-size would also be important for applications, similarly to corresponding recent advances for standard non-interacting ergodic SDEs   \citep{kessler1997estimation, uchi:12, igu:bj, igu:ejs}. Finally, statistical inference in a \emph{low frequency observation regime} for IPSs, i.e.~the case where the 
data step-size is fixed, is also an important direction. The objective here will be the development of accurate closed-form density approximations which move beyond Gaussian proxies, thus extending results obtained in the literature for linear, in the sense of McKean, multivariate SDEs (see e.g.~\cite{ait:08, li:13, IgBe:25}).

\begin{funding}
YI is supported by the EPSRC grant Prob\_AI (EP/Y028783/1). YI and AB acknowledge the support from JST CREST Grant Number JPMJCR2115. GP is partially supported by an ERC-EPSRC Frontier Research Guarantee through Grant No. EP/X038645, ERC Advanced Grant No. 247031 and a Leverhulme Trust Senior Research Fellowship, SRF$\backslash$R1$\backslash$241055.
\end{funding}
\appendix 

\section{Proofs of Key Convergence Results} \label{sec:aux_results} 
This appendix is devoted to proving Lemmas \ref{lemma:base_conv}, \ref{lemma:key_conv} introduced in Section \ref{sec:pre}. We first present two auxiliary results in Appendix \ref{sec:aux} to facilitate showing those Lemmas. Then proofs of Lemmas \ref{lemma:base_conv}, \ref{lemma:key_conv} are provided in Appendices \ref{sec:pf_base_conv} and \ref{sec:pf_key_conv}, respectively. 

We here introduce additional notations. For simplicity, we frequently use $\mathbf{m}_{j-1}^{[i]} \equiv \mathbf{m}_{j-1}^{[i]} (\trueparam)$. We also introduce a set of random variables $\mathcal{S}_t, \, t \in [0,  T]$, as:  
\begin{align*}
\mathcal{S}_t := 
\bigl\{ Y=Y (\Delta)\in \mathbb{R},  \, \Delta > 0  \   \bigl| \,   
\mathcal{F}_t^N\textrm{-}\mathrm{measurable}\  \mathrm{and} \  
\{  \mathbb{E}_{\trueparam} |  Y  (\Delta)|^p \}^{1/p} \lesssim \Delta  \, \mathrm{\, for \, all\,}  \, p \ge 1  \bigr\}. 
\end{align*} 
Finally, we introduce a notation for iterated stochastic integrals. For $q \in \mathbb{N}$, $i_1, \ldots, i_q \in \{1 , \ldots N\}$, $k_1, \ldots, k_q \in \{0, 1, \ldots, d_B \}$ and $j = 1, \ldots, n$: 
\begin{align*}
I_{k_1, \ldots, k_q}^{[i_1, \ldots, i_q]} (j) := 
\int_{t_{j-1}}^{t_j} \int_{t_{j-1}}^{s_q} \cdots \int_{t_{j-1}}^{s_3} \int_{t_{j-1}}^{s_2} d B_{k_1, s_1}^{[i_1]} d B_{k_2, s_2}^{[i_2]} \cdots 
d B_{k_{q-1}, s_{q-1}}^{[i_{q-1}]} d B_{k_q, s_q}^{[i_q]},
\end{align*} 
where we interpret $B_{0, s}^{[i]} \equiv s$ and $\textstyle I_{k_1}^{[i_1]} (j) = \int_{t_{j-1}}^{t_j} d B_{k_1, s}^{[i_1]}$ when $q = 1$.
\subsection{Auxiliary Results} \label{sec:aux}
\begin{lemma} \label{lemma:bds_aux}
Let Assumptions \ref{ass:initial_dist}, \ref{ass:growth} hold. For all $p \ge 1$, all $0 \le s < t \le T$ with $t-s \in (0,1)$ and all $\theta \in \Theta$, we have the following for some constant $C > 0$: 
\begin{align}  
\sup_{t \in [0, T]}  \sup_{i = 1, \ldots, N}
\mathbb{E}_{\theta} 
|X_t^{[i]}|^p   & < C; 
\label{eq:bd_m_ips}  
\\ 
\sup_{t \in [0, T]} \mathbb{E}_{\theta} 
\bigl[ \mathcal{W}_p (\mu_t^N, \delta_0)^q \bigr] & < C, \qquad q \ge p; 
\label{eq:bd_wass_ips}  \\  
\sup_{i = 1, \ldots, N} 
\mathbb{E}_\theta  | X_t^{[i]} - X_{s}^{[i]} |^p & \le C (t-s)^{p/2}.  
\label{eq:bd_t_diff}
\end{align}%
%
\end{lemma}
\begin{remark}
The results collected in Lemma \ref{lemma:bds_aux} are also used in the works \citep{amo:24, amo:23}, where  Lipschitz conditions are assumed on the drift and diffusion coefficients. In our case, the proof of Lemma \ref{lemma:bds_aux} is obtained under the weaker locally Lipschitz condition on the drift function, given in Assumption \ref{ass:growth}. 
\end{remark}
\noindent 
\textit{Proof of Lemma \ref{lemma:bds_aux}.} 
First,  (\ref{eq:bd_m_ips}) follows immediately from \cite[Proposition 2.5]{chen:24} under Assumptions \ref{ass:initial_dist} and \ref{ass:growth}. We now make use of  (\ref{eq:bd_m_ips}) to get (\ref{eq:bd_wass_ips}) since 
$ \textstyle 
\mathbb{E}_{\theta} 
\bigl[ \mathcal{W}_p (\mu_t^N, \delta_0)^q \bigr] 
\le \tfrac{1}{N} \sum_{i = 1}^N \mathbb{E}_\theta |X_t^{[i]}|^q \le C  
$
for some positive constant $C$ independent of $t$ and the index $i$. It remains to show (\ref{eq:bd_t_diff}). We have that: 
\begin{align*}
\mathbb{E}_\theta \big| X_t^{[i]} - X_s^{[i]} \big|^p 
\le C \Bigl( \mathbb{E}_\theta 
\Bigl| \int_s^t V_0 (\alpha, X_u^{[i]}, \mu_u^N) du \Bigr|^p 
+ \sum_{j = 1}^{d_B} \mathbb{E}_\theta 
\Bigl| \int_s^t V_j (\beta, X_u^{[i]}, \mu_u^N) dB_{j, u}^{[i]} \Bigr|^p 
\Bigr). 
\end{align*}
For the first term on the RHS above, we apply H\"older's inequality to obtain: 
\begin{align*}
\mathbb{E}_\theta  
\Bigl| \int_s^t V_0 (\alpha, X_u^{[i]}, \mu_u^N) du \Bigr|^p  \le (t-s)^{p-1} \int_s^t  \mathbb{E}_\theta 
\bigl| V_0 (\alpha, X_u^{[i]}, \mu_u^N) \bigr|^p  du \le C (t-s)^p, 
\end{align*}
where we used Assumption \ref{ass:growth} with the bounds (\ref{eq:bd_m_ips}) and (\ref{eq:bd_wass_ips}). Similarly, Burkholder-Davis-Gundy and H\"older's inequalities yield that for $j = 1, \ldots, d_B$:
\begin{align*}
\mathbb{E}_\theta
\Bigl| \int_s^t V_j (\beta, X_u^{[i]}, \mu_u^N) dB_{j, u}^{[i]} \Bigr|^p 
\le C (t-s)^{\tfrac{p}{2}-1} \times 
\int_s^t \mathbb{E}_\theta
\bigl| V_j (\beta, X_u^{[i]}, \mu_u^N) \bigr|^p du 
\le C (t-s)^{\tfrac{p}{2}},  
\end{align*}
where we used Assumption \ref{ass:growth} with  (\ref{eq:bd_m_ips}) and (\ref{eq:bd_wass_ips}). Thus, we conclude. 
%
%
\begin{lemma} \label{lemma:moments}
Let Assumptions \ref{ass:law_dep}--~\ref{ass:test_function} hold. For any $1 \le j \le n$, $1 \le i_1, i_2, i_3, i_4 \le N$ and $1 \le k_1, k_2, k_3, k_4 \le d$, we have that: 
\begin{align}
& \label{eq:first_moment} \mathbb{E}_{\trueparam} \bigl[ \mathbf{m}_{j-1}^{[i_1], k_1}  \,\big|\,\mathcal{F}_{t_{j-1}}^N \bigr] 
= e_{j-1}^{(1)} (\Delta_n^{3/2});  \\[0.2cm]  
& \mathbb{E}_{\trueparam} 
\bigl[ \mathbf{m}_{j-1}^{[i_1], k_1} \mathbf{m}_{j-1}^{[i_2], k_2}  \,\big|\, \mathcal{F}_{t_{j-1}}^N \bigr] 
= \Sigma_{k_1 k_2} (Z_{j-1}^{[i_1]}) \times \mathbf{1}_{i_1 = i_2}
+ e_{j-1}^{(2)} (\Delta_n),  \label{eq:second_moment} 
\end{align}
\begin{align}
\label{eq:fourth_moment} 
\mathbb{E}_{\trueparam} &\big[ 
\mathbf{m}_{j-1}^{[i_1], k_1} 
\mathbf{m}_{j-1}^{[i_1], k_2}
\mathbf{m}_{j-1}^{[i_3], k_3} 
\mathbf{m}_{j-1}^{[i_4], k_4} \,\big|\, \mathcal{F}_{t_{j-1}}^N \big] 
\\ & \nonumber = \Sigma_{k_1 k_2} (Z_{j-1}^{[i_1]}) \Sigma_{k_3 k_4} (Z_{j-1}^{[i_3]}) \times  \mathbf{1}_{\substack{i_1 = i_2 \\ i_3 = i_4}}  
+ \Sigma_{k_1 k_3} (Z_{j-1}^{[i_1]}) \Sigma_{k_2 k_4} (Z_{j-1}^{[i_2]}) \times  \mathbf{1}_{\substack{i_1 = i_3 \\ i_2 = i_4}} 
\\ & \nonumber \qquad + \Sigma_{k_1 k_4} (Z_{j-1}^{[i_1]}) \Sigma_{k_2 k_3} (Z_{j-1}^{[i_2]}) \times  \mathbf{1}_{\substack{i_1 = i_4 \\ i_2 = i_3}}  + e_{j-1}^{(3)} (\Delta_n), 
%
%
\end{align} 
%
for some $e_{j-1}^{(i)} \in \mathcal{S}_{t_{j-1}}, \, i = 1, 2, 3$.  
\end{lemma}
\begin{remark}
It is critical that (\ref{eq:first_moment}) holds for all coordinates $k$ of each particle. Inclusion of the deterministic $\mathcal{O} (\Delta_n^2) $ term in the drift approximation of the smooth component leads to the above bound. Note that if one excludes such a term from $\mathbf{m}_{S, j-1}^{[i], k_1} (\trueparam)$, then the RHS of (\ref{eq:first_moment}) would instead be $\mathcal{O} (\Delta_n^{1/2})$, for $1 \le k_1 \le d_S$.   
\end{remark}

\noindent 
\textit{Proof of Lemma \ref{lemma:moments}.}
\\
\noindent \textit{Property of $\mathbf{m}_{j-1}^{[i]}$.} First, we study the term $\mathbf{m}_{j-1}^{[i]}$ to show the stated claims. We introduce several notations. For $Z_{j}^{[i], \theta} = (\theta, X_{t_j}^{[i]}, \mu_{t_j}^N), \, 0 \le j \le n$, $1 \le i \le N$ with $\theta = (\alpha_S, \alpha_R, \beta)$, we write
\begin{align}
\begin{aligned}
V_{R, m} (Z_j^{[i], \theta}) & \equiv \mathscr{V}_{R, m}^{(i)} \bigl( \theta, X_{t_j}^{[1]}, \ldots, X_{t_j}^{[N]}); \\ 
\mathscr{L}_m [V_{S, 0} (\alpha_S, \cdot) ](Z_j^{[i], \theta}) & \equiv 
\mathscr{V}_{S,m}^{\mathscr{L}, (i)} \bigl( \theta, X_{t_j}^{[1]}, \ldots, X_{t_j}^{[N]} \bigr),  
\end{aligned} \qquad 0 \le m \le d_B, 
\end{align}
where $\mathscr{V}_{R, m}^{(i)}: \Theta \times (\mathbb{R}^d)^N \to \mathbb{R}^{d_R}$ and $\mathscr{V}_{S, m}^{\mathscr{L}, (i)}: \Theta \times (\mathbb{R}^d)^N \to \mathbb{R}^{d_S}$, which can be fully specified from the definitions of $V_{R, m}$ and $\mathscr{L}_m [V_{S, 0}]$ with (\ref{eq:L}). We write $\mathbb{X}_{t_j}^{-k} = \{ X_{t_j}^{[i]}\}_{i = 1, \ldots, k-1, k+1, \ldots, N}$ for $1 \le k \le N$, and express the mappings $X_{t_j}^{[k]} \mapsto \mathscr{V}_{R, m}^{(i)} (\theta, X_{t_j}^{[1]}, \ldots, X_{t_j}^{[N]})$ and $X_{t_j}^{[k]} \mapsto \mathscr{V}_{S, m}^{\mathscr{L}, (i)} (\theta, X_{t_j}^{[1]}, \ldots, X_{t_j}^{[N]})$ by  $\mathscr{V}_{R, m}^{(i)} (X_{t_j}^{[k]}; \theta, \mathbb{X}_{t_j}^{-k})$ and $\mathscr{V}_{S, m}^{\mathscr{L}, (i)} (X_{t_j}^{[k]}; \theta, \mathbb{X}_{t_j}^{-k})$, respectively. Recall $Z_j^{[i]} \equiv Z_j^{[i], \trueparam}$. Then, under Assumptions \ref{ass:law_dep} \& \ref{ass:test_function}, stochastic Taylor expansion of 
$X_{t_{j}}^{[i], N} \equiv 
\bigl[  (X_{S, t_j}^{[i], N})^\top , (X_{R, t_j}^{[i], N})^\top \bigr]^\top$ yields $\mathbf{m}_{j-1}^{[i]} = C_{j-1}^{[i]} + D_{j-1}^{[i]}$ 
with: 
\begin{align}
C_{j-1}^{[i]} = 
\sum_{m = 1}^{d_B} 
\begin{bmatrix} 
\Delta_n^{-\tfrac{3}{2}} \cdot 
\mathscr{L}_m [V_{S, 0} (\alpha_S^\dagger, \cdot)] 
(Z_{j-1}^{[i]}) \, I_{m, 0}^{[i, i]} (j) \\[0.3cm]
\Delta_n^{-\tfrac{1}{2}} \cdot 
V_{R, m} (Z_{j-1}^{[i]}) \, I_{m}^{[i]} (j) 
\end{bmatrix}, 
\  
D_{j-1}^{[i]} = 
\begin{bmatrix}
\Delta_n^{-\tfrac{3}{2}} \cdot  \rho_S^{[i]} (j) \\[0.3cm]
\Delta_n^{-\tfrac{1}{2}} \cdot  \rho_R^{[i]} (j) 
\end{bmatrix},  \label{eq:decomp_m}
\end{align} 
where: 
%
\begin{gather}
\rho_{S}^{[i]} (j)  
= \sum_{k = 1}^N \sum_{m_1, m_2 = 0}^{d_B} 
\int_{t_{j-1}}^{t_j} \int_{t_{j-1}}^{u} \int_{t_{j-1}}^v  
\mathscr{L}_{m_1} \bigl[ \mathscr{V}_{S, m_2}^{\mathscr{L}, (i)} (\cdot; \trueparam, \mathbb{X}_{w}^{-k}) \bigr] (Z_w^{[k]}) d B_{m_1, w}^{[k]} \, d B_{m_2, v}^{[i]} du; \label{eq:rho_S}  \\ 
\rho_{R}^{[i]} (j)  
 = \sum_{k = 1}^N \sum_{m_1, m_2 = 0}^{d_B} 
\int_{t_{j-1}}^{t_j} \int_{t_{j-1}}^{u} 
\mathscr{L}_{m_1} \bigl[ \mathscr{V}_{R, m_2}^{(i)} (\cdot; \trueparam, \mathbb{X}_{u}^{-k}) \bigr] (Z_u^{[k]}) d B_{m_1, v}^{[k]} \, d B_{m_2, u}^{[i]}.   \label{eq:rho_R} 
\end{gather}
Under Assumptions \ref{ass:law_dep}, \ref{ass:initial_dist}, \ref{ass:growth} and \ref{ass:test_function}, it is shown that: for any $p \ge 2$, there exist $C_1, C_2 > 0$ s.t. for any $k \in 1, \ldots, N$ and $0\le m_1, m_2 \le d_B$,   
\begin{align} 
\sup_{t \in [0, T]} \mathbb{E}_{\trueparam} \Bigl| \mathscr{L}_{m_1} \bigl[ \mathscr{V}_{S, m_2}^{\mathscr{L}, (i)} (\cdot; \theta, \mathbb{X}_{t}^{-k}) \bigr] (Z_t^{[k]}) \Bigr|^p < C_1 (\mathbf{1}_{k = i } + \tfrac{1}{N}); \label{eq:intergrand_S_bd}  \\ 
\sup_{t \in [0, T]}  \mathbb{E}_{\trueparam} 
\Bigl| \mathscr{L}_{m_1} \bigl[ \mathscr{V}_{R, m_2}^{(i)} (\cdot; \theta, \mathbb{X}_{t}^{-k}) \bigr] (Z_t^{[k]}) \Bigr|^p < C_2 (\mathbf{1}_{k = i } + \tfrac{1}{N}).    \label{eq:intergrand_R_bd} 
\end{align}  
We here check (\ref{eq:intergrand_R_bd}), and (\ref{eq:intergrand_S_bd}) is obtained from a similar argument. First, under Assumptions \ref{ass:law_dep}, \ref{ass:initial_dist}, \ref{ass:growth} and \ref{ass:test_function}, we have from Lemma \ref{lemma:bds_aux} that for any $p \ge 2$, there exists a constant $C > 0$ such that for any $\ell_1, \ell_2, \ell_3 = 1, \ldots, d,$
\begin{align} \label{eq:coeff_bd}
\sup_{t \in [0, T]} \sup_{k = 1, \ldots, N} \max \bigl\{ 
\mathbb{E}_{\trueparam} | a_{\ell_1 \ell_2} (Z_t^{[k]}) |^p, \ \ 
\textstyle\sum_{0 \le m \le d_B} \mathbb{E}_{\trueparam} | V_m^{\ell_3} (Z_t^{[k]}) |^p    
  \bigr\} < C, 
\end{align} 
where we used the following estimate: for any $\varphi (\theta, \cdot, \cdot): \mathbb{R}^d \times \mathbb{R}^d \to \mathbb{R}$ having the polyimial growth uniformly in $\theta \in \Theta$ and $q \ge 2$,   
\begin{align}  
\mathbb{E}_{\trueparam} \Bigl| \tfrac{1}{N} \sum_{\ell = 1}^N \varphi (\theta , X_t^{[i]}, X_t^{[\ell]})\Bigr|^q  
& \le \tfrac{1}{N} \sum_{\ell = 1}^N \mathbb{E}_{\trueparam} \Bigl| \varphi (\theta , X_t^{[i]}, X_t^{[\ell]})\Bigr|^q \ \ (\because \mathrm{Jensen's \, inequality})  \nonumber \\ 
&  \le C \quad (\because \mathrm{polynomial \, growth} \  \& \  \mathrm{Lemma \, \ref{lemma:bds_aux}}). 
\end{align}
Furthermore, under Assumptions \ref{ass:law_dep} and \ref{ass:test_function}, we have that: for $e = 1, \ldots, d$,  
\begin{align*}
& \partial_e \{ \mathscr{V}_{R, m_2}^{(i)} (\cdot; \theta, \mathbb{X}_t^{-k}) \} (X_t^{[k]}) \\  
& = \Bigl\{ \partial_e \{ V_{R, m_2}^I (\theta, \cdot)\}(X_t^{[i]}) 
+ \tfrac{1}{N} \sum_{\ell = 1}^N \partial_e
\bigl\{ V_{R, m}^{II} (\theta, \cdot, X_t^{[\ell]}) \bigr\} (X_t^{[i]})
 \Bigr\} \times \mathbf{1}_{k = i} \\ 
& \qquad \qquad  + \tfrac{1}{N} \partial_e 
\bigl\{ V_{R, m}^{II} (\theta, X_t^{[i]}, \cdot) \bigr\} (X_t^{[k]}). 
\end{align*} 
and thus, 
\begin{align} \label{eq:first_deriv_bd}
\sup_{t \in [0, T]} \mathbb{E}_{\trueparam}
\bigl| \partial_e \{ \mathscr{V}_{R, m_2}^{(i)} (\cdot; \theta, \mathbb{X}_t^{-k}) \} (X_t^{[k]})  \bigr|^p \le C (\mathbf{1}_{k = i} + \tfrac{1}{N}). 
\end{align} 
Similarly, we have that: for any $1 \le e_1, e_2 \le d$ and $p \ge 2$, there exists $C >0$ s.t. for any $k = 1, \ldots, N$, 
\begin{align} \label{eq:second_deriv_bd}
\sup_{t \in [0, T]} \mathbb{E}_{\trueparam}
\bigl| \partial_{e_1}  \partial_{e_2}  \{ \mathscr{V}_{R, m_2}^{(i)} (\cdot; \theta, \mathbb{X}_t^{-k}) \} (X_t^{[k]})  \bigr|^p \le C (\mathbf{1}_{k = i} + \tfrac{1}{N}). 
\end{align}  
From (\ref{eq:coeff_bd}), (\ref{eq:first_deriv_bd}) and (\ref{eq:second_deriv_bd}), we conclude (\ref{eq:intergrand_R_bd}). 
\\ 

\noindent
\textit{Proof of (\ref{eq:first_moment})}.  
Due to (\ref{eq:intergrand_S_bd}) and (\ref{eq:intergrand_R_bd}), the stochastic integrals in $\rho_S^{[i]} (j)$ and $\rho_R^{[i]} (j)$ are true martingales. Noting that those  martingales are $0$ under the conditional expectation $\mathbb{E}_{\trueparam} [\cdot | \mathcal{F}_{t_{j-1}}^N]$, we get that 
%
\begin{align}
& \mathbb{E}_{\trueparam} 
\bigl| 
\mathbb{E}_{\trueparam} \bigl[ \rho_{S}^{[i]} (j) | \mathcal{F}_{t_{j-1}}^N \bigr]   \bigr|^p \nonumber  \\   
& \qquad 
\le \Delta_n^{3p - 3} \sum_{k = 1}^N 
\int_{t_{j-1}}^{t_j}  \int_{t_{j-1}}^u \int_{t_{j-1}}^v 
\mathbb{E}_{\trueparam} \Bigl|
\mathscr{L}_{0} \bigl[ \mathscr{V}_{S, 0}^{\mathscr{L}, (i)} (\cdot; \trueparam, \mathbb{X}_{w}^{-k}) \bigr] (Z_w^{[k]})  \Bigr|^p  dw dv du  \nonumber \\ 
& \qquad \le C \Delta_n^{3p} \label{eq:m_S_bd} 
\end{align}
%
%
where we have changed the order of integration and applied Jensen's inequality iteratively and then used (\ref{eq:intergrand_S_bd}) in the last inequality. 
%
%
Similar arguments as above together with (\ref{eq:intergrand_R_bd}) yield that: for any $p \ge 2$, there exists a constant $C > 0$ such that  
\begin{align}
\mathbb{E}_{\trueparam} 
\bigl| 
\mathbb{E}_{\trueparam} \bigl[ \rho_{R}^{[i]} (j) | \mathcal{F}_{t_{j-1}}^N \bigr]   \bigr|^p \le C \Delta_n^{2p}.  
\label{eq:m_R_bd} 
\end{align}  	
Noticing that $\mathbb{E} [ C_{j-1}^{[i]} | \mathcal{F}_{t_{j-1}}^N] = \mathbf{0}_d$ due to $I_{m, 0}^{[i, i]} (j)$ and $I_m^{[i]} (j)$,  we conclude (\ref{eq:first_moment}) from (\ref{eq:m_S_bd}) and (\ref{eq:m_R_bd}) with (\ref{eq:decomp_m}). \\ 

\noindent 
\textit{Proof of (\ref{eq:second_moment})}.  Noticing that: 
\begin{gather*} 
\mathbb{E} \bigl[ I_{k_1}^{[i_1]} (j) I_{k_2}^{[i_2]} (j)  | \mathcal{F}_{t_{j-1}}^N \bigr] = \Delta_n \cdot \mathbf{1}_{\substack{i_1 = i_2 \\ k_1 = k_2}}, \quad 
\mathbb{E} \bigl[ I_{k_1}^{[i_1]} (j) I_{k_2, 0}^{[i_2, i_2]} (j)  | \mathcal{F}_{t_{j-1}}^N \bigr] = \tfrac{\Delta_n^2}{2} \cdot \mathbf{1}_{\substack{i_1 = i_2 \\ k_1 = k_2}}; \\[0.2cm] 
\mathbb{E} \bigl[ I_{k_1, 0}^{[i_1]} (j) I_{k_2, 0}^{[i_2]} (j)  | \mathcal{F}_{t_{j-1}}^N \bigr] = \tfrac{\Delta_n^3}{3} \cdot \mathbf{1}_{\substack{i_1 = i_2 \\ k_1 = k_2}},   
\end{gather*} 
we have from (\ref{eq:decomp_m}) that: 
\begin{align}
& \mathbb{E}_{\trueparam} 
\bigl[ \mathbf{m}_{j-1}^{[i_1], k_1} \mathbf{m}_{j-1}^{[i_2], k_2} | \mathcal{F}_{t_{j-1}}^N \bigr]  
= \Sigma_{k_1 k_2} (Z_{j-1}^{[i_1]}) \cdot \mathbf{1}_{i_1 = i_2}  
\label{eq:second_m_pf}  
\\
& \qquad \nonumber   + 
\mathbb{E}_{\trueparam} 
\bigl[ C_{j-1}^{[i_1], k_1} D^{[i_2], k_2}_{j-1} +  D_{j-1}^{[i_1], k_1} C^{[i_2], k_2}_{j-1}
+ D_{j-1}^{[i_1], k_1} D^{[i_2], k_2}_{j-1} | \mathcal{F}_{t_{j-1}}^N \bigr].  
\end{align}
It follows from (\ref{eq:rho_S}), (\ref{eq:rho_R}) and Lemma \ref{lemma:bds_aux} that for any $1 \le j \le n$ and $1 \le i \le N$, $\mathbb{E}_{\trueparam}  | D_{j-1}^{[i], k} |^p  \le C_p \, \Delta_n^{p/2}$, $p \ge 1$. We then use Jensen and Cauchy-Schwarz inequalities to get that for any $p \ge 1$, 
$
\mathbb{E}_{\trueparam} 
\bigl| \mathbb{E}_{\trueparam} \bigl[  D_{j-1}^{[i_1], k_1} D_{j-1}^{[i_2], k_2}  | \mathcal{F}_{t_{j-1}}^N  \bigr] \bigr|^p   \le C_p \, \Delta_n^p.  
$ 
Subsequently, we study the rest of terms in the RHS of (\ref{eq:second_m_pf}). A further stochastic Taylor expansion for the terms (\ref{eq:rho_S}) and (\ref{eq:rho_R}) with Lemma \ref{lemma:bds_aux} lead to: it holds under Assumptions \ref{ass:law_dep}, \ref{ass:initial_dist}, \ref{ass:growth}, \ref{ass:test_function} that $\rho_S^{[i]} (j) = \widetilde{\rho}_S^{[i]} (j) + \mathcal{E}_{S}^{[i]} (j)$ and $\rho_R^{[i]} (j) = \widetilde{\rho}_R^{[i]} (j) + \mathcal{E}_{R}^{[i]} (j)$ for some $\mathcal{E}_{S}^{[i]}(j), \mathcal{E}_{R}^{[i]} (j)$ such that $\mathbb{E}_{\trueparam}  | \mathcal{E}_{S}^{[i]}(j) |^p  \le C \Delta_n^{5p/2}$ and $\mathbb{E}_{\trueparam} | \mathcal{E}_{R}^{[i]}(j) |^p  \le C \Delta_n^{3p/2}$, $p \ge 1$, where:
%
%
\begin{align*}
\widetilde{\rho}_{S}^{[i]} (j)  
& = \sum_{k = 1}^N \sum_{m_1, m_2 = 0}^{d_B} 
\mathscr{L}_{m_1} \bigl[ \mathscr{V}_{S, m_2}^{\mathscr{L}, (i)} (\cdot; \trueparam, \mathbb{X}_{t_{j-1}}^{-k}) \bigr] (Z_{j-1}^{[k]}) 
I_{m_1, m_2, 0}^{[k, i, i]} (j); \\ 
\widetilde{\rho}_{R}^{[i]} (j)  
& = \sum_{k = 1}^N \sum_{m_1, m_2 = 0}^{d_B} 
\mathscr{L}_{m_1} \bigl[ \mathscr{V}_{R, m_2}^{(i)} (\cdot; \trueparam, \mathbb{X}_{t_{j-1}}^{-k}) \bigr] (Z_{j-1}^{[k]}) I_{m_1, m_2}^{[k, i]}(j).   
\end{align*}
Writing: 
\begin{align}  
\label{eq:tilde_D}
\widetilde{D}^{[i]}_{j-1} = \bigl [ \Delta_n^{- 3/2} \cdot \widetilde{\rho}_S^{[i]} (j)^\top,    \Delta_n^{- 1/2} \cdot   \widetilde{\rho}_R^{[i]} (j)^\top   \bigr]^\top \in \mathbb{R}^d, 
\end{align}  we have that: 
\begin{align*}
\mathbb{E}_{\trueparam} 
\bigl[ C_{j-1}^{[i_1], k_1} D_{j-1}^{[i_2], k_2} | \mathcal{F}_{t_{j-1}}^N \bigr]
&=      \mathbb{E}_{\trueparam} 
\bigl[ C_{j-1}^{[i_1], k_1} \widetilde{D}_{j-1}^{[i_2], k_2} | \mathcal{F}_{t_{j-1}}^N \bigr] + \mathcal{E}_{j-1} (\Delta_n) \\ &= \mathcal{\widetilde{E}}_{j-1} (\Delta_n) + \mathcal{E}_{j-1} (\Delta_n)
\end{align*} 
for $\mathcal{E}_{j-1},  \widetilde{\mathcal{E}}_{j-1} \in \mathcal{S}_{t_{j-1}}$, where we used (\ref{eq:intergrand_S_bd})-(\ref{eq:intergrand_R_bd}) and the following estimates: (see e.g. \citep[Lemma 2.1.5]{mil:21}) for $j = 1, \ldots, n$: 
\begin{align*}
\mathbb{E} \bigl[ I_{m_1}^{[i_1]} (j) \, I_{m_2, m_3, 0}^{[i_2, i_3, i_4]} (j) |  \mathcal{F}_{t_{j-1}}^N \bigr] & = 0, \quad  \mathbb{E} \bigl[ I_{m_1}^{[i_1]} (j) \, I_{m_2, m_3}^{[i_2, i_3]} (j) |  \mathcal{F}_{t_{j-1}}^N \bigr] = 0; \\ 
\mathbb{E} \bigl[ I_{m_1, 0}^{[i_1, i_1]} (j) \, I_{m_2, m_3, 0}^{[i_2, i_3, i_4]} (j) |  \mathcal{F}_{t_{j-1}}^N \bigr] & = 0, \quad  \mathbb{E} \bigl[ I_{m_1, 0}^{[i_1, i_1]} (j) \, I_{m_2, m_3}^{[i_2, i_3]} (j) |  \mathcal{F}_{t_{j-1}}^N \bigr] = 0, 
\end{align*} 
for any $m_1, m_2, m_3 \in \{1, \ldots, d_B \}$ and $i_1, i_2, i_3, i_4 \in \{1, \ldots, N \}$. We now conclude (\ref{eq:second_moment}).  \\ 

\noindent 
\textit{Proof of (\ref{eq:fourth_moment}).} Due to the decomposition $\mathbf{m}_{j-1}^{[i]} = C_{j-1}^{[i]} + D_{j-1}^{[i]}$, specified in (\ref{eq:decomp_m}), the LHS of (\ref{eq:fourth_moment}) is expressed as a summation of the following terms: (i) $ \textstyle \mathbb{E}_{\trueparam} \bigl[ \prod_{q = 1}^4 C_{j-1}^{[i_q], k_q} | \mathcal{F}_{t_{j-1}}^N \bigr]$; (ii) only one $D_{j-1}^{[i_q], k_q}$ is involved in the product; (iii) otherwise. For the term (i), we have from the correlation between Gaussian increments that: 
\begin{align*}
& \mathbb{E}_{\trueparam} \bigl[ \prod_{q = 1}^4 C_{j-1}^{[i_q], k_q} | \mathcal{F}_{t_{j-1}}^N \bigr] 
= \Sigma_{k_1 k_2} (Z_{j-1}^{[i_1]}) \Sigma_{k_3 k_4} (Z_{j-1}^{[i_3]}) \mathbf{1}_{\substack{i_1 = i_2, i_3 = i_4}} \\ & \qquad + \Sigma_{k_1 k_3} (Z_{j-1}^{[i_1]}) \Sigma_{k_2 k_4} (Z_{j-1}^{[i_2]}) \mathbf{1}_{\substack{i_1 = i_3, i_2 = i_4}}  
+  \Sigma_{k_1 k_4} (Z_{j-1}^{[i_1]}) \Sigma_{k_2 k_3} (Z_{j-1}^{[i_2]}) \mathbf{1}_{\substack{i_1 = i_4, i_2 = i_3}}.    
\end{align*}
For the last case (iii), we first note that $\| C_{j-1}^{[i], k} \|_{L_p}$ and $\| D_{j-1}^{[i], k} \|_{L_p}, \, p \ge 1,$ are $\mathcal{O} (1)$ and $\mathcal{O} (\Delta_n^{1/2})$, respectively and each term in case (iii) includes at least two $D_{j-1}$. We then conclude from Cauchy-Schwartz inequality that the terms in case (iii) are all evaluated as $\mathcal{O} (\Delta_n)$ in $L_1$-norm.  

For the second case (ii), wlog we consider $\textstyle \mathbb{E}_{\trueparam} [ D_{j-1}^{[i_1], k_1} \prod_{q = 2}^4 C_{j-1}^{[i_q], k_q}    | \mathcal{F}_{t_{j-1}}^N]$. 
Following the argument in the proof of (\ref{eq:second_moment}), we have that $D_{j-1}^{[i_1], k_1} = \widetilde{D}_{j-1}^{[i_1], k_1} + e_{j-1} (\Delta_n)$ for some $e_{j-1} \in \mathcal{S}_{t_{j-1}}$, where $\widetilde{D}_{j-1}^{[i_1], k_1}$ is specified as in (\ref{eq:tilde_D}). We then have that 
$$
\textstyle \mathbb{E}_{\trueparam} [ \widetilde{D}_{j-1}^{[i_1], k_1} \prod_{q = 2}^4 C_{j-1}^{[i_q], k_q} | \mathcal{F}_{t_{j-1}}^N] = e_{j-1} (\Delta_n)
$$ for $e_{j-1} \in \mathcal{S}_{t_{j-1}}$ by noticing that for any $i_2, i_3, i_4, \ell_1, \ell_2 \in \{1, \ldots, N\}$, $m_1, m_2 \in \{1, \ldots, d_B \}$, $k_2, k_3, k_4 \in \{1, \ldots, d \}$ and $j = 1, \ldots, n$:
\begin{align*}
\mathbb{E}_{\trueparam} \bigl[ I_{m_1, m_2}^{[\ell_1, \ell_2]} (j) \prod_{q = 2}^4 C_{j-1}^{[i_q], k_q} | \mathcal{F}_{t_{j-1}}^N \bigr] = 0, \qquad 
\mathbb{E}_{\trueparam} \bigl[ I_{m_1, m_2, 0}^{[\ell_1, \ell_2, \ell_2]} (j) \prod_{q = 2}^4 C_{j-1}^{[i_q], k_q} | \mathcal{F}_{t_{j-1}}^N \bigr] = 0.  
\end{align*} 
Thus, we have that $\textstyle \mathbb{E}_{\trueparam} [ D_{j-1}^{[i_1], k_1} \prod_{q = 2}^4 C_{j-1}^{[i_q], k_q}    | \mathcal{F}_{t_{j-1}}^N] = e_{j-1} (\Delta_n)$ for $e_{j-1} \in \mathcal{S}_{t_{j-1}}$ and then conclude the terms in case (ii) are characterised as $\mathcal{O} (\Delta_n)$ in 
$L_1$ -norm. Proof of (\ref{eq:fourth_moment}) is now complete. 

We thus conclude Lemma \ref{lemma:moments}.  
%
%
\subsection{Proof of Lemma \ref{lemma:base_conv}} \label{sec:pf_base_conv}

We first show (\ref{eq:conv_aux}) for each $\theta \in \Theta$. We follow the strategy used in the proof \citep[Lemma 5.2]{amo:23} but work under the non-globally Lipschitz condition on the drift function (Assumption \ref{ass:growth}). We have that
$ \textstyle 
\tfrac{\Delta_n}{N} \sum_{i = 1}^N \sum_{j = 1}^n f (\theta, X_{t_{j-1}}^{[i]}, \mu_{t_{j-1}}^N) 
- 
\int_0^T \mathbb{E}_{\mu_t} 
\bigl[ f (\theta, W, \mu_t) \bigr] dt 
= F_1 + F_2 + F_3 
$
with: 
\begin{align*}
F_1 & = \tfrac{\Delta_n}{N} \sum_{i = 1}^N \sum_{j = 1}^n f (\theta, X_{t_{j-1}}^{[i]}, \mu_{t_{j-1}}^N) 
- \tfrac{1}{N} \sum_{i = 1}^N  \int_0^T 
f (\theta, X_{t}^{[i]}, \mu_{t}^N) dt; \\  
F_2 & = \tfrac{1}{N} \sum_{i = 1}^N  \int_0^T 
f (\theta, X_{t}^{[i]}, \mu_{t}^N) dt - 
\tfrac{1}{N} \sum_{i = 1}^N \int_0^T f (\theta, \bar{X}_t^i, \mu_t) dt;  \\ 
F_3 & = \tfrac{1}{N} \sum_{i = 1}^N \int_0^T f (\theta, \bar{X}_t^i, \mu_t) dt -  
\int_0^T \mathbb{E}_{\mu_t} \bigl[ f  (\theta, W, \mu_t) \bigr] dt, 
\end{align*} 
where $\bar{X}_t^{i}$ is the solution of McKean-Vlasov SDE (\ref{eq:mv_I}) with the Brownian motion $B$ replaced by $B^{[i]}$ and ${X}_0$ by $X_0^{[i]}$. We will show $F_k \probconv 0$, $k = 1, 2, 3$ as $n, N \to \infty$. $F_3 \probconv 0$ is immediately obtained from the law of large numbers. For the first term $F_1$, we have from the property (\ref{eq:f_bd}) of $f$ that: 
\begin{align*}
& \mathbb{E}_{\trueparam} |F_1| 
\le \tfrac{1}{N} \sum_{i = 1}^N \sum_{j = 1}^n 
\int_{t_{j-1}}^{t_j} 
\mathbb{E}_{\trueparam} 
\bigl| 
f (\theta, X_{t_{j-1}}^{[i]}, \mu_{t_{j-1}}^N) 
- f (\theta, X_{t}^{[i]}, \mu_{t}^N) 
\bigr|  dt \\
& \le C \tfrac{1}{N} 
\sum_{i = 1}^N \sum_{j = 1}^n 
\int_{t_{j-1}}^{t_j} 
\Bigl\{ \bigl\| 
| X_{t_{j-1}}^{[i]} - X_{t}^{[i]} |
+ \mathcal{W}_2 (\mu_{t_{j-1}}^N, \mu_{t}^N) 
\bigr\|_2 \\ 
& \qquad 
\times 
\bigl\| 1 + |X_{t_{j-1}}^{[i]}|^{\ell_1} 
+ | X_{t}^{[i]}  |^{\ell_1} 
+ \mathcal{W}_{2}  (\mu_{t_{j-1}}^N, \delta_0)^{\ell_2} 
+ \mathcal{W}_{2}  (\mu_{t}^N, \delta_0)^{\ell_2}  
\bigr\|_2 \Bigr\} dt \\[0.2cm] 
& \le C n \Delta_n \times \Delta_n^{1/2} =  C T \Delta_n^{1/2},
\end{align*}
where we have applied the Cauchy-Schwartz inequality and Lemma \ref{lemma:bds_aux} -- the finite moments of $X^{[i]}_t$ uniformly in $t \in [0, T]$ and $\| X_{t_{j-1}}^{[i]} - X_t^{[i]} \|_2 \le C \Delta_n^{1/2}$ for all $t \in [t_{j-1}, t_j]$ and $1 \le j \le n$. Thus, $\mathbb{E}_{\trueparam} |F_1|\to 0$ as $n \to \infty$ and then $F_1 \probconv 0$. Similarly, for the second term $F_2$, we have that:
\begin{align*}
& \mathbb{E}_{\trueparam}  |F_2| 
\le \tfrac{1}{N} 
\sum_{i = 1}^N \int_0^T \mathbb{E}_{\trueparam}  
\bigl| 
f (\theta, X_t^{[i]}, \mu_t^N) - f (\theta, \bar{X}_t^{i}, \mu_t)
\bigr|  dt \\ 
& \le C \tfrac{1}{N} 
\sum_{i = 1}^N \int_0^T 
\Bigl\{ \bigl\| 
| X_{t}^{[i]} - \bar{X}_{t}^{i} |
+ \mathcal{W}_2 (\mu_{t}^N, \mu_{t}) 
\bigr\|_2 \\ 
& \qquad \qquad \qquad \qquad \qquad 
\times 
\bigl\| 1 + |X_{t}^{[i]}|^{\ell_1} 
+ | \bar{X}_{t}^{i}  |^{\ell_1} 
+ \mathcal{W}_{2} (\mu_{t}^N, \delta_0)^{\ell_2}
+ \mathcal{W}_{2} (\mu_{t}, \delta_0)^{\ell_2} 
\bigr\|_2 \Bigr\} dt \\ 
& \le C T N^{- q(d)}, 
\end{align*}
where $q (d)$ is a positive number depending on the state dimension $d$ of each particle. In the last inequality, we made use of the following three results: Lemma \ref{lemma:bds_aux}, the time-uniform finite moments bound of McKean-Vlasov SDE under Assumptions \ref{ass:initial_dist}, \ref{ass:growth} \citep[Theorem 3.3]{dos:19}, the \emph{propagation of chaos (PoC)} under Assumptions \ref{ass:initial_dist}, \ref{ass:growth} \citep[Proposition 3.1]{dos:22}, specifically,   
$
\sup_{1 \le i \le N} \mathbb{E} \bigl[ \sup_{t \in [0, T]} | X_t^{[i]} -  \bar{X}_t^i |^2 \bigr] 
\le C \times R (N, d), 
$ 
with 
\begin{align*}
R (N, d) = 
\begin{cases}
N^{-1/2} & d < 4; \\ 
N^{-1/2} \log (N) & d = 4; \\ 
N^{- 2/d} & d > 4.    
\end{cases} 
\end{align*}
Also, in the above, to bound the term $\mathcal{W}_2 (\mu_t^N, \mu_t)$, we considered: 
$ \| \mathcal{W}_2 (\mu_t^N, \mu_t) \|_2 \le \| \mathcal{W}_2 (\mu_t^N, \bar{\mu}^N_t) \|_2  + \| \mathcal{W}_2 (\bar{\mu}_t^N, \mu_t) \|_2$, where $\bar{\mu}^N_t$ is the empirical measure from $N$-independent samples $\{ \bar{X}_t^i \}_i$. Then, we have from PoC that $\textstyle \mathbb{E}_{\trueparam} [ \mathcal{W}_2 (\mu_t^N, \bar{\mu}_t^N)^2] \le \tfrac{1}{N} \sum_{1 \le i \le N} 
\mathbb{E}_{\trueparam} |X_t^{[i]} - \bar{X}_t^i |^2 \le C \times R (N, d)$ for some constant $C > 0$ independent of $t \in [0, T]$. We also obtain that $\mathbb{E}_{\trueparam} [\mathcal{W}_2 (\bar{\mu}_t^N, \mu_t)^2] \le C \times R (N, d)$ by applying the result \cite[Theorem 5.8]{carmona:18} with the time-uniform finite moments of McKean-Vlasov SDE. Thus, we have $\mathbb{E}_{\trueparam} |F_2| 
\probconv 0$, and we conclude the pointwise convergence of (\ref{eq:conv_aux}). To prove the uniform convergence w.r.t $\theta \in \Theta$, it sufficies to show that the sequence of $C(\Theta; \mathbb{R})$-valued random variable 
$ \textstyle 
\bigl\{ \mathcal{S}_{n, N} (\theta) \bigr\}_{n, N} \equiv 
\bigl\{  \tfrac{\Delta_n}{N} \sum_{i = 1}^N \sum_{j = 1}^n f (\theta, X_{t_{j-1}}^{[i]}, \mu_{t_{j-1}}^N) \bigr\}_{n, N} 
$ is tight. Under the property (\ref{eq:R_3}) of the function $f \in \mathcal{R}$ together with Lemma \ref{lemma:bds_aux}, it holds that for all $n, N$: 
\begin{align*}
\mathbb{E}_{\trueparam} 
\Bigl[ \sup_{\theta \in \Theta}
\bigl| \nabla_\theta \mathcal{S}_{n, N} (\theta)  \bigr|  \Bigr]
\le \tfrac{\Delta_n}{N} \sum_{i = 1}^N \sum_{j = 1}^n 
\mathbb{E}_{\trueparam} \Bigl[ \sup_{\theta \in \Theta} 
\bigl|  \nabla_\theta f (\theta, X_{t_{j-1}}^{[i]}, \mu_{t_{j-1}}^N) \bigr| \Bigr] \le T C, 
\end{align*}
for some $C \ge 0$, which is  sufficient for $\bigl\{ \mathcal{S}_{n, N} (\theta) \bigr\}$ being a tight sequence in the metric space $\bigl( C(\Theta; \mathbb{R}), \, \| \cdot \|_\infty  \bigr) $.  We thus conclude Lemma \ref{lemma:base_conv}. 
\subsection{Proof of Lemma \ref{lemma:key_conv}}  \label{sec:pf_key_conv}
We first show (\ref{eq:aux_conv_1}) and (\ref{eq:aux_conv_2}) for fixed $\theta$, i.e. pointwise convergence. To get~(\ref{eq:aux_conv_1}), for fixed $\theta$, we rewrite the LHS of (\ref{eq:aux_conv_1}) as $\textstyle \sum_{j = 1}^n G_{j}^{N, k_1} (\theta)$ with 
$ \textstyle 
G_{j}^{N, k_1} (\theta) := 
\tfrac{\sqrt{\Delta_n}}{N}  \sum_{i = 1}^N  
g (Z_{j-1}^{[i], \theta}) \mathbf{m}_{j-1}^{[i], k_1}.  
$
\, 
$G_j^{N, k_1} (\theta)$ is $\mathcal{F}_{t_j}^N$-measurable, so it suffices to show from \citep[Lemma 9]{genon:93} that 
\begin{align*}
\sum_{j = 1}^n \mathbb{E}_{\trueparam} [ G_j^{N, k_1} (\theta) | \mathcal{F}_{t_{j-1}}^N ] \probconv 0, \quad  
\sum_{j = 1}^n \mathbb{E}_{\trueparam} [ (G_j^{N, k_1} (\theta))^2 | \mathcal{F}_{t_{j-1}}^N ] \probconv 0.  
\end{align*}
Indeed, we have from Lemmas \ref{lemma:base_conv} and \ref{lemma:moments} that, for any $\theta \in \Theta$:  
\begin{align*}
& \sum_{j = 1}^n \mathbb{E}_{\trueparam} \bigl[ G_j^{N, k_1} (\theta) \,\big|\, \mathcal{F}_{t_{j-1}}^N \bigr] 
=\tfrac{\sqrt{\Delta_n}}{N} 
\sum_{i = 1}^N \sum_{j = 1}^n g (Z_{j-1}^{[i], \theta}) \times e_{j-1}^{(1), k_1} (\Delta_n^{3/2}) \probconv 0; \\[0.2cm] 
& \sum_{j = 1}^n \mathbb{E}_{\trueparam} \bigl[ \bigl( G_j^{N, k_1} (\theta) \bigr)^2 \,\big|\, \mathcal{F}_{t_{j-1}}^N \bigr] 
= \tfrac{1}{N} \times \tfrac{\Delta_n}{N} 
\sum_{i = 1}^N \sum_{j = 1}^n g (Z_{j-1}^{[i], \theta}) \times \Sigma_{k_1 k_1} (Z_{j-1}^{[i]})
\\
& \qquad \qquad \qquad \qquad 
+ \tfrac{\Delta_n}{N^2} 
\sum_{\substack{1 \le  i_1, i_2 \le N \\ i_1 \neq i_2}} \sum_{j = 1}^n g (Z_{j-1}^{[i_1], \theta}) g (Z_{j-1}^{[i_2], \theta}) \times e_{j-1}^{(2), k_1} (\Delta_n) \probconv 0,  
\end{align*} 
where $e_{j-1}^{(\iota), k_1} \in \mathcal{S}_{t_{j-1}}, \, \iota = 1, 2$. We next prove  (\ref{eq:aux_conv_2}) for any fixed $\theta$. 
We write the LHS of (\ref{eq:aux_conv_2}) as 
$ \textstyle \sum_{j = 1}^n H_{j}^{N, k_1 k_2} (\theta)
$
with 
$ \textstyle 
H_j^{N, k_1 k_2} (\theta) := \tfrac{\Delta_n}{N}  \sum_{i = 1}^N g (Z_{j-1}^{[i], \theta})  \,   \mathbf{m}_{j-1}^{[i], k_1}  
\, \mathbf{m}_{j-1}^{[i], k_2}.  
$
Then, Lemmas \ref{lemma:base_conv}-\ref{lemma:moments} give that: 
\begin{align*}
\sum_{j = 1}^n \mathbb{E}_{\trueparam} &\big[ H_j^{N, k_1 k_2} (\theta) \,\big|\, \mathcal{F}_{t_{j-1}}^N \big] 
\\[-0.3cm] &= \tfrac{\Delta_n}{N}  \sum_{i = 1}^N \sum_{j = 1}^n
g (Z_{j-1}^{[i], \theta})  \Sigma_{k_1 k_2} (Z_{j-1}^{[i]})
+ \tfrac{\Delta_n}{N}  \sum_{i = 1}^N \sum_{j = 1}^n 
g (Z_{j-1}^{[i], \theta}) e_{j-1}^{(2), k_1, k_2} (\Delta_n)  \\ 
& \qquad \qquad \probconv \int_0^T \mathbb{E}_{\mu_t} 
\bigl[ g (\theta, W, \mu_t) \Sigma_{k_1 k_2} (\trueparam, W, \mu_t) \bigr] dt;
\end{align*}
\begin{align*}
& \sum_{j = 1}^n \mathbb{E}_{\trueparam} \bigl[ \bigl( H_j^{N, k_1 k_2} (\theta) \bigr)^2 \,\big|\, \mathcal{F}_{t_{j-1}} \bigr] \\[-0.2cm] &\,\, = \tfrac{\Delta_n^2}{N^2} \sum_{i_1, i_2 = 1}^N \sum_{j = 1}^n
g(Z_{j-1}^{[i_1], \theta}) g (Z_{j-1}^{[i_2], \theta})  \times  
\mathbb{E}_{\trueparam} \bigl[ \mathbf{m}_{j-1}^{[i_1], k_1} \mathbf{m}_{j-1}^{[i_1], k_2} \mathbf{m}_{j-1}^{[i_2], k_1} \mathbf{m}_{j-1}^{[i_2], k_2}
| \mathcal{F}_{t_{j-1}}^N \bigr] \probconv 0, 
\end{align*}
where $e_{j-1}^{(2), k_1, k_2} \in \mathcal{S}_{t_{j-1}}$. Uniform convergence then follows from the tightness of random sequences (LHS of (\ref{eq:aux_conv_1}), 
(\ref{eq:aux_conv_2})) in $\bigl( C (\Theta; \mathbb{R}), \| \cdot \|_\infty \bigr)$ given the properties of $g\in\mathcal{R}$ and Assumptions \ref{ass:growth}--\ref{ass:reg_coeff}. The proof of Lemma \ref{lemma:key_conv} is now complete. 
\clearpage 

\begin{center}
SUPPLEMENTARY MATERIAL
\end{center}

This material contains proofs for Theorem \ref{thm:clt} - the Central Limit Theorem (CLT) of the proposed contrast estimator in Section \ref{sec:pf_clt} in the main text. Appendices \ref{sec:pf_prop_lln} \& \ref{sec:pf_prop_clt} show Propositions \ref{prop:lln} \& \ref{prop:clt}, respectively, used to prove the CLT. We stress that the definition of notations used in this material basically follows those in the main text, in particular, provided in Sections \ref{sec:intro} \& \ref{sec:pre} and Appendix \ref{sec:aux_results}. Notice that all references noted with Arabic numerals (e.g., Proposition 1 and equation (1)) correspond to those in the main text.

\section{Proof of Proposition \ref{prop:lln}} \label{sec:pf_prop_lln} 

We first introduce some notation. We write: 
$
\mathcal{I}_{\alpha_S} \equiv \{1, \ldots, d_{\alpha_S} \}, \  
\mathcal{I}_{\alpha_R} \equiv \{d_{\alpha_S} + 1, \ldots, d_{\alpha} \}, \   
\mathcal{I}_{\beta} \equiv \{d_{\alpha} + 1, \ldots, d_\theta \}
$ and: 
\begin{align} 
\label{eq:L_S}
\Lambda_{S, j-1}^{[i]} \equiv  		\begin{bmatrix} 
\Lambda_{SS, j-1}^{[i]} (\trueparam)  
&
\Lambda_{SR, j-1}^{[i]} (\trueparam)   
\end{bmatrix}   \in \mathbb{R}^{d_S \times d}, \qquad  
1 \le i \le N, \ 1 \le j \le n.  
\end{align}   
We also introduce a generic notation for classes of  random variables as follows: 
\begin{itemize}
\item We express a random variable as $\mathcal{M}_{n, N}^{(0)} (\theta)$ if it admits a representation:
$$ 
\sum_{j = 1}^n \sum_{i = 1}^N f (Z_{j-1}^{[i], \theta})
$$
for some $f \in \mathcal{R}$.  
\item We express a random variable as $\mathcal{M}_{n, N}^{(1)} (\theta)$ if it admits a representation:
$$ \sum_{j = 1}^n \sum_{i = 1}^N \sum_{\ell = 1}^d f_\ell (Z_{j-1}^{[i], \theta}) \mathbf{m}_{j-1}^{[i], \ell} (\trueparam) 
$$
for some $f_\ell \in \mathcal{R}$.    
\item We express a random variable as $\mathcal{M}_{n, N}^{(2)} (\theta)$ if it admits a representation:
$$ 
\sum_{j = 1}^n \sum_{i = 1}^N \sum_{\ell_1, \ell_2 = 1}^d 
f_{\ell_1, \ell_2}  (Z_{j-1}^{[i], \theta}) \mathbf{m}_{j-1}^{[i], \ell_1} (\trueparam) \mathbf{m}_{j-1}^{[i], \ell_2} (\trueparam) 
$$ 
for some $f_{\ell_1, \ell_2} \in \mathcal{R}$.     
\end{itemize} 
We study each of the following four cases:
\begin{align*}
\mathrm{I}&.\,(k_1, k_2) \in \mathcal{I}_{\alpha_S} \times \mathcal{I}_{\alpha_S}; \qquad 
\mathrm{II}.\, (k_1, k_2) \in \mathcal{I}_{\alpha_R} \times \mathcal{I}_{\alpha_R}; \\
\mathrm{III}&.\,(k_1, k_2) \in \mathcal{I}_{\beta} \times \mathcal{I}_{\beta};\qquad  
\mathrm{IV}.\,(k_1, k_2) \in \mathcal{I}_{w_1} \times \mathcal{I}_{w_2}, \, w_1, w_2 \in \{\alpha_S, \alpha_R, \beta\}, \, w_1 \neq w_2.  
\end{align*}

\noindent 
\textbf{Case I.} $(k_1, k_2) \in \mathcal{I}_{\alpha_S} \times \mathcal{I}_{\alpha_S}$. We have from (\ref{eq:contrast_decomp}) that: 
\begin{align}
&  \bigl[ \mathscr{H}_{n, N} (\theta) \bigr]_{k_1 k_2}
=  \tfrac{\Delta_n^2}{N} \partial_{\theta, k_1} \partial_{\theta, k_2} \ell_{n, N} (\theta) \nonumber \\ 
& \qquad  = \tfrac{\Delta_n}{N} \sum_{j = 1}^n \sum_{i = 1}^N 
\Bigl\{ 2 \, \partial_{\theta, k_1} V_{S, 0} (\alpha_S, \sample{j-1}{i})^\top  
\Lambda_{SS, j-1}^{[i]} (\theta)
\partial_{\theta, k_2} V_{S, 0} (\alpha_S, \sample{j-1}{i})  
\nonumber 	\\ 
& \qquad \qquad + 2 \, \partial_{\theta, k_2} V_{S, 0} (\alpha_S, \sample{j-1}{i})^\top  
\partial_{\theta, k_1} \Lambda_{SS, j-1}^{[i]} (\theta)
b_{S} (\alpha_S, \sample{j-1}{i})   \nonumber \\ 
& \qquad \qquad + 2 \, \partial_{\theta, k_1} V_{S, 0} (\alpha_S, \sample{j-1}{i})^\top  
\partial_{\theta, k_2} \Lambda_{SS, j-1}^{[i]} (\theta)
b_{S} (\alpha_S, \sample{j-1}{i})      
\nonumber \\ 
& \qquad \qquad  
+ b_S (\alpha_S, \sample{j-1}{i})^\top \partial_{\theta, k_1} \partial_{\theta, k_2} 
\Lambda_{SS, j-1}^{[i]} (\theta) 
b_S (\alpha_S, \sample{j-1}{i}) \Bigr\} \nonumber \\
\nonumber 
& \quad \quad \quad\quad\quad + \tfrac{{\Delta_n^2}}{N} \mathcal{M}_{n, N}^{(0)} (\theta)
+ h (\tfrac{\sqrt{\Delta_n^3}}{N}) \mathcal{M}_{n, N}^{(1)} (\theta) 
+ \tfrac{\Delta_n^2}{N} \mathcal{M}_{n, N}^{(2)} (\theta),  
\end{align} 
where $h$ is a non-decreasing function s.t. $| h (x) | \lesssim x, \, x > 0$. 
The last three terms in the above expression converge to $0$ in probability due to Lemmas \ref{lemma:base_conv} and  \ref{lemma:key_conv}. We then conclude from Lemmas \ref{lemma:base_conv} and  \ref{lemma:key_conv}, the consistency of $\hat{\theta}^{n, N}$ by noticing that the terms involving $b_S$ also converge to $0$ when $\theta = \hat{\theta}^{n, N}$.  \\ 

\noindent 
\textbf{Case II.}  $(k_1, k_2) \in \mathcal{I}_{\alpha_R} \times \mathcal{I}_{\alpha_R}$.  Noticing that $\Lambda (\theta)$, $\Sigma (\theta)$ are independent of $\alpha_R$,  we have: 
\begin{align*}
& \bigl[ \mathscr{H}_{n, N} (\theta) \bigr]_{k_1 k_2}
= 2 \tfrac{\Delta_n}{N} \sum_{j = 1}^n \sum_{i = 1}^N 
\partial_{\theta, k_1} \mathcal{B}_{j-1}^{[i]} (\theta)^\top  
\Lambda_{j-1}^{[i]} (\theta) 
\partial_{\theta, k_2} \mathcal{B}_{j-1}^{[i]} (\theta) \nonumber  \\
& 
+ 2 \tfrac{1}{N} \sum_{j = 1}^n \sum_{i = 1}^N  
\partial_{\theta, k_1} \partial_{\theta, k_2} \mathcal{B}_{j-1}^{[i]} (\theta)^\top  
\bigl( \Lambda_{S, j-1}^{[i]} (\theta) \bigr)^\top b_S (\alpha_S, \sample{j-1}{i}) 
\nonumber \\ 
& 
+ 2 \tfrac{{\Delta_n}}{N} \sum_{j = 1}^n \sum_{i = 1}^N  
\partial_{\theta, k_1} \partial_{\theta, k_2} \mathcal{B}_{j-1}^{[i]} (\theta)^\top  
\Lambda_{j-1}^{[i]} (\theta) \, 
\begin{bmatrix}
\tfrac{1}{2} \mathscr{L}_0 [V_{S, 0} (\alpha_S^\dagger, \cdot) ](Z_{j-1}^{[i], \trueparam}) 
- \tfrac{1}{2} \mathscr{L}_0 
[V_{S, 0} (\alpha_S, \cdot)] (Z_{j-1}^{[i], \theta})      \\[0.2cm] 
b_R (Z_{j-1}^{[i], \theta})
\end{bmatrix}^\top   
\\ &\qquad+ \tfrac{\sqrt{\Delta_n}}{N}\mathcal{M}_{n, N}^{(1)} (\theta),  
\nonumber 
\end{align*}  
where $\mathcal{B}^{[i]}_{j-1} (\theta)$ and $\Lambda^{[i]}_{S, j-1} (\theta)$ are defined in (\ref{eq:B}) and (\ref{eq:L_S}), respectively. Note that: 
\begin{align*}
\partial_{\theta, k_1} \mathcal{B}_{j-1}^{[i]} (\theta)^\top  
& \Lambda_{j-1}^{[i]} (\theta) 
\partial_{\theta, k_2} \mathcal{B}_{j-1}^{[i]} (\theta) \\ &= 
\partial_{\theta, k_1} V_{R, 0} (Z_{j-1}^{[i], \theta})^\top  
\Sigma_{RR}^{-1}  (\beta, \sample{j-1}{i}, \mu_{t_{j-1}}^N)
\partial_{\theta, k_2} V_{R, 0} (Z_{j-1}^{[i], \theta}),   
\end{align*} 
with the equality obtained from a computation similar to a corresponding one used in the proof of Lemma \ref{lemma:key_matrix}. We also notice that the second term is exactly zero since
\begin{align*}
& \partial_{\theta, k_1} \partial_{\theta, k_2} \mathcal{B}_{j-1}^{[i]} (\theta)^\top 
\bigl( \Lambda_{S, j-1}^{[i]}  (\theta) \bigr)^\top \\ 
& \quad 
= \bigl( \partial_{\theta, k_1} \partial_{\theta, k_2} V_{R, 0} ( Z_{j-1}^{[i], \theta}) \bigr)^\top 
\Sigma_{RR}^{-1} (Z_{j-1}^{[i], \theta})
[\Sigma_{RS} (Z_{j-1}^{[i], \theta})
\ \  
\Sigma_{RR} (Z_{j-1}^{[i], \theta}) 
] 
\bigl( \Lambda_{S, j-1}^{[i]}  (\theta) \bigr)^\top 
= \mathbf{0}_{d_R}.  
\end{align*} 
We then conclude from Lemmas \ref{lemma:base_conv}, \ref{lemma:key_conv}, and the consistency of the estimator.  \\

\noindent 
\textbf{Case III.}  $(k_1, k_2) \in \mathcal{I}_{\beta} \times \mathcal{I}_{\beta}$. We have that: 
\begin{align*}
\bigl[ \mathscr{H}_{n, N} (\theta) \bigr]_{k_1 k_2}
= \tfrac{\Delta_n}{N} \sum_{j = 1}^n \sum_{i = 1}^N 
\Big\{ \sum_{1 \le \kappa \le 3} &\mathscr{A}_{j-1}^{[i], (\kappa)} (\theta) \Big\}
+ h_1 \Bigl( \tfrac{\sqrt{\Delta_n^3}}{N} \Bigr) \times  \mathcal{M}^{(0)} (\theta)
\\ &+ h_2 \Bigl( \tfrac{\sqrt{\Delta_n}}{N}  \Bigr) \times \mathcal{M}^{(1)} (\theta),  \nonumber 
\end{align*}  
where $h_i : [0, \infty) \to \mathbb{R}$ such that $|h_i (x)| \lesssim |x|$, $i = 1, 2$ and: 
\begin{align*}
\mathscr{A}_{j-1}^{[i], (1)} (\theta) 
& =  \mathbf{m}_{j-1}^{[i]} (\trueparam)^\top 
\partial_{\theta, k_1} \partial_{\theta, k_2} \Lambda_{j-1}^{[i]} (\theta)
\mathbf{m}_{j-1}^{[i]} (\trueparam) 
+  \partial_{\theta, k_1} \partial_{\theta, k_2}  \log \det \Sigma_{j-1}^{[i]} (\theta); \\[0.3cm] 
\mathscr{A}_{j-1}^{[i], (2)} (\theta) 
& = \sum_{1 \le \ell_1, \ell_2 \le d} f_{\ell_1, \ell_2}^{k_1, k_2} (Z_{j-1}^{[i], \theta}) 
\bigl(  \mathbf{m}_{j-1}^{[i], \ell_1} (\theta)  -   \mathbf{m}_{j-1}^{[i], \ell_1} (\trueparam) \bigr)
\bigl(  \mathbf{m}_{j-1}^{[i], \ell_2} (\theta)  -   \mathbf{m}_{j-1}^{[i], \ell_2} (\trueparam) \bigr); \\[0.1cm] 
\mathscr{A}_{j-1}^{[i], (3)} (\theta) 
& = \sqrt{\Delta_n} \sum_{1 \le \ell  \le d} f_{\ell}^{k_1, k_2} (Z_{j-1}^{[i], \theta})  
\bigl(  \mathbf{m}_{j-1}^{[i], \ell} (\theta)  -   \mathbf{m}_{j-1}^{[i], \ell} (\trueparam) \bigr), 
\end{align*} 
for some $f_{\ell_1, \ell_2}^{k_1, k_2}, \, f_{\ell}^{k_1, k_2} \in \mathcal{R}$.  From Lemmas \ref{lemma:base_conv}, \ref{lemma:key_conv} we have that: 
\begin{align*}
& \tfrac{\Delta_n}{N} \sum_{j = 1}^n \sum_{i = 1}^N \mathscr{A}_{j-1}^{[i], (1)} (\theta)  \probconv  \\
& 
\int_0^T \mathbb{E}_{\mu_t} \Bigl[ 
\mathrm{tr}  \bigl\{\partial_{\theta, k_1} \partial_{\theta, k_2} \Lambda (\theta, W, \mu_t) \cdot \Sigma (\trueparam, W, \mu_t) \bigr\} +  
\partial_{\theta, k_1} \partial_{\theta, k_2} \log \det \Sigma   (\theta, W, \mu_t) \Bigr] dt; \\[0.4cm]
& h_1 \Bigl( \tfrac{\sqrt{\Delta_n^3}}{N} \Bigr) \times  \mathcal{M}^{(0)} (\theta) \probconv 0,   \qquad 
h_2 \Bigl( \tfrac{\sqrt{\Delta_n}}{N}  \Bigr) \times \mathcal{M}^{(1)} (\theta) \probconv 0,   
\end{align*}
uniformly in $\theta \in \Theta$. In particular, from the consistency of $\hat{\theta}^{n ,N}$ we get that for any $\lambda \in [0,1]$:
\begin{align*}
& \tfrac{\Delta_n}{N} \sum_{j = 1}^n \sum_{i = 1}^N \mathscr{A}_{j-1}^{[i], (1)} \bigl( \trueparam + \lambda (\hat{\theta}^{n, N} - \trueparam) \bigr) \\
& \qquad \probconv 
\int_0^T \mathbb{E}_{\mu_t} \Bigl[  
\bigl\{ 
\mathrm{tr}  \bigl( \partial_{\theta, k_1} \Sigma \cdot \Lambda \cdot  \partial_{\theta, k_2} \Sigma  \cdot \Lambda \bigr) 
\bigr\} (\trueparam, W, \mu_t) \Bigr] dt \equiv 2 \times  [\Gamma_\beta^{(\mathrm{H})} (\theta) ]_{k_1 k_2}, 
\end{align*}  
where we have used that for $Z^\dagger = (\trueparam, W, \mu_t)$, 
\begin{align*}
\partial_{\theta, k_1} \partial_{\theta, k_2} \Lambda (Z^\dagger)  & = 
2 \, \bigl( \Lambda \cdot \partial_{\theta, k_1} \Sigma  \cdot \Lambda 
\cdot \partial_{\theta, k_2} \Sigma  \cdot \Lambda \bigr) (Z^\dagger) 
-  
\bigl( \Lambda \cdot \partial_{\theta, k_1} \partial_{\theta, k_2} \Sigma \cdot \Lambda \bigr) (Z^\dagger); \\    
\partial_{\theta, k_1} \partial_{\theta, k_2} \log \det \Sigma  (Z^\dagger) 
& = - \bigl\{ \mathrm{tr} \bigl( \Lambda \cdot \partial_{\theta, k_1} \Sigma  \cdot \Lambda \cdot \partial_{\theta, k_2} \Sigma \bigr) \bigr\} (Z^\dagger) 
+ \mathrm{tr} \bigl( \Lambda \cdot \partial_{\theta, k_1} \partial_{\theta, k_2} \Sigma \bigr) (Z^\dagger). 
\end{align*} 
Following a similar argument in the proof of Lemma \ref{lemma:step-III} in Section \ref{sec:pf_consistency} together with the rate of convergence $\hat{\alpha}_S^{n ,N} - \alpha_S^\dagger = o_{\mathbb{P}_\trueparam} (\sqrt{\Delta_n})$ and Lemma \ref{lemma:base_conv}, we also get that for any $\lambda \in [0, 1]$, 
\begin{align*}
\tfrac{\Delta_n}{N} \sum_{j = 1}^n \sum_{i = 1}^N \bigl\{ \mathscr{A}_{j-1}^{[i], (\kappa)} \bigl( \trueparam + \lambda (\hat{\theta}^{n, N} - \trueparam) \bigr) \bigr\}  \probconv 0, \qquad  \kappa = 2, 3.  
\end{align*}
The proof for Case III is now complete. 
\\ 

\noindent 
\textbf{Case IV.} $(k_1, k_2) \in \mathcal{I}_{w_1} \times \mathcal{I}_{w_2}, \, w_1, w_2 \in \{\alpha_S, \alpha_R, \beta\}, \, w_1 \neq w_2$. Following similar computations as in the above three cases -- in particular the matrix operation in the proof of Lemma \ref{lemma:key_matrix} and the rate of convergence for $\hat{\alpha}_{S}^{n, N}$ -- we get that $[\mathscr{H}_{n, N} (\theta)]_{k_1 k_2} \probconv 0$. We then conclude  Proposition \ref{prop:lln}. 

\section{Proof of Proposition \ref{prop:clt}} \label{sec:pf_prop_clt}
%
In what follows, we often omit explicit use of 
$\trueparam$ to simplify the presentation, e.g.~we write $\mathbf{m}_{j-1}^{[i]} (\trueparam) \equiv \mathbf{m}_{j-1}^{[i]}, \ Z_{j-1}^{[i], \trueparam} \equiv Z_{j-1}^{[i]}$. We also write $\mathscr{L}_0 V_{S, 0} (Z_{j-1}^{[i]}) \equiv \mathscr{L}_0 [V_{S, 0}(\alpha_S^\dagger, \cdot)] (Z_{j-1}^{[i]})$. We then introduce, for $1 \le j \le n, \, 1 \le i \le N$:  
\begin{align}
\widetilde{\ell}_{j-1}^{\, [i]} (\theta) 
:= 
\bigl\{
\mathbf{m}_{j-1}^{[i]} (\theta)^\top \Lambda_{j-1}^{[i]} (\theta) \mathbf{m}_{j-1}^{[i]} (\theta)  
+ \log \det \Sigma_{j-1}^{[i]} (\theta) 
\bigr\}.  \label{eq:tilde_ell}
\end{align} 
Using (\ref{eq:tilde_ell}),  term $\mathscr{G}_{n, N} (\trueparam)$ defined in (\ref{eq:G}) is expressed as $\textstyle \mathscr{G}_{n, N} (\trueparam) = \sum_{i = 1}^N  \sum_{j = 1}^n  \widetilde{\mathscr{G}}_{j-1}^{\, [i]}$, where 
for $m_1 \in \mathcal{I}_{\alpha_S}$, $m_2 \in \mathcal{I}_{\alpha_R}$, $m_3 \in \mathcal{I}_{\beta}$: 
%
%
\begin{align}
\label{eq:g_alpha_S}
\widetilde{\mathscr{G}}_{j-1}^{\, [i], m_1} 
& = - 2 \sqrt{\tfrac{\Delta_n}{N}} 
\partial_{\theta, m_1} V_{S, 0} (\alpha_S^\dagger, \sample{j-1}{i})^\top 
\Lambda_{S, j-1}^{[i]}  \mathbf{m}_{j-1}^{[i]}  
\\ &\qquad -  \sqrt{\tfrac{\Delta_n^3}{N}}  
\partial_{\theta, m_1} \mathscr{L}_0  V_{S, 0} (Z_{j-1}^{[i]})^\top \Lambda_{S, j-1}^{[i]} 
\mathbf{m}_{j-1}^{[i]} \nonumber \\ 
& \qquad \nonumber 
+ \sqrt{\tfrac{\Delta_n^2}{N}} \times 
\bigl\{  
\bigl(\mathbf{m}_{j-1}^{[i]} \bigr)^\top 
\partial_{\theta, m_1} \Lambda_{j-1}^{[i]} 
\mathbf{m}_{j-1}^{[i]} 
+ \partial_{\theta, m_1} \log \det \Sigma_{j-1}^{[i]}
\bigr\};   
\\[0.2cm] 
\nonumber 
\widetilde{\mathscr{G}}_{j-1}^{\, [i], m_2}  
&= - 2 \sqrt{\tfrac{\Delta_n}{N}}  
\begin{bmatrix}
\tfrac{1}{2} \partial_{\theta, m_2} \mathscr{L}_0 V_{S, 0} (Z_{j-1}^{[i]})  \\[0.1cm] 
\partial_{\theta, m_2} V_{R, 0} (Z_{j-1}^{[i]}) \\ 
\end{bmatrix}^\top  
\Lambda_{j-1}^{[i]} 
\mathbf{m}_{j-1}^{[i]} ;  
\\[0.2cm]
\label{eq:g_beta}
\widetilde{\mathscr{G}}_{j-1}^{\, [i], m_3} 
& =  -  \sqrt{\tfrac{\Delta_n^2}{N}}  
\partial_{\theta, m_3} \mathscr{L}_0  V_{S, 0} (Z_{j-1}^{[i]})^\top 
\Lambda_{S, j-1}^{[i]} 
\mathbf{m}_{j-1}^{[i]}  \\
& \quad  \quad 
+ \sqrt{\tfrac{\Delta_n}{N}} \times 
\bigl\{  
\bigl(\mathbf{m}_{j-1}^{[i]}\bigr)^\top 
\partial_{\theta, m_3} \Lambda_{j-1}^{[i]}
\mathbf{m}_{j-1}^{[i]} 
+ \partial_{\theta, m_3} \log \det \Sigma_{j-1}^{[i]} 
\bigr\}.  
\nonumber 
\end{align} 
Following the arguments in \cite{amo:23, amo:24}, we show the CLT of $\mathscr{G}_{n, N} (\theta)$ via the \emph{martingale limit theorem}. In particular, we show the following limits. For $1 \le k_1, k_2 \le d_\theta$:
\begin{align}
\mathscr{J}_{n, N}^{k_1} & \equiv \sum_{j = 1}^n \sum_{i = 1}^N \mathbb{E}_{\trueparam} [ \widetilde{\mathscr{G}}_{j-1}^{\,[i], k_1} | \mathcal{F}_{t_{j-1}}^N] \probconv 0;   \label{eq:score}\\ 
\mathscr{J}_{n, N}^{k_1, k_2}  & \equiv \sum_{j = 1}^n
\sum_{i_1, i_2 = 1}^N  
\mathbb{E}_{\trueparam} [ 
\widetilde{\mathscr{G}}_{j-1}^{\, [i_1], k_1}  
\widetilde{\mathscr{G}}_{ j-1}^{\, [i_2], k_2} | \mathcal{F}_{t_{j-1}}^N] \probconv 4 \bigl[ \Gamma (\trueparam) \bigr]_{k_1 k_2};   \label{eq:second_m}\\ 
\widetilde{\mathscr{J}}_{n, N}^{\, k_1} & \equiv 
\sum_{j = 1}^n \sum_{i_1, i_2, i_3, i_4 = 1}^N 
\mathbb{E}_{\trueparam} \Bigl[ \prod_{\lambda = 1}^4
\widetilde{\mathscr{G}}_{j-1}^{\, [i_\lambda], k_1} | \mathcal{F}_{t_{j-1}}^N \Bigr] \probconv 0,   \label{eq:fourth_m} 
\end{align}
as $ n, N  \to \infty$ with $N \Delta_n \to 0$. 

\subsection{Proof of (\ref{eq:score})} \label{sec:pf_score}
\textbf{Case I.} \, $k_1 \in \mathcal{I}_{\alpha_S}. $  This case corresponds to the score w.r.t.~$\alpha_S$.  We have that $ \mathscr{J}_{n, N}^{k_1} =  \mathscr{J}_{n, N}^{k_1, (1)}   + 
\mathscr{J}_{n, N}^{k_1, (2)}$ with:  
\begin{align*}
\mathscr{J}_{n, N}^{k_1, (1)} 
& :=  
- 2 \sqrt{\tfrac{\Delta_n}{N}}  \sum_{i = 1}^N \sum_{j = 1}^n 
\mathbb{E} \left[   
\partial_{\theta, k_1} V_{S, 0} (\alpha_S^\dagger, \sample{j-1}{i})^\top 
\Lambda_{S, j-1}^{[i]} \,  
\mathbf{m}_{j-1}^{[i]} 
\,\big|\, \mathcal{F}_{t_{j-1}}^N \right]   \\ 
& \qquad  \qquad
-  \sqrt{\tfrac{\Delta_n^3}{N}}  \sum_{i = 1}^N \sum_{j = 1}^n 
\mathbb{E} \left[   
\partial_{\theta, k_1} \mathscr{L}_0  V_{S, 0} (Z_{j-1}^{[i]})^\top 
\Lambda_{S, j-1}^{[i]}   \, 
\mathbf{m}_{j-1}^{[i]} 
\,\big|\, \mathcal{F}_{t_{j-1}}^N \right]; \\[0.2cm]  
\mathscr{J}_{n, N}^{k_1, (2)} 
& :=  
\sqrt{\tfrac{\Delta_n^2}{N}}  \sum_{i = 1}^N \sum_{j = 1}^n 
\mathbb{E} \left[   
\bigl( \mathbf{m}_{j-1}^{[i]} \bigr)^\top 
\bigl( \partial_{\theta, k_1} \Lambda_{j-1}^{[i]} \bigr) 
\mathbf{m}_{j-1}^{[i]} 
+ \partial_{\theta, k_1} \log \det \Sigma_{j-1}^{[i]} 
\,\big|\, \mathcal{F}_{t_{j-1}}^N \right].  
\end{align*}
From Lemma \ref{lemma:bds_aux} and (\ref{eq:first_moment}) in Lemma \ref{lemma:moments} we obtain  that 
$ \textstyle 
\mathbb{E}_{\trueparam} \bigl| \mathscr{J}_{n, N}^{k_1, (1)}  \bigr| 
\le  C \sqrt{\tfrac{\Delta_n}{N}}  \sum_{i = 1}^N \sum_{j = 1}^n \Delta_n^{3/2} = C \times T \sqrt{N \Delta_n^2}, 
$ 
for some constant $C> 0$ depending on $T > 0$ but not on $n, N$. We also have from Lemma \ref{lemma:bds_aux} and (\ref{eq:second_moment}) in Lemma \ref{lemma:moments} that:  
\begin{align*}
\mathscr{J}_{n, N}^{k_1, (2)} 
& = \sqrt{\tfrac{\Delta_n^2}{N}}  \sum_{i = 1}^N \sum_{j = 1}^n 
\bigl\{ 
\mathrm{tr} \bigl[ \bigl( \partial_{\theta, k_1} \Lambda_{j-1}^{[i]} \bigr)   \Sigma_{j-1}^{[i]} \bigr]
+ \partial_{\theta, k_1} \log \det \Sigma_{j-1}^{[i]} 
\bigr\} 
\\ &\qquad + \sqrt{\tfrac{\Delta_n^2}{N}}  \sum_{i = 1}^N \sum_{j = 1}^n e_{j-1}^{[i, i]} (\Delta_n) \\[0.2cm] 
& =  \sqrt{\tfrac{\Delta_n^2}{N}}  \sum_{i = 1}^N \sum_{j = 1}^n e_{j-1}^{[i, i]} (\Delta_n), 
\end{align*}   
where $e_{j-1}^{[i, i]} \in \mathcal{S}_{t_{j-1}}$, and in the last line we have used the fact that  
$ 
\partial_{\theta, k_1} \log \det \Sigma_{j-1}^{[i]} 
= \mathrm{tr} \bigl[  \Lambda_{j-1}^{[i]}  \partial_{\theta, k_1} \Sigma_{j-1}^{[i]} \bigr]  
= -  \mathrm{tr} \bigl[  \partial_{\theta, k_1}  \Lambda_{j-1}^{[i]} \Sigma_{j-1}^{[i]} \bigr].    
$ 
Thus, there exists a constant $C = C(T) > 0$ such that 
$
\mathbb{E}_{\trueparam} \bigl|  \mathscr{J}_{n, N}^{k_1, (2)}  \bigr|   
\le C \sqrt{N \Delta_n^2}. 
$ We conclude {Case I}.
\\ 

\noindent 
\textbf{Case II.}  $k_1 \in \mathcal{I}_{\alpha_R}. $
This case corresponds to the score w.r.t.~$\alpha_R$. We have that: 
\begin{align*}
\mathscr{J}_{n, N}^{k_1} 
\equiv  
- 2 \sqrt{\tfrac{\Delta_n}{N}}  \sum_{i = 1}^N \sum_{j = 1}^n 
\mathbb{E}_{\trueparam} \left[   
\begin{bmatrix}
\tfrac{1}{2} \partial_{\theta, k_1} \mathscr{L}_0 V_{S, 0} (Z_{j-1}^{[i]})  \\[0.1cm] 
\partial_{\theta, k_1} V_{R, 0} (Z_{j-1}^{[i]}) \\ 
\end{bmatrix}^\top  
\Lambda_{j-1}^{[i]} \, 
\mathbf{m}_{j-1}^{[i]} 
\, \Bigl|  \, \mathcal{F}_{t_{j-1}}^N \right].  
\end{align*}  
An argument similar to the one used in 
Case I.~yields that 
$
\mathbb{E}_{\trueparam} \bigl| \mathscr{J}_{n, N}^{k_1} \bigr| 
\le  C  \times \sqrt{N \Delta_n^2}, \, d_{\alpha_S} + 1 \le k_1 \le d_\alpha, 
$
for some constant $C =  C (T) > 0$. We conclude {Case II}. \\ 

\noindent 
\textbf{Case III. } $k_1 \in \mathcal{I}_{\beta}. $ This case corresponds to the score w.r.t.~$\beta$.  We have that $ \mathscr{J}_{n, N}^{k_1}  =  \mathscr{J}_{n, N}^{k_1, (1)} + 
\mathscr{J}_{n, N}^{k_1, (2)}$ with:  
\begin{align*}
\mathscr{J}_{n, N}^{k_1, (1)} 
& :=  -  \sqrt{\tfrac{\Delta_n^2}{N}}  \sum_{i = 1}^N \sum_{j = 1}^n 
\mathbb{E}_{\trueparam} \left[   
\partial_{\theta, k_1} 
\mathscr{L}_0 V_{S, 0} (Z_{j-1}^{[i]})^\top 
\Lambda_{S, j-1}^{[i]}  \, 
\mathbf{m}_{j-1}^{[i]} 
\,\big|\, \mathcal{F}_{t_{j-1}}^N \right]; \\  
\mathscr{J}_{n, N}^{k_1, (2)} 
& := 
\sqrt{\tfrac{\Delta_n}{N}}  \sum_{i = 1}^N \sum_{j = 1}^n 
\mathbb{E}_{\trueparam} \left[   
\bigl( \mathbf{m}_{j-1}^{[i]} \bigr)^\top 
\partial_{\theta, k_1} \Lambda_{j-1}^{[i]}
\mathbf{m}_{j-1}^{[i]} 
+ \partial_{\theta, k_1} \log \det \Sigma_{j-1}^{[i]} 
\,\big|\, \mathcal{F}_{t_{j-1}}^N \right].  
\end{align*} 
A analysis similar to the one for {Case I.}~yields the bounds
$ \mathbb{E}_{\trueparam} \bigl|  \mathscr{J}_{n, N}^{k_1, (1)}  \bigr|  
\le  C_ 1 \sqrt{N \Delta_n^3}$,   
$\mathbb{E}_{\trueparam} \bigl|  \mathscr{J}_{n, N}^{k_1, (2)}  \bigr|  
\le  C_2 \sqrt{N \Delta_n}
$
for some constants $C_1 = C_1 (T) > 0$ and $C_2 = C_2 (T) > 0$.  The proof of (\ref{eq:score}) is now complete.  
\subsection{Proof of (\ref{eq:second_m})} 
We analyse the following cases separately: 
\begin{align*}
\mathrm{I}&. (k_1, k_2) \in \mathcal{I}_{\alpha_R} \times \mathcal{I}_{\alpha_R}; \quad  
\mathrm{II}. (k_1, k_2) \in \mathcal{I}_{\beta} \times \mathcal{I}_{\beta};  \\
\mathrm{III}&. (k_1, k_2) \in \mathcal{I}_{\alpha_S} \times \mathcal{I}_{\alpha_S}; \quad   
\mathrm{IV}. (k_1, k_2) \in \mathcal{I}_{w_1} \times \mathcal{I}_{w_2}, \quad  w_1, w_2 \in \{\alpha_S, \alpha_R, \beta\}, \, w_1 \neq w_2. 
\end{align*}
\noindent
\textbf{Case I.}  $(k_1, k_2) \in \mathcal{I}_{\alpha_R} \times \mathcal{I}_{\alpha_R}$.  We introduce for $k \in \mathcal{I}_{\alpha_R}$,  $1 \le i \le N$ and $1 \le j \le n$:  
\begin{align} \label{eq:B_R}
\mathcal{B}_{j-1}^{[i], k}= 
\begin{bmatrix}
\tfrac{1}{2} \partial_{x_R}^\top V_{S, 0}  (\alpha_S^\dagger, X_{t_{j-1}}^{[i]})  \, 
\partial_{\theta, k} V_{R, 0} (\alpha_R^\dagger, \sample{j-1}{i}, \mu_{t_{j-1}}^N ) \\[0.2cm]
\partial_{\theta, k} V_{R, 0} (\alpha_R^\dagger, \sample{j-1}{i}, \mu_{t_{j-1}}^N ) 
\end{bmatrix}.   
\end{align}  
We have that: 
\begin{align*}
\mathscr{J}_{n, N}^{k_1,  k_2}  
& = 4 \tfrac{\Delta_n}{N}  \sum_{j = 1}^n \sum_{i_1, i_2 = 1}^N 
\mathbb{E}_{\trueparam} \bigl[ \prod_{q = 1, 2}
\bigl\{ 
(\mathcal{B}_{j-1}^{[i_q], k_q})^\top 
\Lambda_{j-1}^{[i_q]} \mathbf{m}_{j-1}^{[i_q]} \bigr\} 
\,\big|\,  \mathcal{F}_{t_{j-1}}^N
\bigr] \\ 
& = 4 \tfrac{\Delta_n}{N}  \sum_{j = 1}^n \sum_{i = 1}^N 
(\mathcal{B}_{j-1}^{[i], k_1})^\top  \Lambda_{j-1}^{[i]} \,  \mathcal{B}_{j-1}^{[i], k_2}  
+ \tfrac{\Delta_n}{N} \sum_{j = 1}^n \sum_{i_1, i_2 = 1}^N e_{j-1}^{[i_1, i_2]} (\Delta_n) \qquad  \\ 
& = 4 \tfrac{\Delta_n}{N} \sum_{j = 1}^n \sum_{i = 1}^N \partial_{\theta, k_1} V_{R, 0} (Z_{j-1}^{[i]})^\top \Sigma_{RR}^{-1} (Z_{j-1}^{[i]})  
\partial_{\theta, k_2} V_{R, 0} (Z_{j-1}^{[i]}) 
\\ &\qquad \qquad + \tfrac{\Delta_n}{N} \sum_{j = 1}^n \sum_{i_1, i_2 = 1}^N {e}_{j-1}^{[i_1, i_2]} (\Delta_n),   
\end{align*} 
for some $e_{j-1}^{[i_1, i_2]} \in \mathcal{S}_{t_{j-1}}$, where in the second line, we have used equation \eqref{eq:second_moment} from Lemma \ref{lemma:moments} and in the third line, we have used the equality:
\begin{align*}
\Lambda_{j-1}^{[i]} \mathcal{B}_{j-1}^{[i], k_2} = 
\begin{bmatrix}
\mathbf{0}_{d_S}^\top  &  
\Bigl( \Sigma_{RR}^{-1} (Z_{j-1}^{[i]}) \partial_{\theta, k_2} V_{R, 0} (Z_{j-1}^{[i]}) \Bigr)^\top 
\end{bmatrix}^\top   
\end{align*} 
which follows from a computation similar to the ones  provided in the proof of Lemma \ref{lemma:key_matrix}.
We thus conclude:
\begin{align*}
\mathscr{J}_{n, N}^{k_1,  k_2}   \probconv 4 \int_0^T \mathbb{E}_{\mu_t} \bigl[
\bigl( \partial_{\theta, k_1} V_{R , 0}^\top \, 
\Sigma_{RR}^{-1} \, 
\partial_{\theta, k_2} V_{R , 0} \bigr) (\trueparam, W, \mu_t)  \bigr]  dt 
\end{align*} 
as $n, N \to \infty$ with $N \Delta_n \to 0$. 
\\ 

\noindent 
\textbf{Case II.} $(k_1, k_2) \in \mathcal{I}_{\beta} \times \mathcal{I}_{\beta}$. We have from Lemma \ref{lemma:moments} that  $\textstyle \mathscr{J}_{n, N}^{k_1,  k_2} = \sum_{\iota = 1, 2, 3} G_{k_1,  k_2}^{(\iota)}$  with: 
\begin{align*}
G_{k_1,  k_2}^{(1)}  
& = \tfrac{\Delta_n}{N} \sum_{j = 1}^{n} \sum_{i_1, i_2 = 1}^N \widetilde{G}_{j-1}^{[i_1, i_2]};  \nonumber \\ 
G_{k_1,  k_2}^{(2)} 
& = \tfrac{\Delta_n}{N} \sum_{j = 1}^n \sum_{i_1, i_2 = 1}^N  e_{j-1}^{(1), [i_1, i_2]} (\Delta_n)
+ \tfrac{\Delta_n^{3/2}}{N} \sum_{j = 1}^n \sum_{i_1, i_2 = 1}^N  e_{j-1}^{(2), [i_1, i_2]} (\Delta_n^{3/2});  \\ 
G_{k_1,  k_2}^{(3)} 
& = \tfrac{\Delta_n^{3/2}}{N} \sum_{j = 1}^n \sum_{i_1, i_2 = 1}^N  \sum_{\alpha \in \{1, \ldots, d\}^3 } \! \! \! \! e_{j-1}^{[i_1, i_2], \alpha} (1) \cdot \mathbb{E}_{\trueparam} \bigl[ \mathbf{m}_{j-1}^{[i_1], \alpha_1} \mathbf{m}_{j-1}^{[i_2], \alpha_2} \mathbf{m}_{j-1}^{[i_2], \alpha_3}    | \mathcal{F}_{t_{j-1}}^N \bigr] 
\end{align*}  
for $e_{j-1}^{(1), [i_1, i_2]}, e_{j-1}^{(2), [i_1, i_2]}, e_{j-1}^{[i_1, i_2], \alpha} \in \mathcal{S}_{t_{j-1}}$, where we have set: 
\begin{align*}
\widetilde{G}_{j-1}^{[i_1, i_2]} 
:= \mathbb{E}_{\trueparam} \Bigl[  
\prod_{q = 1, 2} 
\bigl\{ 
(\mathbf{m}_{j-1}^{[i_q]})^\top \partial_{\theta, k_q} \Lambda_{j-1}^{[i_q]} \mathbf{m}_{j-1}^{[i_q]}  
+ \partial_{\theta, k_q} \log \det \Sigma_{j-1}^{[i_q]}
\bigr\} 
\,\big|\, \mathcal{F}_{t_{j-1}}^N\Bigr]. 
\end{align*}
We immediately see that $G_{k_1, k_2}^{(2)}$ converges to $0$ as $N \Delta_n \to 0$ in $L_1$. We then study the first term. We apply (\ref{eq:second_moment}) and (\ref{eq:fourth_moment}) in Lemma \ref{lemma:moments} to get that:
\begin{align*}
& \widetilde{G}_{j-1}^{[i_1, i_2]} 
= \mathrm{tr} \bigl(   \partial_{\theta, k_1} \Lambda_{j-1}^{[i_1]} \, \Sigma_{j-1}^{[i_1]} \bigr) \mathrm{tr} \bigl(   \partial_{\theta, k_2} \Lambda_{j-1}^{[i_2]} \, \Sigma_{j-1}^{[i_2]} \bigr)  \nonumber \\
& \quad 
+ 
\sum_{1 \le \ell_1, \ell_2, \ell_3, \ell_4 \le d}
\bigl[  \partial_{\theta, k_1} \Lambda_{j-1}^{[i_1]}  \bigr]_{\ell_1 \ell_2}
\bigl[ \partial_{\theta, k_1} \Lambda_{j-1}^{[i_2]}  \bigr]_{\ell_3 \ell_4} 
\Big\{  
\bigl[ \Sigma_{j-1}^{[i_1]}  \bigr]_{\ell_1 \ell_3}  \bigl[ \Sigma_{j-1}^{[i_2]}  \bigr]_{\ell_2 \ell_4} 
\\ &\qquad\qquad\qquad\qquad\qquad\qquad\qquad\qquad\qquad\qquad\quad\quad + \bigl[ \Sigma_{j-1}^{[i_1]}  \bigr]_{\ell_1 \ell_4}  \bigl[ \Sigma_{j-1}^{[i_2]}  \bigr]_{\ell_2 \ell_3}   \Big\} 
\times \mathbf{1}_{i_1 = i_2} \\ 
& \quad 
+ \partial_{\theta, k_2} \log  \det \Sigma_{j-1}^{[i_2]}  \times 
\mathrm{tr} \bigl( 
\partial_{\theta, k_1} \Lambda_{j-1}^{[i_1]} \Sigma_{j-1}^{[i_1]} \bigr) 
+ \partial_{\theta, k_1} \log  \det \Sigma_{j-1}^{[i_1]}  \times 
\mathrm{tr} \bigl( 
\partial_{\theta, k_2} \Lambda_{j-1}^{[i_2]} \Sigma_{j-1}^{[i_2]} \bigr)  \nonumber \\   
& \quad + \partial_{\theta, k_1} \log  \det \Sigma_{j-1}^{[i_1]}  \times  
\partial_{\theta, k_2} \log  \det \Sigma_{j-1}^{[i_2]} 
+ e_{j-1}^{[i_1, i_2]} (\Delta_n), 
\end{align*}
where $e_{j-1}^{[i_1, i_2]} \in \mathcal{S}_{t_{j-1}}$.  Noticing that 
$ 
\partial_{\theta, k_q} \log \det \Sigma_{j-1}^{[i_q]} = - \mathrm{tr} \bigl( \partial_{\theta, k_q} \Lambda_{j-1}^{[i_q]}  \, \Sigma_{j-1}^{[i_q]} \bigr)$, $q = 1, 2
$, 
we obtain: 
\begin{align*}
{G}_{k_1, k_2}^{(1)}
& = 2 \tfrac{\Delta_n}{N} \sum_{j  = 1}^n \sum_{i = 1}^N 
\mathrm{tr} \bigl( \partial_{\theta, k_1} \Lambda_{j- 1}^{[i]} \,  \Sigma_{j-1}^{[i]} 
\,  \partial_{\theta, k_2} \Lambda_{j- 1}^{[i]} \,  \Sigma_{j-1}^{[i]} 
\, \bigr) + \tfrac{\Delta_n}{N} \sum_{j = 1}^n \sum_{i_1, i_2 =1}^N  e_{j-1}^{[i_1, i_2]} (\Delta_n) \nonumber  \\
& = \nonumber  2 \tfrac{\Delta_n}{N} \sum_{j  = 1}^n \sum_{i = 1}^N 
\mathrm{tr} \bigl( \partial_{\theta, k_1} \Sigma_{j- 1}^{[i]} \,  \Lambda_{j-1}^{[i]} 
\,  \partial_{\theta, k_2} \Sigma_{j- 1}^{[i]} \,  \Lambda_{j-1}^{[i]} 
\, \bigr)  + \tfrac{\Delta_n}{N} \sum_{j = 1}^n \sum_{i_1, i_2 =1}^N  e_{j-1}^{[i_1, i_2]} (\Delta_n),
\end{align*}
where in the last line we have used $\partial_{\theta, k_q} \Lambda_{j- 1}^{[i]} = - \Lambda_{j-1}^{[i]} \,  \partial_{\theta, k_q} \Sigma_{j-1}^{[i]}  \, \Lambda_{j-1}^{[i]}$, $q = 1, 2$, and:
$$
\mathrm{tr} \bigl( \Lambda_{j-1}^{[i]} 
\, \partial_{\theta, k_1} \Sigma_{j- 1}^{[i]} \,  \Lambda_{j-1}^{[i]} 
\,  \partial_{\theta, k_2} \Sigma_{j- 1}^{[i]} \,  \bigr)   = \mathrm{tr} \bigl( \partial_{\theta, k_1} \Sigma_{j- 1}^{[i]} \,  \Lambda_{j-1}^{[i]} 
\,  \partial_{\theta, k_2} \Sigma_{j- 1}^{[i]} \,  \Lambda_{j-1}^{[i]} 
\, \bigr).  
$$
We thus obtain from Lemma \ref{lemma:base_conv} that: as $n, N \to \infty$ with $N \Delta_n \to 0$, 
\begin{align*}
G_{k_1, k_2}^{(1)} \probconv 2 \int_0^T 
\mathbb{E}_{\mu_t} \bigl[ \bigl\{ \mathrm{tr} \bigl( 
\partial_{\theta, k_1} \Sigma 
\, \Lambda \, 
\partial_{\theta, k_2} \Sigma  
\, \Lambda \bigr) \bigr\} 
(\trueparam, W, \mu_t) \bigr] dt.  
\end{align*} 
We now analyze the third term $G_{k_1, k_2}^{(3)}$. Using the decomposition of $\mathbf{m}_{j-1}^{[i]}$ given in (\ref{eq:decomp_m}), we get that: for $\alpha \in \{1, \ldots, d\}^3$, 
\begin{align}
\begin{aligned}
& \mathbb{E}_{\trueparam} \bigl[ \mathbf{m}_{j-1}^{[i_1], \alpha_1} \mathbf{m}_{j-1}^{[i_2], \alpha_2} \mathbf{m}_{j-1}^{[i_2], \alpha_3}| \mathcal{F}_{t_{j-1}}^N \bigr] \\ 
& \quad 
= \mathbb{E}_{\trueparam} \bigl[ C_{j-1}^{[i_1], \alpha_1} C_{j-1}^{[i_2], \alpha_2} C_{j-1}^{[i_2], \alpha_3}| \mathcal{F}_{t_{j-1}}^N \bigr]  
+ \mathbb{E}_{\trueparam} \bigl[ F_{j-1}^{[i_1, i_2], \alpha} | \mathcal{F}_{t_{j-1}}^N \bigr], 
\end{aligned} \label{eq:third_order_m}
\end{align}
where $F_{j-1}^{[i_1, i_2], \alpha}$ is a product of random variables specified via (\ref{eq:decomp_m}) and involves at least one $D_{j-1}^{[i]}$. Due to the existence of the term $D_{j-1}^{[i]}$, we have that $\mathbb{E}_{\trueparam} \bigl[ F_{j-1}^{[i_1, i_2], \alpha} | \mathcal{F}_{t_{j-1}}^N \bigr] = e_{j-1} (\Delta_n^{1/2})$ for $e_{j-1} \in \mathcal{S}_{t_{j-1}}$. We also observe that the first term on the right-hand side of (\ref{eq:third_order_m}) becomes $0$ when taking the expectation of an odd number of products of Gaussian variables. We then have that:
\begin{align} \label{eq:third_order_m_size}
\mathbb{E}_{\trueparam} \bigl[ \mathbf{m}_{j-1}^{[i_1], \alpha_1} \mathbf{m}_{j-1}^{[i_2], \alpha_2} \mathbf{m}_{j-1}^{[i_2], \alpha_3}| \mathcal{F}_{t_{j-1}}^N \bigr] 
= e_{j-1} (\Delta_n^{1/2}), \quad e_{j-1} \in \mathcal{S}_{t_{j-1}},
\end{align} 
and that:  $G_{k_1, k_2}^{(3)} \probconv 0$ as $N \Delta_n \to 0$.  We thus conclude Case II.  
%
\\ 

\noindent 
\textbf{Case III.} $(k_1, k_2) \in \mathcal{I}_{\alpha_S} \times \mathcal{I}_{\alpha_S}$ . Making use of Lemma \ref{lemma:moments} and the estimate (\ref{eq:third_order_m_size}), we get $\textstyle \mathscr{J}_{n, N}^{k_1,  k_2} = \sum_{\iota = 1, 2} H_{k_1, k_2}^{(\iota)}$  with:    
\begin{align}
H_{k_1, k_2}^{(1)}  
& =  4 \tfrac{\Delta_n}{N} \sum_{j = 1}^n \sum_{i_1, i_2 = 1}^N  
\mathbb{E}_{\trueparam} 
\Bigl[ \prod_{q = 1, 2} 
\bigl\{ 
\partial_{\theta, k_q} V_{S, 0} (\alpha_S^\dagger, \sample{j-1}{i_q})^\top  \Lambda_{S, j-1}^{[i_q]}  \mathbf{m}_{j-1}^{[i_q]} 
\bigr\} 
\,\big|\,   \mathcal{F}_{t_{j-1}}^N 
\Bigr];  \nonumber  \\
H_{k_1, k_2}^{(2)}  
& =   \tfrac{1}{N} \sum_{j = 1}^n \sum_{i_1, i_2 = 1}^N e_{j-1}^{[i_1, i_2], k_1, k_2} (\Delta_n^2), \quad e_{j-1}^{[i_1, i_2], k_1, k_2} \in \mathcal{S}_{t_{j-1}}. \nonumber   
\end{align}
%
We have from Lemma \ref{lemma:moments} that: 
\begin{align*} 
H_{k_1, k_2}^{(1)} & = 4 \tfrac{\Delta_n}{N} \sum_{j = 1}^n \sum_{i = 1}^N \partial_{\theta, k_1} 
V_{S, 0} (\alpha_S^\dagger, \sample{j-1}{i})^\top 
\Lambda_{S, j-1}^{[i]} 
\Sigma_{j-1}^{[i]}
\bigl( \Lambda_{S, j-1}^{[i]} \bigr)^\top  
\partial_{\theta, k_2}  V_{S, 0} (\alpha_S^\dagger, \sample{j-1}{i})  \\ 
&\quad  + \tfrac{\Delta_n}{N} \sum_{j = 1}^n \sum_{i_1, i_2  = 1}^N e_{j-1}^{[i_1, i_2]} (\Delta_n)   \\
& \nonumber
=  4 \tfrac{\Delta_n}{N} \sum_{j = 1}^n \sum_{i = 1}^N 
\partial_{\theta, k_1} 
V_{S, 0} (\alpha_S^\dagger, \sample{j-1}{i})^\top 
\Lambda_{SS, j-1}^{[i]} 
\partial_{\theta, k_2}  V_{S, 0} (\alpha_S^\dagger, \sample{j-1}{i}) \\ & 
\quad + \tfrac{\Delta_n}{N} \sum_{j = 1}^n \sum_{i_1, i_2 = 1}^N e_{j-1}^{[i_1, i_2]} (\Delta_n),  
\end{align*} 
where $e_{j-1}^{[i_1, i_2]} \in \mathcal{S}_{t_{j-1}}$. 
Using Lemma \ref{lemma:base_conv} and $\Lambda_{SS} = 4 \Sigma_{SS}^{-1}$, we conclude that as $n , N \to \infty$, $N \Delta_n \to 0$: 
\begin{align*}
\mathscr{J}_{n, N}^{k_1, k_2} 
\probconv 
16  \int_0^T \mathbb{E}_{\mu_t} 
\bigl[ 
\bigl( (\partial_{\theta, k_1} V_{S, 0})^\top \, 
\Sigma_{SS}^{-1} \, 
\partial_{\theta, k_2}  V_{S, 0} \bigr) 
(\trueparam, W, \mu_t) \bigr] dt. 
\end{align*} 
The proof of Case III is now complete. \\ 

\noindent 
\textbf{Case IV.}  $(k_1, k_2) \in \mathcal{I}_{w_1} \times \mathcal{I}_{w_2}, \, w_1, w_2 \in \{\alpha_S, \alpha_R, \beta\}, \, w_1 \neq w_2$.  We will show that
$ \mathscr{J}_{n, N}^{k_1, k_2} \probconv  0$ as $n, N \to \infty, \, N \Delta_n \to 0.$ We focus only on the case $(k_1, k_2) \in \mathcal{I}_{\alpha_S} \times \mathcal{I}_{\alpha_R}$ as the other scenarios can be proved via a similar argument. Using  expression (\ref{eq:B_R}) and the estimate (\ref{eq:third_order_m_size}) again, we have that $\textstyle \mathscr{J}_{n, N}^{k_1, k_2} = \sum_{\iota = 1, 2}  I_{k_1, k_2}^{(\iota)}$ with: 
\begin{align*}
I_{k_1, k_2}^{(1)}
= 4 \tfrac{\Delta_n}{N} \sum_{j = 1}^n \sum_{i_1, i_2 = 1}^N \widetilde{I}_{j-1}^{[i_1, i_2]}, \qquad 
I_{k_1, k_2}^{(2)}
= \tfrac{1}{N} \sum_{j = 1}^n \sum_{i_1, i_2 = 1}^N e_{j-1}^{[i_1, i_2]} (\Delta_n^2), \quad 
e_{j-1}^{[i_1, i_2]} \in \mathcal{S}_{t_{j-1}}. 
\end{align*} 
where we have set: 
\begin{align*}
\widetilde{I}_{j-1}^{\, [i_1, i_2]} 
:= 
\mathbb{E}_{\trueparam} 
\bigl[   
\bigl( \partial_{\theta, k_1} V_{S, 0} (\alpha_S^\dagger, \sample{j-1}{i_1} )^\top 
\Lambda_{S, j-1}^{[i_1]}  \mathbf{m}_{j-1}^{[i_1]} \bigr) 
\bigl( \Lambda_{j-1}^{[i_2]} \mathcal{B}_{j-1}^{[i_2], k_2} 
\, \mathbf{m}_{j-1}^{[i_2]} \bigr) 
\big|\, \mathcal{F}_{t_{j-1}}^N 
\bigr].  
\end{align*}
It follows from Lemma \ref{lemma:moments} that: 
\begin{align*}
\widetilde{I}_{j-1}^{\, [i_1, i_2]} 
& = \partial_{\theta, k_1} V_{S, 0} (\alpha_S^\dagger, \sample{j-1}{i_1}) 
\Lambda_{S, j-1}^{[i_1]} 
\underbrace{\Sigma_{j-1}^{[i_1]} \, \Lambda_{j-1}^{[i_1]} }_{= I_d}   \mathcal{B}_{j-1}^{[i_1], k_2} \times \mathbf{1}_{i_1 = i_2} 
 + {e}_{j-1}^{[i_1, i_2]} (\Delta_n) \\
& = {e}_{j-1}^{[i_1, i_2]} (\Delta_n),  
\end{align*} 
where $e_{j-1}^{[i_1, i_2]} \in \mathcal{S}_{t_{j-1}}$ and in the last line we used $\Lambda_{S, j-1}^{[i_1]} \mathcal{B}_{j-1}^{[i_1], k_2} = \mathbf{0}_{d_S}$, which follows from a similar calculation to the one in the proof of Lemma \ref{lemma:key_matrix}. We thus conclude $I_{k_1, k_2}^{(\iota)}  \probconv 0, \, \iota = 1, 2$ as $n, N \to \infty$ with $N \Delta_n \to 0$.  The proof of (\ref{eq:second_m}) is now complete. 

\subsection{Proof of (\ref{eq:fourth_m})} 
\textbf{Case I.} $k_1 \in \mathcal{I}_{\alpha_S}$. We write (\ref{eq:g_alpha_S}) as 
$\textstyle \widetilde{\mathscr{G}}_{j-1}^{[i], k_1} = \sqrt{\tfrac{\Delta_n}{N}} \Xi_{j-1}^{[i], k_1, (1)} + \sqrt{\tfrac{\Delta_n^3}{N}} \Xi_{j-1}^{[i], k_1, (2)} + \sqrt{\tfrac{\Delta_n^2}{N}} \Xi_{j-1}^{[i], k_1, (3)}$.  Noticing that $\Xi_{j-1}^{[i], k_1, (q)} = e_{j-1} (1)$ for some $e_{j-1} \in \mathcal{S}_{t_{j-1}}$, for all $i \in \{1, \ldots, N\}$, $j \in \{1, \ldots, d\}$ and $q = 1, 2, 3$,  we have that 
$\widetilde{\mathscr{J}}_{n, N}^{\, k_1} =\widetilde{\mathscr{J}}_{n, N}^{\, k_1, (1)} +  \widetilde{\mathscr{J}}_{n, N}^{\, k_1, (2)}
+ \widetilde{\mathscr{J}}_{n, N}^{\, k_1, (3)}$ with: 
\begin{align*}
\widetilde{\mathscr{J}}_{n, N}^{\, k_1, (1)}
& = \tfrac{\Delta_n^2}{N^2} \sum_{j = 1}^{n} \sum_{i_1, i_2, i_3, i_4 = 1}^N \mathbb{E}_{\trueparam} \Bigl[ \prod_{q = 1}^4 \Xi_{j-1}^{[i_q], k_1, (1)}  | \mathcal{F}_{t_{j-1}}^N \Bigr];  \\
\widetilde{\mathscr{J}}_{n, N}^{\, k_1, (2)} 
&=  4 \tfrac{\Delta_n^{{5}/{2}}}{N^2} \sum_{j = 1}^{n} \sum_{i_1, i_2, i_3, i_4 = 1}^N 
\mathbb{E}_{\trueparam} \Bigl[  \Xi_{j-1}^{[i_1], k_1, (1)} 
\Xi_{j-1}^{[i_2], k_1, (1)} 
\Xi_{j-1}^{[i_3], k_1, (1)} 
\Xi_{j-1}^{[i_4], k_1, (3)}  | \mathcal{F}_{t_{j-1}}^N \Bigr]; \\ 
\widetilde{\mathscr{J}}_{n, N}^{\, k_1, (3)} 
&  = \tfrac{1}{N^2} \sum_{j = 1}^{n} \sum_{i_1, i_2, i_3, i_4 = 1}^N e_{j-1}^{[i_1, i_2, i_3, i_4]} (\Delta_n^{3}), 
\end{align*} 
for some $e_{j-1}^{[i_1, i_2, i_3, i_4]} \in \mathcal{S}_{t_{j-1}}$. We immediately have that $\widetilde{\mathscr{J}}_{n, N}^{\, k_1, (3)}$ converges to $0$ as $N \Delta_n \to 0$ under the $L_1$-norm. For the first term, we apply (\ref{eq:fourth_moment}) in Lemma \ref{lemma:moments} to get that: 
\begin{align*}
\widetilde{\mathscr{J}}_{n, N}^{\, k_1, (1)} 
& = \tfrac{1}{N^2} \sum_{j = 1}^{n} \sum_{i_1, i_2 = 1}^N e_{j-1}^{[i_1, i_2]} (\Delta_n^2)
+ \tfrac{1}{N^2} \sum_{j = 1}^{n} \sum_{i_1, i_2, i_3, i_4 = 1}^N e_{j-1}^{[i_1, i_2, i_3, i_4]} (\Delta_n^3),
\end{align*} 
for some $e_{j-1}^{[i_1, i_2]}, e_{j-1}^{[i_1, i_2, i_3, i_4]} \in \mathcal{S}_{t_{j-1}}$. Thus,$\widetilde{\mathscr{J}}_{n, N}^{\, k_1, (1)}$ converges to $0$ under $L_1$ as $n, N \to \infty$ and $N \Delta_n \to 0$. We next study the second term $\widetilde{\mathscr{J}}_{n, N}^{\, k_1, (2)}$. First, we note that this term is expressed as:
\begin{align} 
\label{eq:J_2_rep}
&\widetilde{\mathscr{J}}_{n, N}^{\, k_1, (2)} 
= \tfrac{\Delta_n^{5/2}}{N^2}\sum_{j = 1}^n \sum_{\gamma \in \{1, \ldots, N\}^4} \\  & \ \  \biggl\{
\sum_{ \lambda \in \{1, \ldots, d\}^5}
e^{(1), \gamma, \lambda}_{j-1} (1) \times   \mathbb{E}_{\trueparam} \bigl[ 
\mathbf{m}_{j-1}^{[\gamma_1], \lambda_1}
\mathbf{m}_{j-1}^{[\gamma_2], \lambda_2} 
\mathbf{m}_{j-1}^{[\gamma_3], \lambda_3} 
\mathbf{m}_{j-1}^{[\gamma_4], \lambda_4} \mathbf{m}_{j-1}^{[\gamma_4], \lambda_5}
\,\big|\, \mathcal{F}_{t_{j-1}}^N \bigr] \nonumber  \\ 
& \qquad\qquad + \sum_{ \lambda \in \{1, \ldots, d\}^3} e^{(2), \gamma, \lambda}_{j-1}
(1) \times \mathbb{E}_{\trueparam} 
\bigl[ 
\mathbf{m}_{j-1}^{[\gamma_1], \lambda_1}
\mathbf{m}_{j-1}^{[\gamma_2], \lambda_2} 
\mathbf{m}_{j-1}^{[\gamma_3], \lambda_3} 
\,\big|\, \mathcal{F}_{t_{j-1}}^N \bigr] 
\biggr\}
\nonumber 
\end{align} 
for $e^{(\iota), \gamma, \lambda}_{j-1} \in \mathcal{S}_{t_{j-1}}, \, \iota = 1, 2$. Similarly to the analysis in obtaining (\ref{eq:third_order_m_size}), using the decomposition $\mathbf{m}_{j-1} = C_{j-1} + D_{j-1}$ specified in (\ref{eq:decomp_m}) and noticing that products of odd numbers of Gaussian variates become $0$ under expectation, we have that the conditional expectations included in (\ref{eq:J_2_rep}) can be  treated as $e_{j-1}^{\gamma, \lambda} (\Delta_n^{1/2})$ for $e_{j-1}^{\gamma, \lambda} \in \mathcal{S}_{t_{j-1}}$. Thus, (\ref{eq:J_2_rep}) converges to $0$ in $L_1$-norm as $N \Delta_n \to 0$. Proof of \textbf{Case I.} is now complete. 
\\ 

\noindent 
\textbf{Case II.} $k_1 \in \mathcal{I}_{\alpha_R}$. This case follows the same argument used in \textbf{Case I.}, and we omit the details to avoid repetition. 
\\

\noindent 
\textbf{Case III.} $k_1 \in \mathcal{I}_{\beta}$. We write (\ref{eq:g_beta}) as $\textstyle \widetilde{\mathscr{G}}_{j-1}^{[i], k_1} = \sqrt{\tfrac{\Delta_n}{N}} \Xi_{j-1}^{[i], k_1, (1)} + \sqrt{\tfrac{\Delta_n^2}{N}} \Xi_{j-1}^{[i], k_1, (2)}$ to get that $\widetilde{\mathscr{J}}_{n, N}^{k_1} = \widetilde{\mathscr{J}}_{n, N}^{k_1, (1)} + \widetilde{\mathscr{J}}_{n, N}^{k_1, (2)} + \widetilde{\mathscr{J}}_{n, N}^{k_1, (3)}$ with: 
\begin{align*}
\widetilde{\mathscr{J}}_{n, N}^{k_1, (1)} 
& = \tfrac{\Delta_n^2}{N^2} \sum_{j = 1}^{n} \sum_{i_1, i_2, i_3, i_4 = 1}^N \mathbb{E}_{\trueparam} \Bigl[ \prod_{q = 1}^4 \Xi_{j-1}^{[i_q], k_1, (1)}  \,\big|\, \mathcal{F}_{t_{j-1}}^N \Bigr];  \\
\widetilde{\mathscr{J}}_{n, N}^{\, k_1, (2)} 
&=  4 \tfrac{\Delta_n^{{5}/{2}}}{N^2} \sum_{j = 1}^{n} \sum_{i_1, i_2, i_3, i_4 = 1}^N 
\mathbb{E}_{\trueparam} \Bigl[  \Xi_{j-1}^{[i_1], k_1, (1)} 
\Xi_{j-1}^{[i_2], k_1, (1)} 
\Xi_{j-1}^{[i_3], k_1, (1)} 
\Xi_{j-1}^{[i_4], k_1, (2)}  \,\big|\, \mathcal{F}_{t_{j-1}}^N \Bigr]; \\ 
\widetilde{\mathscr{J}}_{n, N}^{\, k_1, (3)} 
&  = \tfrac{1}{N^2} \sum_{j = 1}^{n} \sum_{i_1, i_2, i_3, i_4 = 1}^N {e}_{j-1}^{[i_1, i_2, i_3, i_4]} (\Delta_n^{3}),  
\end{align*} 
for some $e_{j-1}^{[i_1, i_2, i_3, i_4]} \in \mathcal{S}_{t_{j-1}}$. We immediately get that $\widetilde{\mathscr{J}}_{n, N}^{k_1, (3)} \to 0$ under $L_1$ as $N \Delta_n \to 0$. The second term is shown to converge to $0$ under $L_1$ due to a similar analysis employed for (\ref{eq:J_2_rep}) in \textbf{Case I.} as each term in the summation in $\widetilde{\mathscr{J}}_{n, N}^{\, k_1, (2)}$ involves products of odd numbers of variable  $\mathbf{m}_{j-1}^{[i], k}$, which results in $\mathbb{E}_{\trueparam} \bigl[  \Xi_{j-1}^{[i_1], k_1, (1)} 
\Xi_{j-1}^{[i_2], k_1, (1)} 
\Xi_{j-1}^{[i_3], k_1, (1)} 
\Xi_{j-1}^{[i_4], k_1, (2)}  | \mathcal{F}_{t_{j-1}}^N \bigr] = e_{j-1} (\Delta_n)$ for some $e_{j-1} \in \mathcal{S}_{t_{j-1}}$. For the term $\widetilde{\mathscr{J}}_{n, N}^{k_1, (1)}$, we make use of the following result, whose proof is postponed to the end of this section. 
\begin{lemma} \label{lemma:final_lemma}
It holds that for any $\gamma = (\gamma_1, \gamma_2, \gamma_3, \gamma_4) \in \{1, \ldots, N\}^4$, $k \in  \mathcal{I}_{\beta}$ and $1 \le j \le n$, 
\begin{align}
\begin{aligned}  
\label{eq:prod_four_beta}
& \mathbb{E}_{\trueparam} 
\Bigl[ \prod_{q = 1}^4 \bigl\{ (\mathbf{m}_{j-1}^{[\gamma_q]})^\top  \partial_{\theta, k} \Lambda^{[\gamma_q]}_{j-1} \mathbf{m}_{j-1}^{[\gamma_q]} 
+  \partial_{\theta, k} \log \det \Sigma_{j-1}^{[\gamma_q]} \bigr\}  \,\big|\, \mathcal{F}_{t_{j-1}}^N \Bigr] \\ 
& \qquad 
= e_{j-1}^{\gamma} (1) \times\mathbf{1}_{
(\gamma_1, \gamma_2) = (\gamma_3, \gamma_4) \, \mathrm{or}\,  (\gamma_1, \gamma_3) = (\gamma_2, \gamma_4)\,
\mathrm{or}\,  (\gamma_1, \gamma_2) = (\gamma_4, \gamma_3)} 
+ f_{j-1}^{\gamma} (\Delta_n),  
\end{aligned} 
\end{align} 
for $e_{j-1}^\gamma,  \, f_{j-1}^\gamma \in \mathcal{S}_{t_{j-1}}$. 
\end{lemma}
\noindent Due to Lemma \ref{lemma:final_lemma}, we immediately obtain that $\widetilde{\mathscr{J}}_{n, N}^{k_1, (1)} \to 0$ under $L_1$ as $N \Delta_n \to  0$. We thus conclude proof of \textbf{Case III.}, and now the proof of Proposition \ref{prop:clt} is complete. \\ 

\noindent 
\textit{Proof of Lemma \ref{lemma:final_lemma}.} 
Using the decomposition $\mathbf{m}_{j-1}^{[i]} = C_{j-1}^{[i]} + D_{j-1}^{[i]}$ specified in (\ref{eq:decomp_m}) and noticing that $\mathbb{E}_{\trueparam} |D_{j-1}^{[i]}|^p = \mathcal{O} (\Delta_n^{1/2}), \, p \ge 1$, we have that (LHS of (\ref{eq:prod_four_beta})) $= \textstyle \sum_{\iota = 1, 2} E^{\gamma, (\iota)}_{j-1} + g_{j-1}^{\gamma} (\Delta_n)$ for some $g_{j-1}^{\gamma} \in \mathcal{S}_{t_{j-1}}$, where: 
\begin{align}
E^{\gamma, (1)}_{j-1} & = 
\mathbb{E}_{\trueparam} 
\Bigl[ \prod_{q = 1}^4 
\bigl\{ (C_{j-1}^{[\gamma_q]})^\top \bigl( \partial_{\theta, k} \Lambda_{j-1}^{[\gamma_q]} \bigr) \, C_{j-1}^{[\gamma_q]} 
+  \partial_{\theta, k} \log \det \Sigma_{j-1}^{[\gamma_q]} \bigr\}  \,\big|\, \mathcal{F}_{t_{j-1}}^N \Bigr];  \nonumber \\ 
E^{\gamma, (2)}_{j-1} 
& =  2 \sum_{q_1 = 1}^4
\mathbb{E}_{\trueparam} \Bigl[  
\prod_{q_2 \in \{1, 2, 3, 4\} \setminus \{q_1\}}  \Bigl\{ (C_{j-1}^{[\gamma_{q_2}]})^\top \bigl( \partial_{\theta, k} \Lambda_{j-1}^{[\gamma_{q_2}]} \bigr) \, C_{j-1}^{[\gamma_{q_2}]} 
+  \partial_{\theta, k} \log \det \Sigma_{j-1}^{[\gamma_{q_2}]} \Bigr\} \nonumber \\  
& \qquad \qquad \qquad \qquad \times C_{j-1}^{[\gamma_{q_1}]} 
\bigl( \partial_{\theta, k} \Lambda_{j-1}^{[\gamma_{q_1}]} \bigr) \, D_{j-1}^{[\gamma_{q_1}]}  \,  \big|\, \mathcal{F}_{t_{j-1}}^N \Bigr].  \nonumber  
\end{align} 
For the first term, we note that for any $1 \le q \le 4$, $k \in \mathcal{I}_{\beta}$ and $1 \le j \le n$:
\begin{align*} 
\mathbb{E}_{\trueparam} 
&\bigl[
\bigl\{ (C_{j-1}^{[\gamma_q]})^\top \bigl( \partial_{\theta, k} \Lambda_{j-1}^{[\gamma_q]} \bigr) \, C_{j-1}^{[\gamma_q]} 
+  \partial_{\theta, k} \log \det \Sigma_{j-1}^{[\gamma_q]} \bigr\}  \,\big|\, \mathcal{F}_{t_{j-1}}^N \bigr] 
\\ &\qquad \qquad\qquad= \mathrm{tr} \bigl[ \bigl( \partial_{\theta, k} \Lambda_{j-1}^{[\gamma_q]} \bigr) \, \Sigma_{j-1}^{[\gamma_q]} \bigr]  +  \partial_{\theta, k} \log \det \Sigma_{j-1}^{[\gamma_q]} 
= 0,  \nonumber 
\end{align*}  
as we have seen in Section \ref{sec:pf_score}, {Case I}. Also, due to the independence of Gaussian variables in $C_{j-1}^{[\gamma_q]}$ between different particles, we observe that the term $E_{j-1}^{\gamma, (1)}$ takes non-trivial value evaluated as $\mathcal{O}(1)$ under $L_1$-norm only when $(\gamma_1, \gamma_3) = ( \gamma_2, \gamma_4)$ or $(\gamma_1, \gamma_2) = ( \gamma_3, \gamma_4)$ or $(\gamma_1, \gamma_2) = ( \gamma_4, \gamma_2)$, otherwise $0$. We then study the second term. As we have seen in proof of (\ref{eq:fourth_moment}) in Appendix \ref{sec:aux} in the main text, we have that  $D_{j-1}^{[\gamma_{q_1}]} = \widetilde{D}_{j-1}^{[\gamma_{q_1}]} + e_{j-1} (\Delta_n)$ for some $e_{j-1} \in \mathcal{S}_{t_{j-1}}$, where $\widetilde{D}_{j-1}^{[\gamma_{q_1}]}$ is defined in (\ref{eq:tilde_D}).  Since it follows from moments of stochastic integrals that: 
\begin{align*}
\mathbb{E}_{\trueparam} \Bigl[  
\prod_{q_2 \in \{1, 2, 3, 4\} \setminus \{q_1\}}  \Bigl\{ (C_{j-1}^{[\gamma_{q_2}]})^\top \bigl( \partial_{\theta, k} \Lambda_{j-1}^{[\gamma_{q_2}]} \bigr) \, &C_{j-1}^{[\gamma_{q_2}]} 
+  \partial_{\theta, k} \log \det \Sigma_{j-1}^{[\gamma_{q_2}]} \Bigr\} \\ &\times  C_{j-1}^{[\gamma_{q_1}]} 
\bigl( \partial_{\theta, k} \Lambda_{j-1}^{[\gamma_{q_1}]} \bigr) \, \widetilde{D}_{j-1}^{[\gamma_{q_1}]}\,\Big|\,\mathcal{F}_{t_{j-1}}^N \Bigr] = 0
\end{align*} 
for any $1 \le q_1 \le 4$, we get that $E^{\gamma, (2)}_{j-1} = f_{j-1}^\gamma (\Delta_n)$ for some $f_{j-1}^\gamma \in \mathcal{S}_{t_{j-1}}$. We thus conclude.  

\bibliographystyle{imsart-nameyear} 
\bibliography{ips_hypo_aap}       %

\end{document}